\documentclass[reqno,draft]{amsart}

\usepackage{amssymb} 
\usepackage{xypic}

\newtheorem{theorem}{Theorem}[section]

\newtheorem*{theoremA}{Theorem A}
\newtheorem*{theoremB}{Theorem B}
\newtheorem*{theoremC}{Theorem C}
\newtheorem*{theoremD}{Theorem D}
\newtheorem*{theoremE}{Theorem E}
\newtheorem{lemma}[theorem]{Lemma}
\newtheorem{prop}[theorem]{Proposition}
\newtheorem{cor}[theorem]{Corollary}

\theoremstyle{definition}
\newtheorem{definition}[theorem]{Definition}

\theoremstyle{remark}

\newtheorem*{remarks}{Remarks}
\newtheorem*{remark}{Remark}
\newtheorem*{example}{Example}
\newtheorem*{examples}{Examples}
\newtheorem*{question}{Question}

\newcommand{\bbold}{\mathbb}
\newcommand{\abs}[1]{\lvert#1\rvert}

\renewcommand\hat\widehat

\newcommand{\supp}{\operatorname{supp}}

\newcommand{\R}{{\bbold R}}
\newcommand{\Q}{{\bbold Q}}
\newcommand{\Z}{{\bbold Z}}
\newcommand{\N}{{\bbold N}}
\newcommand{\C}{{\bbold C}}
\newcommand{\F}{{\bbold F}}

\newcommand{\ZpX}{\Z_p\langle X\rangle}

\newcommand{\QpX}{\Q_p\langle X\rangle}

\renewcommand{\k} {{{\boldsymbol{k}}}}
\newcommand{\kX} {\k\langle X\rangle}

\newcommand{\kXprime} {\k\langle X'\rangle}

\newcommand{\RX} {R\langle X\rangle}
\newcommand{\RXprime} {R\langle X'\rangle}

\newcommand{\fm}{\mathfrak m}

\newcommand{\Frac}{\operatorname{Frac}}

\renewcommand{\geq}{\geqslant}
\renewcommand{\leq}{\leqslant}
\renewcommand{\bar}{\overline}

\newcommand{\res}{\operatorname{res}}

\DeclareFontFamily{OMS}{smallo}{}
\DeclareFontShape{OMS}{smallo}{m}{n}{<->s*[.65]cmsy10}{}
\DeclareSymbolFont{smallo@m}{OMS}{smallo}{m}{n}
\DeclareMathSymbol{\smallo}{\mathord}{smallo@m}{79}

\DeclareMathAlphabet{\mathbf}{OML}{cmm}{b}{it}

\numberwithin{equation}{section}

\author[Aschenbrenner]{Matthias Aschenbrenner}
\address{Kurt G\"odel Research Center for Mathematical Logic\\
Universit\"at Wien\\
1090 Wien\\ Austria}
\email{matthias.aschenbrenner@univie.ac.at}

\author[Srhir]{Ahmed Srhir}
\address{
D\'epartement de Math\'ematique \\
Facult\'e des Sciences \\
Universit\'e Ibn Tofail\\
K\'enitra\\ Morocco}
\email{ahmedsrhir@hotmail.com}

\begin{document}

\title[Analytic Nullstellens\"atze]{Analytic Nullstellens\"atze \\ and the Model Theory of Valued Fields}
\date{March 2024}

\begin{abstract}
We present a uniform framework for establishing Nullstellens\"atze for power series rings using quantifier elimination
results for valued   fields. As an application we obtain   Nullstellens\"atze for $p$-adic   power series (both formal and convergent) analogous to
R\"uckert's complex and Risler's real Nullstellensatz, as well as a $p$-adic analytic version of Hilbert's 17th Problem.
Analogous statements for restricted power series, both real and $p$-adic, are also considered.
\end{abstract}

\maketitle

\section*{Introduction}

\noindent
R\"uckert's Nullstellensatz~\cite{Rueckert} 
establishes a one-to-one correspondence between
radical ideals in rings of germs of complex analytic functions  and the zero sets of these ideals.
This fundamental theorem plays   a role  in complex analytic geometry which is similar to that of  Hilbert's Nullstellensatz in classical
algebraic geometry.
A  counterpart   of Hilbert's Nullstellensatz for   polynomials over real closed fields (instead of algebraically closed fields) is due to Ris\-ler~\cite{Risler70}, who
also established
a version of R\"uckert's theorem   for germs of real analytic functions~\cite{Risler72} as well as an analogue of
Hilbert's 17th Problem in this setting.
An adaptation of Risler's theorem for formal power series over real closed fields was shown by Merrien~\cite{Merrien72, Merrien73}
 (with simplifications in \cite{Lassalle}) and used for applications to germs of real $C^\infty$-functions.

It was noted early in the history of model theory   that there is a close relationship between
Nullstellensatz-type statements and A.~Robinson's concept of model completeness: for example, Hilbert's Nullstellensatz is essentially
equivalent to the model completeness of the theory of algebraically closed fields, whereas Risler's Nullstellensatz
for polynomials  corresponds to the model completeness of the theory of real closed fields.
(This connection has been formalized in \cite{Weispfenning77}.)

A.~Robinson~\cite{Robinson} established R\"uckert's  Nullstellensatz using non-standard methods, and asked for a   model-theoretic proof. This was provided by Weispfenning~\cite{Weispfenning75}, who also gave a proof of Risler's theorem for real analytic function germs using similar methods.
His arguments, which were   translated into somewhat more algebraic language by Robbin~\cite{RobbinI, RobbinII},
 combine the ubiquitous Weierstrass Preparation Theorem  
 (which permits an induction on the number of indeterminates) with  model completeness   of algebraically closed fields
 respectively  real closed fields.
 
In this paper we establish a uniform method for obtaining Nullstellens\"atze in power series rings  by 
 relying on quantifier elimination theorems for certain theories of {\it valued}\/ fields instead of model completeness   for {\it fields}\/ in order to prove the requisite specialization theorems.
 The advantage of this approach is that it naturally suggests how to also prove Nullstellens\"atze and   versions of Hilbert's 17th Problem for $p$-adic power series, which seem to be new. (See \cite[Conjectures~4.3, 4.4]{FS09}.)
 
\subsection*{Nullstellens\"atze for power series over the $p$-adics}
To formulate these, consider the rational function 
$$\gamma_p(x) := \frac{1}{p} \left( \wp(x)- \frac{1}{\wp(x)} \right)^{-1}\in\Q(x)$$
where $\wp(x):=x^p-x$ (the Artin-Schreier operator). 
One calls $\gamma=\gamma_p$ the {\em $p$-adic Kochen operator.}\/
It plays a role in   $p$-adic algebra similar to that of the squaring operator $x^2$ in the real world.
To substantiate this, we remark that $\gamma$ is {\it $p$-integral definite,} that is, $\abs{\gamma(a)}_p\leq 1$ for all
$a\in\Q_p$.  (Here and below $\abs{\,\cdot\,}_p$ denotes the $p$-adic absolute value on~$\Q_p$.)
As a consequence, setting~$F:=\Q_p(X)$ 
where~$X=(X_1,\dots,X_m)$ is a tuple of distinct indeterminates,
every rational function of the form $\gamma(f)$ where~$f\in F$ is $p$-integral definite, and hence  so is every element of 
$$\Lambda_F := \left\{ \frac{f}{1-pg}: f,g\in\Z[\gamma(F)]\right\}.$$
Kochen~\cite{Kochen}    proved that conversely, the subring
$\Lambda_F$ of $F$ contains {\it every}\/ $p$-integral definite rational function in $F$;
this may be seen as a $p$-adic analogue of Hilbert's 17th Problem. Kochen  also established a $p$-adic
Nullstellensatz for   ideals in~$\Q_p[X]$, similar to Risler's Nullstellensatz for ideals in $\R[X]$, the statement of which also involves the ring $\Lambda_F$. Our first main theorem is an analogous result
for ideals in the ring~$\Q_p\{X\}$  of convergent power series over $\Q_p$ in the indeterminates~$X$.
In the following, we let $A$ be a subring of~$\Q_p[[X]]$ with fraction field~$F=\Frac(A)$.
Then~$1\neq pg$ for each
$g\in\Z[\gamma(F)]$, and so we may define the subring $\Lambda_F$ of $F$ as above.
 We also let 
 $\Lambda_F A$  denote the subring of $F$ generated by $\Lambda_F$ and $A$.
 With these notations, here is our $p$-adic analytic Nullstellensatz:

\begin{theoremA}
Suppose $A=\Q_p\{X\}$, and let~$f,g_1,\dots,g_n\in A$. If
$$g_1(a)=\cdots=g_n(a)=0 \ \Longrightarrow\  f(a)=0\quad\text{for each $a\in \Q_p^m$ sufficiently close to $0$,}$$ 
then there are   $h_1,\dots,h_n\in \Lambda_F A$   and $k\geq 1$  such that $f^k=g_1h_1+\cdots+g_nh_n$.
\end{theoremA}

\noindent
The converse of the implication in this theorem holds trivially.
Next, an analytic version of Kochen's $p$-adic Hilbert's 17th Problem:

\begin{theoremB}
Suppose $A=\Q_p\{X\}$, and let~$f,g\in A$. Then
$\abs{f(a)}_p \leq \abs{g(a)}_p$ for all~${a\in \Q_p^m}$ sufficiently close to $0$ if and only if  $f\in g\Lambda_F$.
\end{theoremB}

\noindent
We also have a $p$-adic version of Merrien's theorem for formal power series, where we now let $t$ be a single indeterminate over $\Q_p$:
 
\begin{theoremC}
Suppose $A=\Q_p[[X]]$, and let 
$f,g_1,\dots,g_n\in A$ be such that 
$$g_1(a)=\cdots=g_n(a)=0 \ \Longrightarrow\  f(a)=0\quad\text{for each $a\in \Q_p[[t]]^m$ with $a(0)=0$.}$$ 
Then  there are   $h_1,\dots,h_n\in \Lambda_F A$   and $k\geq 1$  such that $f^k=g_1h_1+\cdots+g_nh_n$.
\end{theoremC}

\noindent
In our arguments we systematically work in the framework of {\it Weierstrass systems}\/ from~\cite{DL}, which axiomatize the algebraic properties of the ring $A$ needed for the proofs to go through
(crucially among them, the Weierstrass Division Property). As a consequence, for example, Theorem~C also holds for the subring~$A=\Q[[X]]$ of~$\Q_p[[X]]$, 
the ring~$A=\Q_p[[X]]^{\operatorname{a}}$ of power series in $\Q_p[[X]]$ which are algebraic over~$\Q_p(X)$,
or the ring~$A= \Q_p[[X]]^{\operatorname{da}}$ of differentially algebraic power series.
(These facts can also be deduced from the Artin Approximation Theorem as proved in \cite{DL}, but this would be overkill.)

\subsection*{Restricted analytic variants}
Both Theorems~A and B above deal with {\it germs}\/ of $p$-adic analytic functions at~$0$.
The subring $\QpX$ of $\Q_p\{X\}$ consisting of those power series which converge for all $a\in\Q_p^m$ with $\abs{a}_p\leq 1$
also satisfies a Weierstrass Division Theorem (albeit of a slightly different kind than   in  $\Q_p\{X\}$). Our  method is adaptable to this setting, and so we   also obtain
versions of Theorems~A and~B for the ring~$A=\QpX$ of {\it restricted}\/ (sometimes called {\it strictly convergent}\/) power series  over $\Q_p$,
where the condition~``for each~$a\in\Q_p^m$ sufficiently close to $0$'' in both cases is replaced by ``for each~$a\in\Q_p^m$ with~${\abs{a}_p\leq 1}$''.
These theorems also hold for certain subrings of~$\QpX$ such as the ring~$\QpX^{\operatorname{a}}:=\Q_p[[X]]^{\operatorname{a}}\cap\QpX$.
(See Corollaries~\ref{cor:p-adic H17 rest} and \ref{cor:p-adic NS rest} for the precise statements.)

The same methods also yield a restricted analytic Nullstellensatz
and Hil\-bert's 17th Problem for real closed fields $\k$ equipped with a complete ultrametric absolute value~$\abs{\,\cdot\,}$, such as the Levi-Civita
field: the completion of the   group algebra~$\R[t^{\Q}]$ equipped with the
ultrametric absolute value satisfying~$\abs{at^q}=\operatorname{e}^{-q}$ for~$a\in\R^\times$,~${q\in\Q}$. (Cf.~\cite[Chapter~1, \S{}7]{LR}.)
Let $\kX$ denote the subring of~$\k[[X]]$ consisting 
of those  power series which converge for all $a\in\k^m$ with $\abs{a}\leq 1$.

\begin{theoremD}
Let $f,g_1,\dots,g_n\in \kX$. Then
\begin{enumerate}
\item  $g_1(a)=\cdots=g_n(a)=0 \Rightarrow f(a)=0$ for all $a\in \k^m$ with $\abs{a}\leq 1$  if and only if there are $k\geq 1$ and~$b_1,\dots,b_l,h_1,\dots,h_n\in \kX$   such that~$f^{2k}+b_1^2+\cdots+b_l^2=g_1h_1+\cdots+g_nh_n$; and
\item $f(a)\geq 0$ for all $a\in\k^m$ with $\abs{a}\leq 1$ if and only if
there are~$g,h_1,\dots,h_k\in \kX\setminus\{0\}$  such that
$g^2f  =  h_1^2+\cdots+h_k^2$.
\end{enumerate}
\textup{(}Similarly with $\kX^{\operatorname{a}}:=\k[[X]]^{\operatorname{a}}\cap\kX$ in place of $\kX$.\textup{)}
\end{theoremD}

\noindent
A version of part (2) of this theorem for $\k[X]$ in place of $\kX$ was shown by Dickmann~\cite[Theorem~2]{Dickmann}.
For our final  theorem (an analogue of Theorem~B for $\kX$) we let $H_K$ be the subring of~$K:=\Frac(\kX)$ generated by the elements~$1/(1+h)$ where~$h$ is sum of squares in~$K$; this   ``real'' analogue of the $p$-adic Kochen ring~$\Lambda_F$ from above was  introduced by Sch\"ulting~\cite{Schuelting}. We also let $R:=\{a\in\k:\abs{a}\leq 1\}$ be the valuation ring of $\k$,
with maximal ideal~$\fm:=\{a\in\k:\abs{a}<1\}$,
and we let~$\RX$ be the subring of $\kX$ consisting of all~$f\in\kX$ with 
$\abs{f(a)}\leq 1$ for all~$a\in R^m$.
The ring~$H_K$ is a Pr\"ufer domain; hence so is the subring~$H_K\RX$ of $K$ generated by~$H_K$ and $\RX$. (In particular, $H_K\RX$ is integrally closed.)

\begin{theoremE}
Let $f,g\in\kX$. Then 
\begin{align*}
\text{$\abs{f(a)}\leq\abs{g(a)}$ for all $a\in R^m$}&\quad\Longleftrightarrow\quad f\in g\,H_K\RX, \\
\text{$\abs{f(a)}<\abs{g(a)}$ for all $a\in R^m$}&\quad\Longleftrightarrow\quad
f\in g\,\fm H_K\RX.
\end{align*}
\end{theoremE}

\noindent
The theorems above permit   various further potential generalizations which we have not attempted here.
For example, let  $E$ be a finite extension of $\Q_p$. Jarden-Ro\-quette~\cite{JR} proved a Nullstellensatz for ideals in $E[X]$
and established the $p$-adic version of Hilbert's 17th Problem for rational functions in $E(X)$.
(Here, the definition of the $p$-adic Kochen operator  needs to be suitably modified.)
It should be routine to adapt Theorems~A, B, C to this setting. Less routine would be an adaptation of our
theorems to germs of analytic functions on basic $p$-adic  subsets of analytic varieties, 
analogous to~\cite[Theorems~1.2, 7.2]{JR}. (See also \cite[\S{}7]{PR}.) It would also be interesting to generalize Theorems~D and~E
in a similar way. (For polynomials this was studied in \cite{HY,Lavi}.)
It would also be interesting to use Theorem~C in a similar way as Merrien used his Nullstellensatz for formal
power series over the reals in order to obtain a $p$-adic analogue of his Nullstellensatz for germs of $C^\infty$-functions;
see \cite[Conjecture~4.3]{FS09}. (Here one should employ the concept of {\it strict differentiability}\/ for multivariate $p$-adic functions
developed in~\cite{BGN,Gloeckner,deSmedt}, in order to avoid the well-known disadvantages of the usual notion of differentiability 
over the $p$-adics caused by the failure  of the Mean Value Theorem.)





\subsection*{Organization of the paper}
After a first preliminary section, in Section~\ref{sec:W-systems} we recall basic facts about Weierstrass systems
and establish the setting of infinitesimal Weierstrass structures, which provide a convenient framework
to state theorems like the ones discussed earlier in this introduction.
As a warm-up,   since these cases are conceptually easier than the $p$-adic version, in Sections~\ref{sec:Rueckert} and~\ref{sec:Risler} we then revisit  R\"uckert's and Risler's Nullstellens\"atze using our method. Here, the relevant model-theoretic results are quantifier elimination theorems for
the theory $\operatorname{ACVF}$ of algebraically closed non-trivially valued fields (A.~Robinson) and for 
the theory $\operatorname{RCVF}$ of real closed non-trivially convexly valued fields (Cherlin-Dickmann), respectively.
The model-theoretic ingredient for the proofs of Theorems~A,~B,~C is an analogous elimination theorem
for the theory~$p\operatorname{CVF}$ of {\it $p$-adically closed non-trivially $p$-convexly valued fields.}\/ We discuss this theory and the relevant QE theorem (which is less prominent in the literature than those for  $\operatorname{ACVF}$ and
 $\operatorname{RCVF}$: but see \cite{Belair,Guzy}) in Section~\ref{sec:pCVF}, before giving the proofs of our Theorems~A,~B,~C in Section~\ref{sec:p-adic NS}.
Finally,  Section~\ref{sec:rest} contains the  restricted analytic   adaptations of these results, including   proofs of Theorems~D and~E.

\subsection*{Related work}
The  Weierstrass systems used here are close to those of~\cite{DL,vdD88}. Cluckers-Lipshitz~\cite{CL} introduced a very general class of subrings of formal power series rings stable under Weierstrass Division, 
 called ``separated Weierstrass systems'',
which in a sense amalgamates the Weierstrass systems used here and our 
restricted Weierstrass systems from Definition~\ref{def:restricted W-system}. 
Given a separated Weierstrass system~$\mathcal A$, this leads to the concept of a ``separated analytic $\mathcal A$-structure'' on a henselian valued field, expanding our Definition~\ref{def:W-structure}, and in the case of characteristic zero, to a quantifier elimination in a certain multi-sorted
language~\cite[Theorem~6.3.7]{CL}. In the algebraically closed case (without assumptions on the characteristic) one also
has quantifier elimination  in a one-sorted language \cite[Theorem~4.5.15]{CL}. A one-sorted
QE for the valued field~$\Q_p$ expanded by primitives for the functions defined by
the restricted power series was proved by Denef and van den Dries~\cite{DD}.
Cubides Kovacics and Haskell~\cite{CH} similarly show that 
the real closed ordered valued field~$\k$ from the context of Theorems~D and~E with a separated analytic $\mathcal A$-structure where~$\mathcal A$ contains the rings~$\RX$ of restricted power series,
admits quantifier elimination in a natural one-sorted language. 
Bleybel~\cite{Bleybel} proves a quantifier elimination theorem for the $p$-adic analogue of the Levi-Civita field
(i.e., the completion of the group algebra
$\Q_p[t^\Q]$ equipped with the ultrametric absolute value satisfying~$\abs{at^q}=\operatorname{e}^{-q}$ for~$a\in\Q_p^\times$,~$q\in\Q$)
in a $3$-sorted language which contains function symbols for (among other things) 
the functions on the valuation ring given by
overconvergent $p$-adic power series. It would be interesting to know whether these     QE theorems for valued fields with analytic structure can be used to give alternative proofs of our Theorems~A--E. But we think that the advantage of our approach is that it shows that 
the relevant Nullstellens\"atze and related results can already be obtained from
the classical {\it algebraic}\/ QE theorems for valued fields (together with easy uses of Weierstrass Division), rather than their more involved {\it analytic}\/ counterparts.

\subsection*{Notational conventions}
Throughout this paper $k$, $l$, $m$, $n$ range over the set~$\N=\{0,1,2,\dots\}$ of natural numbers, and $p$ is a   prime number.
In this paper all rings are commutative with $1$. Let $R$ be a ring.
We let $R^\times$ be the group of units of~$R$, and if $R$ is an integral domain, then we denote the fraction field  of   $R$ by $\Frac(R)$.

\subsection*{Acknowledgments}
We thank the anonymous referees for the careful reading of the paper and for numerous suggestions (for example, to include Section~\ref{sec:mtvf}) which improved  the paper.

\section{Valuation-Theoretic Preliminaries}\label{sec:val th}

\noindent
In this paper a {\it valued field}\/ is   a field equipped with one of its valuation rings. Let~$K$ be a valued
field. Unless   specified otherwise we let $\mathcal O$ be the distinguished valuation ring of $K$,
with maximal ideal $\smallo$. The residue field of $K$ is $\res(K):=\mathcal O/\smallo$, with residue morphism $a\mapsto\overline{a}=a+\smallo\colon\mathcal O\to\res(K)$. 
The value group of $K$ is the ordered abelian group $\Gamma=K^\times/\mathcal O^\times$, written additively, with
ordering 
$a\mathcal O^\times \leq b\mathcal O^\times$ iff $b/a\in\mathcal O$. The map~$a\mapsto va=a\mathcal O^\times\colon K^\times\to\Gamma$ is a valuation on $K$.
If we need to indicate the dependence on $K$ we attach a subscript $K$, so $\mathcal O=\mathcal O_K$, $v=v_K$, etc.
If $K\subseteq L$ is an extension of valued fields (that is, $K\subseteq L$ is a field extension and~$\mathcal O=\mathcal O_L\cap K$),
then we identify $\res(K)$ with a subfield of $\res(L)$ and $\Gamma$ with an ordered subgroup of $\Gamma_L$ in the natural way.
For~$a,b\in K$ we set 
$$a\preceq b\, :\Leftrightarrow\, va\geq vb, \quad a\prec b \, :\Leftrightarrow\,  va>vb,\quad a\asymp b \, :\Leftrightarrow\, va =vb, \quad 
a\sim b \, :\Leftrightarrow\, a-b\prec a.$$
The binary relation $\preceq$   is an example of a dominance relation on $K$:

\begin{definition}
A {\em dominance relation} on an integral domain $R$ is a binary relation~$\preceq$ on~$R$ such that for all $f,g,h\in R$:
\begin{list}{*}{\setlength\leftmargin{1em}}
\item[(D1)] $1\not\preceq 0$;
\item[(D2)] $f\preceq f$;
\item[(D3)] $f\preceq g \text{ and } g\preceq h \Rightarrow f\preceq h$;
\item[(D4)] $f\preceq g$ or $g\preceq f$;
\item[(D5)] $f\preceq g \Leftrightarrow fh\preceq gh$, provided $h\ne 0$;
\item[(D6)] $f\preceq h \text{ and } g\preceq h \Rightarrow f+g \preceq h$.
\end{list}
\end{definition}

\noindent
Thus, if $\mathcal O$ is a valuation ring of $K$, we obtain a
dominance relation on $K$ via
\begin{equation}\label{domval}
f\preceq g \ :\Longleftrightarrow\  vf\geq vg \ \Longleftrightarrow\ 
\text{$f=gh$ for some $h\in \mathcal O$.}
\end{equation}
Conversely, if $\preceq$ is a dominance relation on $K$, then clearly
$$\mathcal O:=\{f\in K:f\preceq 1\}$$ is a valuation ring of $K$,
and if $v$
denotes the corresponding valuation on $K$, 
then the equivalence~\eqref{domval}
holds, for all $f,g\in K$. We call $\mathcal O$ the valuation ring {\em associated}\/ to
the dominance relation $\preceq$. This yields a one-to-one
correspondence between dominance relations on $K$ and valuation rings of $K$.

\medskip
\noindent
Let $R$ be an integral domain.
The {\em trivial}\/  dominance relation 
$\preceq_{\operatorname{t}}$  on $R$ is the one with~$r \preceq_{\operatorname{t}} s$ for all $r,s\in R$ with~$s\neq 0$.
For any dominance relation $\preceq$ on $R$ there is a unique dominance relation $\preceq_F$
on $F=\Frac(R)$ such that $(R,{\preceq})\subseteq (F,{\preceq_F})$:
$$\frac{r_1}{s} \preceq_F \frac{r_2}{s}\quad\Longleftrightarrow\quad r_1\preceq r_2
\qquad (r_1,r_2,s\in R,\ s\neq 0).$$

\begin{example} 
The non-trivial dominance relations on $\Z$ are exactly the $p$-adic dominance relations
$\preceq_p$, given by
$a \preceq_p b$ iff for all $n$: $b\in p^n\Z\Rightarrow a\in p^n\Z$.
\end{example}

\subsection{Model-theoretic treatment of valued fields}\label{sec:mtvf}
We now briefly explain basic model-theoretic terminology and notation, focussing on valued fields. A concise presentation of  all the general model theory needed here is in
\cite[Appendix~B]{ADH}. A textbook  which also pays attention to valued fields is \cite{PrestelDelzell};
for more specialized expositions of the model theory of valued fields  see~\cite{Chatzidakis,vdD}.

A (one-sorted) language $\mathcal L$ is a pair $(\mathcal L^{\operatorname{r}},\mathcal L^{\operatorname{f}})$
where $\mathcal L^{\operatorname{r}}$, $\mathcal L^{\operatorname{f}}$ are disjoint (possibly empty,  or infinite) collections of {\it relation symbols}\/ and  
 {\it function symbols}\/, respectively. Each of these symbols has an associated {\it arity}\/ (a   natural number);
 symbols of arity~$n$ are also said to be   {\it $n$-ary.}\/  A $0$-ary
 function symbol is called a {\it constant symbol.}\/
It is customary to present a language as a disjoint union $\mathcal L = \mathcal L^{\operatorname{r}}\cup\mathcal L^{\operatorname{f}}$ of its sets of relation and function symbols, while separately specifying the arities of the
 various   symbols. 
For example, $\mathcal L_{\operatorname{R}}=\{0,1,{-},{+},{\,\cdot\,}\}$ is the language  of rings: here $0$, $1$ are constant symbols, $-$ is a unary function symbol, and~$+$,~$\cdot$ are binary function symbols (so~$\mathcal L_{\operatorname{R}}$ has no relation symbols).  Another   example,   relevant for this paper, is
the expansion~$\mathcal L_{\preceq}=\mathcal L_{\operatorname{R}}\cup \{\preceq\}$ of $\mathcal L_{\operatorname{R}}$  by a binary relation symbol~$\preceq$. 
Here,   $\mathcal L$ is an {\it expansion}\/ of a language $\mathcal L_0$, and $\mathcal L_0$ a {\it reduct}\/ of $\mathcal L$,
if  $\mathcal L_0^{\operatorname{r}}\subseteq  \mathcal L^{\operatorname{r}}$ and~$\mathcal L_0^{\operatorname{f}}\subseteq  \mathcal L^{\operatorname{f}}$, with unchanged arities for the symbols in $\mathcal L_0$.

 Let $\mathcal L$ be a language. An {\it $\mathcal L$-structure}\/ $\mathbf M$ consists of its {\it underlying set}\/ $M\neq\emptyset$ together with 
  {\it interpretations}\/ of each symbol of $\mathcal L$ in $\mathbf M$: for each $m$-ary
  $R\in \mathcal L^{\operatorname{r}}$,   a relation~$R^{\mathbf M}\subseteq M^m$,  and for $n$-ary $f\in\mathcal L^{\operatorname{f}}$, a function~$f^{\mathbf M}\colon M^n\to M$.
(So each constant symbol $c$ is interpreted by a function~$M^0\to M$,   identified with the unique element $c^{\mathbf M}$ in its image,  $M^0$ being a singleton.)
In practice the superscript~$\mathbf M$ is often suppressed, thus a structure and its universe are denoted by the same letter, and so are the symbols of~$\mathcal L$ and their interpretations in $\mathbf M$.
{\it Morphisms,}\/ {\it embeddings,}\/ and {\it substructures}\/ of $\mathcal L$-structures are defined in the natural way; e.g., an $\mathcal L$-structure $\mathbf M$ is a substructure of an $\mathcal L$-structure~$\mathbf N$ if~$M\subseteq N$ and~$R^{\mathbf M}=R^{\mathbf N}\cap M^m$ and
$f^{\mathbf N}|_{M^n}=f^{\mathbf M}$ for each $m$-ary~$R\in  \mathcal L^{\operatorname{r}}$ and
$n$-ary~$f\in\mathcal L^{\operatorname{f}}$. If $\mathcal L_0$ is a reduct of $\mathcal L$, then $\mathbf M|\mathcal L_0$ is the {\it $\mathcal L_0$-reduct}\/ of $\mathbf M$: the $\mathcal L_0$-structure with same underlying set $M$ and the same interpretations of the symbols in $\mathcal L_0$ as in $\mathbf M$.

Each ring $R$ can be construed as an  $\mathcal L_{\operatorname{R}}$-structure   by 
interpreting~$0$,~$1$ by the additive and multiplicative identity of $R$, respectively, $-$ by the function $r\mapsto -r$,
and $+$, $\cdot$ by~$(r,s)\mapsto r+s$ and~$(r,s)\mapsto r\cdot s$, respectively.
Morphisms, embeddings, and substructures of $\mathcal L_{\operatorname{R}}$-structures then correspond to
morphisms, embeddings, and substructures of rings, respectively.
(But of course we are, in principle, free to also view $R$ as 
an  $\mathcal L_{\operatorname{R}}$-structure in an unnatural way, by
choosing the interpretations of the symbols of $\mathcal L_{\operatorname{R}}$
any way we like.)
In this paper we also always construe each valued field $(K,\mathcal O)$ as an $\mathcal L_{\preceq}$-structure by 
interpreting~$\preceq$ as the dominance
relation   on $K$ given by \eqref{domval}.
Then ``valued subfield'' corresponds to ``$\mathcal L_{\preceq}$-substructure.''
(This way of treating valued fields as model theoretic structures is essentially equivalent to
that in \cite[Chapter~4]{PrestelDelzell}, which employs valuation divisibilities instead of dominance relations. For 
other choices see \cite{Chatzidakis,vdD}.)

The purpose of a language $\mathcal L$ is to generate {\it $\mathcal L$-formulas,}\/ which 
express properties of (finite tuples of) elements of $\mathcal L$-structures.
They  are (well-formed, in a certain precise sense) finite words on an alphabet consisting of
 $\mathcal L^{\operatorname{r}}\cup \mathcal L^{\operatorname{f}}$ together with the symbols
 $\top$, $\perp$, $\neg$, $\vee$, $\wedge$, $=$, $\forall$, $\exists$, 
 to be thought of as {\it true,}\/ {\it false,}\/ {\it not,} {\it or,} {\it and,} {\it equals,} {\it there exists,} and {\it for all,}\/ respectively,
 together with an infinite  collection of {\it variables}\/ (usually written $x$, $y$, $z$, \dots),
 thought of as ranging over the universe of an $\mathcal L$-structure (and not over, say, subsets thereof).
 Variables which do not appear within the scope of a quantifier in an $\mathcal L$-formula $\varphi$ are said to be {\it free}\/
 in~$\varphi$, and $\mathcal L$-formulas without free variables are {\it $\mathcal L$-sentences.}\/
 For example, the commutativity of addition in a ring is expressed by the $\mathcal L_{\operatorname{R}}$-sentence
 $\forall x\forall y(x+y=y+x)$. Strictly speaking, this should be $\forall x\forall y =+xy+yx$; but as usual we take the liberty of writing binary  symbols, including $=$, in infix 
 rather than prefix notation, and of freely employing parentheses to enhance   readability.
Below~$\varphi$,~$\psi$,~$\theta$ range over $\mathcal L$-formulas and $\sigma$ over $\mathcal L$-sentences.
 We also use common abbreviations like   $\forall x$ for~$\forall x_1\cdots\forall x_m$,
where $x=(x_1,\dots,x_m)$ is a tuple of  distinct variables, and~$\varphi\leftrightarrow\psi$ in place of $(\varphi\wedge\psi)\vee (\neg\varphi\wedge\neg\psi)$.

  Given  $x$ as above, writing $\varphi(x)$   indicates that $x$ contains all free variables in~$\varphi$.
For an $\mathcal L$-structure~$\mathbf M$ and $a\in M^m$ one then defines, by induction on the construction of~$\varphi$, when {\it $\mathbf M$ satisfies $\varphi(a)$}\/---in symbols: $\mathbf M\models\varphi(a)$---in the natural way.
The subset of $M^m$ {\it defined by $\varphi(x)$}\/ in~$\mathbf M$ is~$\varphi^{\mathbf M}:=\big\{a\in M^m:\mathbf M\models\varphi(a)\big\}$.
(For example, for a valued field~$K$ and the $\mathcal L_{\preceq}$-formulas~$\varphi(x)$, $\psi(x)$, $\theta(x)$, where $x$ is  a single variable, given by $x\preceq 1$, ${x\preceq 1}\wedge {\neg 1\preceq x}$, and~${x\preceq 1}\wedge\exists y({x\cdot y=1}\wedge y\preceq 1)$, respectively, we have~$\mathcal O=\varphi^K$, $\smallo=\psi^K$, and~$\mathcal O^\times=\theta^K$.)
Note: if $y$ is a variable which neither occurs in $x$ nor free in $\varphi$, then 
the $(m+1)$-tuple $(x,y)$ also contains all free variables in~$\varphi$, so we could write
$\varphi=\varphi(x,y)$ according to our earlier convention; however,
for~$a\in M^m$,~$b\in M$ we have $\mathbf M\models\varphi(a,b)$ iff~$\mathbf M\models\varphi(a)$, and thus
in particular, for a $\mathcal L$-sentence $\sigma$ the status of $\mathbf M\models\sigma$
does not depend on~$a$.   The $\mathcal L$-structures~$\mathbf M$,~$\mathbf N$ are said to be {\it elementarily
equivalent}\/ ($\mathbf M\equiv\mathbf N$) if for each $\sigma$ we have~$\mathbf M\models\sigma$ iff~$\mathbf N\models \sigma$. 
   
Let $\Sigma$ be a set of $\mathcal L$-sentences. Then $\mathbf M$ is a {\it model}\/ of $\Sigma$ (in symbols: $\mathbf M\models\Sigma$) if~$\mathbf M\models\sigma$
for all $\sigma\in\Sigma$. We also write $\Sigma\models\sigma$ if every model of $\Sigma$ satisfies~$\sigma$.  If~$\Sigma$ contains each $\sigma$ with $\Sigma\models\sigma$, then $\Sigma$ is   an {\it $\mathcal L$-theory.}\/ 
  Clearly $\Sigma^{\models}:={\{\sigma:\Sigma\models\sigma\}}$ is the smallest $\mathcal L$-theory containing $\Sigma$. 
If $\Sigma$ has a model, then $\mathbf M\equiv\mathbf N$ for all~${\mathbf M,\mathbf N\models\Sigma}$ iff
for all $\sigma$, either~$\Sigma\models\sigma$ or~$\Sigma\models\neg\sigma$; in this case, $\Sigma$ is said to be
{\it complete.}\/  
Taking~$\mathcal L=\mathcal L_{\preceq}$ and for $\Sigma$ the set of $\mathcal L_{\preceq}$-sentences
  which consists of the axioms for fields (which can be formulated in the sublanguage~$\mathcal L_{\operatorname{R}}$ of~$\mathcal L_{\preceq}$) together with (universal) $\mathcal L_{\preceq}$-sentences formalizing the properties~(D1)--(D6),
  we obtain the $\mathcal L_{\preceq}$-theory $\operatorname{VF}:=\Sigma^{\models}$ of valued fields.
The $\mathcal L_{\preceq}$-theory $\operatorname{VF}$ is not complete since (for example) it 
has models of different characteristic.

If $A\subseteq M$, we let $\mathcal L_A$ be the expansion of $\mathcal L$ by a 
new constant symbol $\underline{a}$ for each~$a\in A$, and $\mathbf M_A$ be the $\mathcal L_A$-structure with $\mathcal L$-reduct $\mathbf M$ where each $\underline{a}$ is interpreted by $a$.
We say that $\mathbf M$ is an {\it elementary substructure}\/ of $\mathbf N$, and $\mathbf N$ an {\it elementary extension}\/ of $\mathbf M$,
if~$M\subseteq N$ and $\mathbf M_M\equiv\mathbf N_M$. If every extension of models of $\Sigma$ is elementary, then
$\Sigma$ is said to be {\it model complete.}\/ 
Finally, we say that  $\varphi$ is quantifier-free if $\forall$, $\exists$ do not occur in $\varphi$, and~$\Sigma$ is said
to admit {\it quantifier elimination}\/ (QE) if for each  $\varphi(x)$, where $x=(x_1,\dots,x_m)$, there is a quantifier-free~$\psi(x)$ with $\Sigma\models\forall x(\varphi\leftrightarrow\psi)$.
In this case, $\Sigma$ is also model complete. Moreover, $\Sigma$ admits QE iff it is {\it substructure complete}\/:
if $\mathbf M,\mathbf N\models\Sigma$ and $\mathbf A$ is a common substructure of both $\mathbf M$ and $\mathbf N$,
then $\mathbf M_A\equiv\mathbf N_A$.
 (See, e.g., \cite[Corollary~B.11.6]{ADH} or \cite[remark after Theorem~3.4.1]{PrestelDelzell}.)

\subsection{Extension of valuations}
Let  $A$ be a local subring of a field~$K$ with maximal ideal~$\fm=\fm_A$  and residue field $\k=A/\fm$.
We recall that any maximal element of the class of all local subrings of $K$ lying over $A$, partially ordered by~$B\leq B' :\Longleftrightarrow$ $B'$ lies over $B$, is a valuation ring of $K$.
 (Krull; cf., e.g.,~\cite[Pro\-po\-si\-tion~3.1.13]{ADH}.) 
 In particular, by Zorn, there is always a valuation ring of $K$ lying over $A$. Under suitable hypotheses on $A$ we can
be more precise (see, e.g.,~\cite[Lem\-ma~A1]{JR}):

\begin{lemma}\label{lem:existence of places}
If   $A$ is a regular local ring and $K=\Frac(A)$,
then there is a valuation ring~$\mathcal O$ of~$K$ lying over $A$ such that the natural inclusion $A\to\mathcal O$
induces an isomorphism $\k\to\mathcal O/\smallo$.
\end{lemma}

\noindent
(Indeed, for each regular system of parameters  $x_1,\dots,x_d$ of $A$ there is a unique
valuation ring $\mathcal O$ as in
Lemma~\ref{lem:existence of places}  with
$\Gamma=\Z v(x_1)\oplus\cdots\oplus \Z v(x_d)$, ordered lexicographically.)

\subsection{Coarsening and specialization}
Let $K$ be a valued field and
$\Delta$ be a convex subgroup of
$\Gamma=v(K^\times)$. We have the ordered quotient group~$\dot{\Gamma}=\Gamma/\Delta$
and the valuation 
$$\dot{v}=v_\Delta \colon
K^\times\to\dot{\Gamma}, \qquad
\dot{v}(a):=v(a) + \Delta\text{ for $a\in K^\times$}$$ on the field $K$, 
the 
{\em $\Delta$-coarsening of $v$}.
The valuation ring of $\dot{v}$ is
$$\dot{\mathcal{O}}\ =\ \mathcal O_\Delta\ :=\ \{a\in K: \text{$va\geq \delta$  for some  $\delta\in \Delta$} \},$$
which has $\mathcal{O}$ as a subring and has maximal ideal  
$$\dot{\smallo}\ :=\ \{a\in K: va>\Delta\}\subseteq \smallo.$$ 
The valued field $(K,\dot{\mathcal O})$  is called the {\it $\Delta$-coarsening of~$K$}. 
Put $\dot K:=\dot{\mathcal{O}}/\dot{\smallo}$, 
the residue field of $\dot{\mathcal{O}}$, and   
$\dot{a}:= a + \dot{\smallo} \in \dot{K}$ for 
$a\in \dot{\mathcal{O}}$. 
Then for $a\in \dot{\mathcal{O}}\setminus\dot{\smallo}$ 
the value~$va$ depends only on the residue class $\dot a\in \dot K$. This gives
the valuation~$v\colon \dot K^{\times} \to \Delta$, $v\dot a:= va$
with valuation ring~$\mathcal{O}_{\dot K} = \{\dot{a}: a\in \mathcal{O}\}$, having maximal ideal $\smallo_{\dot K} = \{\dot{a}:a\in \smallo\}$. Throughout 
$\dot{K}$ stands for the valued field $(\dot{K},\mathcal{O}_{\dot K})$, called the {\em $\Delta$-specialization} of~$K$.
The composed map 
$$\mathcal{O} \to \mathcal{O}_{\dot K} \to \res(\dot K)=
\mathcal{O}_{\dot K}/\smallo_{\dot K}$$
has kernel $\smallo$, and thus 
induces a field isomorphism 
$\res(K) \overset{\cong}{\longrightarrow} \res(\dot K)$, and
we identify~$\res(K)$ and $\res(\dot K)$ via this map.

Every  subring of $K$ containing $\mathcal O$ is a valuation ring of $K$.
In fact,   every such subring~$\mathcal O_1\supseteq\mathcal O$  of $K$
arises as the valuation ring   of the $\Delta_1$-coarsening of $v$ for some convex subgroup $\Delta_1$ of~$\Gamma$, namely
  $\Delta_1:=v(\mathcal O_1\setminus\{0\})$. The map $\mathcal O_1\mapsto\Delta_1$ is an inclusion-preserving
bijection from the set of subrings of $K$ containing $\mathcal O$ onto the set of convex subgroups of $\Gamma$. (Cf., e.g., \cite[Lemmas~3.1.4, 3.1.5]{ADH}.)
In particular, the collection of subrings of $K$ containing~$\mathcal O$ is totally ordered under inclusion, and we have $\mathcal O_1\supset\mathcal O$ iff $\Delta_1\neq\{0\}$.
The following is easy to verify:

\begin{lemma}\label{lem:lift val spec}
Let $\k:=\res(K)$ and $\mathcal O_{\k}$ be a valuation ring of  $\k$. Then $\mathcal O_0 := {\{ a\in\mathcal O: \overline{a}\in\mathcal O_{\k} \}}$ is a valuation ring of $K$.  
If~$\Delta_0$ is the convex subgroup of the value group of~$\mathcal O_0$  with~$\dot{\mathcal O_0}=\mathcal O$,
then  the $\Delta_0$-specialization of~$(K,\mathcal O_0)$
is $(\k,\mathcal O_{\k})$.
\end{lemma}

\subsection{Convexity} \label{sec:convexity}
Let $K$ be an ordered field. Recall that 
a subring  of $K$ is convex iff it contains the interval $[0,1]$, and as a consequence,
every convex subring of $K$ is a valuation ring of $K$. (See, e.g., \cite[Lemma~3.5.10]{ADH}.)
Let $R$ be a subring  of $K$; then the convex hull 
$$A := \big\{ a\in K: \text{$\abs{a}\leq \varepsilon$ for some $\varepsilon\in R$, $\varepsilon>0$} \big\}$$ 
of $R$ in $K$ is a subring of $K$. This is the smallest convex subring of $K$ containing~$R$,
with maximal ideal
$$\fm=\big\{a\in K:\text{$\abs{a}\leq 1/\varepsilon$ for all $\varepsilon\in R$, $\varepsilon>0$} \big\}.$$
Let $\mathcal O_0$ be  the convex hull of $\Q$ in $K$. Then
a subring of $K$ is convex iff it contains~$\mathcal O_0$; in particular, the 
collection of convex subrings of $K$ is totally ordered under inclusion.
We   generalize these concepts to an arbitrary valued field.
Thus   let now $K$ be a valued field with valuation ring $\mathcal O$ and associated dominance relation~$\preceq$.

\begin{definition}
We say that a subring  of $K$ is {\em convex}\/ (in the valued field~$K$) if it contains~$\mathcal O$. 
Clearly $\mathcal O$ and $K$ are convex subrings of $K$; we call these the {\em trivial}\/  convex   subrings of $K$.
\end{definition}

\noindent
Note that every  convex subring of $K$ is a valuation ring of $K$. In fact, for each convex subgroup $\Delta$ of the value group of $K$, the valuation ring of the $\Delta$-coarsening of   $K$ is a convex subring of $K$, and each convex subring of $K$ arises this way.  Thus the  
set of convex subrings of $K$ is totally ordered under inclusion.

\begin{lemma}\label{lem:convex hull}
Let~$R$ be a subring of   $K$. Then the subring $A=\mathcal OR$ of $K$ generated by $\mathcal O$ and $R$ 
is the smallest convex subring of $K$ which contains $R$. We have
$$A = \big\{ a\in K: \text{$a \preceq \varepsilon$ for some  $\varepsilon\in R$} \big\},$$
with maximal ideal
$$\fm  = \big\{ a\in K: \text{$a\prec 1/\varepsilon$ for all non-zero $\varepsilon\in R$}\big\}.$$
\end{lemma}
\begin{proof}
The first statement is easy to verify; for the second statement use that the maximal ideal of the local ring~$A$
is given by $\fm=\{a\in A:\text{$a=0$ or $a\notin A^\times$}\}$.
\end{proof}

\noindent
Given a subring $R$ of $K$, we call   $A=\mathcal OR$  the
{\it convex hull}\/ of $R$ in $K$, denoted by~$\operatorname{conv}_K(R)$, or by $\operatorname{conv}(R)$ if $K$ is clear from the context.
(If~$K$ comes equipped with a field ordering and $\mathcal O=\mathcal O_0$  as above, then this agrees with
the convex hull of $R$ in the ordered field~$K$.)
The map $R\mapsto \operatorname{conv}(R)$ is a closure operator on the collection of subrings of $K$, that is,
$$R\subseteq \operatorname{conv}(R),\quad R\subseteq S\ \Longrightarrow\ \operatorname{conv}(R)\subseteq \operatorname{conv}(S),
\quad \operatorname{conv}(\operatorname{conv}(R))=\operatorname{conv}(R).$$
Let $L$ be a valued field extension of $K$ and $R$ be a convex subring of $K$.
It is easily seen that  
$\operatorname{conv}_L(R)$ lies over $R$.
Moreover~(cf.~\cite[remarks after Proposition~3.1.20]{ADH}):

\begin{lemma}\label{lem:unique convex ext}
If $L$ is algebraic over $K$, then $\operatorname{conv}_L(R)$ is the unique
convex subring of $L$ lying over $R$.
\end{lemma}

\noindent
We conclude this section with some reminders about  $\Z$-groups.

\subsection*{$\Z$-groups}
Let $\Gamma$ be an ordered abelian group, written additively.  We   let~$\Gamma^>:={\{\gamma\in \Gamma:\gamma>0\}}$, and similarly with $\geq$ instead of $>$.
If $\Gamma^>$ has a  smallest   element, denoted by $1$, then we   identify $\Z$ with an ordered   subgroup of $\Gamma$
via the embedding~$r\mapsto r\cdot 1\colon \Z\to\Gamma$, making~$\Z$ the smallest non-zero  convex subgroup of $\Gamma$.
We recall that $\Gamma$ is said to be a {\it $\Z$-group}\/ if $\Gamma^>$ has a smallest element $1$ and for each~$n\geq 1$ and $\gamma\in\Gamma$ there are some~$i\in\{0,\dots,n-1\}$ such that~$\gamma-i\in n\Gamma$.
The ordered abelian group $\Z$ of integers 
is a $\Z$-group, and every $\Z$-group is elementarily equivalent to $\Z$ in the language $\{0,{-},{+},{\leq}\}$
of ordered abelian groups. (See, e.g.,~\cite[\S{}4.1]{PrestelDelzell}.)
For future reference here are a
few other well-known (and easy to prove) facts :

\begin{lemma}\label{lem:Z-gps}
Let $\Delta\neq\{0\}$ be a convex subgroup of $\Gamma$.
Then $\Gamma$ is a $\Z$-group iff $\Delta$ is a $\Z$-group and $\Gamma/\Delta$ is divisible.
\end{lemma}

\begin{lemma}\label{lem:Z-subgp}
Let $\Delta$ be a subgroup of a $\Z$-group $\Gamma$ which contains the smallest  element of~$\Gamma^>$. Then
$\Delta$ is a $\Z$-group iff $\Gamma/\Delta$ is torsion-free.
\end{lemma}

\section{Weierstrass Systems and Infinitesimal Analytic Structures}\label{sec:W-systems}

\noindent
In this section we let $\k$ be a field, of arbitrary characteristic unless stated otherwise. We let $(X_m)_{m\geq 1}$ be a sequence of distinct indeterminates
over $\k$. As usual~$\k[[X_1,\dots,X_m]]$ denotes the ring of formal power series over $\k$ in the indeterminates $X_1,\dots,X_m$.
In the following we let $X=(X_1,\dots,X_m)$,  we let $\alpha$, $\beta$   range over $\N^m$, and for~$\alpha=(\alpha_1,\dots,\alpha_m)$ put~$\abs{\alpha}:=\alpha_1+\cdots+\alpha_m$ and $X^\alpha:=X_1^{\alpha_1}\cdots X_m^{\alpha_m}$.
Consider an element
$$f=\sum_\alpha f_\alpha X^\alpha  \qquad (f_\alpha\in\k)$$
 of $\k[[X]]$.
If $f\neq 0$, then
the {\it order}\/ of $f$ is the smallest $d$ such that  $f_\alpha\neq 0$ for some~$\alpha$ with~$\abs{\alpha}=d$.
Suppose now $m\geq 1$.
One says that $f\in \k[[X]]$ is \emph{regular of order~$d\in\N$ in $X_{m}$}\/ if~$f(0,X_{m})\in \k[[X_{m}]]$ is non-zero of order $d$.
Note that if~$f=f_1\cdots f_n$ where~$f_1,\dots,f_n\in\k[[X]]$, then $f$ is regular in $X_m$ of order $d$ iff   $f_i$ is regular in $X_m$ of order $d_i$ for $i=1,\dots,n$,
with $d_1+\cdots+d_n=d$.

\medskip
\noindent
Let $d\in\N$, $d\geq 1$.
Then we have a $\k$-algebra automorphism $\tau_d$ of  $\k[[X]]$ 
such that~$\tau_d(X_i)=X_i+X_m^{d^{m-i}}$ for $i=1,\dots,m-1$ and~$\tau_d(X_m)=X_m$.
(See,  e.g.,~\cite[I, \S{}4, Satz~3 and the remark following it]{GR}.)
For non-zero~$f\in\k[[X]]$ there is always some $d\in\N$, $d\geq 1$ such that~$\tau_d(f)$ is regular in~$X_m$ (of some order). In particular,
given non-zero $f_1,\dots,f_n\in \k[[X]]$,   there is some $d\in\N$, $d\geq 1$
such that~$\tau_d(f_1),\dots,\tau_d(f_n)$   are all regular in $X_m$.

\medskip 
\noindent
We also recall that $\k[[X]]$  is a complete local noetherian integral domain with maximal ideal  generated by $X_1,\dots,X_m$. 
Thus  for $f\in \k[[X]]$ we have
$f\in \k[[X]]^\times$ iff $f(0)\in\k^\times$;
we say that $f$ is a {\em $1$-unit}\/ if $f(0)=1$.

\subsection{Weierstrass systems}
The following definition is a special instance of a concept introduced by Denef-Lipshitz~\cite{DL}:

\begin{definition} \label{def:DL}
A {\em Weierstrass system over $\k$} is a family of rings $(W_m)_{m\geq 0}$ 
such that for all $m$ we have, with $X=(X_1,\dots,X_m)$:
\begin{itemize}
\item[(W1)] $\k[X] \subseteq  W_m \subseteq \k[[X]]$;
\item[(W2)] for each
permutation $\sigma$ of $\{1,\dots,m\}$,   the $\k$-algebra automorphism
$f(X) \mapsto  f(X_{\sigma(1)},\dots,X_{\sigma(m)})$ of $\k[[X]]$ maps $W_m$ onto itself;
\item[(W3)] 
$W_m \cap \k[[X']] = W_{m-1}$ for $X'=(X_1,\dots,X_{m-1})$, $m\geq 1$; and
\item[(W4)] if $g\in W_m$ is regular of order $d$ in $X_m$ ($m\geq 1$),
 then for every $f\in 
W_m$ there are~$q\in W_m$ and
a polynomial $r\in W_{m-1}[X_{m}]$ of degree $<d$ (in $X_{m}$) such that
$f=qg+r$. (Weierstrass Division.)
\end{itemize}
(In the case where $\k$ has positive characteristic,  \cite{DL} includes further axioms, but those are not needed for our purposes.)
\end{definition}

\noindent
Examples of Weierstrass systems include 
\begin{enumerate}
\item the system 
$\bigl(\k[[X_1,\dots,X_m]]\bigr)$
consisting of all formal power series rings;
\item the system $\bigl(\k[[X_1,\dots,X_m]]^{\operatorname{a}}\bigr)$, where $\k[[X_1,\dots,X_m]]^{\operatorname{a}}$ is
the ring of power series in~$\k[[X_1,\dots,X_m]]$ which are algebraic over $\k[X_1,\dots,X_m]$ (cf.~\cite{Lafon});
\item if $\operatorname{char}\k=0$, the system $\bigl(\k[[X_1,\dots,X_m]]^{\operatorname{da}}\bigr)$, where $\k[[X_1,\dots,X_m]]^{\operatorname{da}}$ is
the ring of all power series  $f\in \k[[X]]=\k[[X_1,\dots,X_m]]$ which are \emph{differentially algebraic} over $\k$,
that is, the fraction field of the subring $\k\big[\partial^{\abs{\beta}} f/ \partial X^\beta\big]$ of~$\k[[X]]$
generated by the partial derivatives of $f$
has finite transcendence degree over $\k$ (see \cite[\S{}5]{vdD88});
\item assuming that $\k$ comes equipped with a complete absolute value, 
the system~$\bigl(\k\{X_1,\dots,X_m\}\bigr)$, where~$\k\{X_1,\dots,X_m\}$
consists of all power series in~$\k[[X_1,\dots,X_m]]$ that converge in some 
neighborhood of the origin. (See, e.g.,~\cite[Kapitel~I, \S{}4]{GR}.)
\end{enumerate}
{\it In the rest of this section we let  $W=(W_m)$  be a Weierstrass system over $\k$.}\/ Note that the definition of  Weierstrass system given above implicitly  depends on our fixed choice of the sequence $(X_m)$ of indeterminates.
If~$Y=(Y_1,\dots,Y_m)$ is an arbitrary tuple of distinct indeterminates, then we
let $\k \lfloor Y \rfloor := \big\{f\in \k[[Y]]: f(X)\in W_m\big\}$.
Here as before   $X=(X_1,\dots,X_m)$. We also denote by $(X)$ the ideal of~$W_m=\k\lfloor X\rfloor$
generated by $X_1,\dots,X_m$. 
In the rest of this subsection we  also let
$Y=(Y_1,\dots,Y_n)$ be a tuple of distinct indeterminates disjoint from~$X$.

With these conventions, we now list  some basic facts about Weierstrass systems which we will be using.

\begin{lemma}\label{lem:units}
$W_m \cap \k[[X]]^\times =W_m^\times$.
\end{lemma}
\begin{proof}
Let $g\in W_m \cap \k[[X]]^\times$. Then $X_m g$ is regular of order $1$ in $X_m$, hence   (W4) yields~$X_m=(X_m g)q+r$ where $q\in W_m$, $r\in W_{m-1}\subseteq\k[[X']]$, thus $r=0$ and so~$1=gq$.
\end{proof}

\noindent
By \cite[Remarks~1.3(2)]{DL} we have:

\begin{lemma}\label{lem:DL subst}
Let $f\in \k\lfloor X,Y\rfloor$, $g\in (X)^n$.
Then $f(X,g(X))\in \k\lfloor X\rfloor$.
\end{lemma}
\begin{proof}
For later use we include the proof. With $g=(g_1,\dots,g_n)$, repeated application of (W4) in $\k\lfloor X,Y\rfloor$ yields
\begin{align*}
f(X,Y) 	&= u_n(X,Y) (Y_n-g_n(X))+R_1(X,Y_1,\dots,Y_{n-1}) \\
		&= u_n(X,Y) (Y_n-g_n(X))+u_{n-1}(X,Y)(Y_{n-1}-g_{n-1}(X))  \\
		&\qquad\qquad +R_2(X,Y_1,\dots,Y_{n-2}) \\
		&\qquad \vdots \\
		&= u_n(X,Y) (Y_n-g_n(X))+\cdots+ u_1(X,Y)(Y_1-g_1(X)) + R_n(X)
\end{align*}
and hence $f(X,g(X))=R_n(X)\in \k\lfloor X\rfloor$.
\end{proof}

\begin{cor}\label{cor:DL subst}
For all $f\in \k\lfloor Y\rfloor$ and $g\in (X)^n$  we have
$f(g(X))\in \k\lfloor X\rfloor$.
\end{cor}
\begin{proof}
By (W3) we have
$$f':=f(X_1,\dots,X_n)\in   \k\lfloor X_1,\dots,X_{m+n}\rfloor.$$ 
Let $\sigma$ be a permutation of $\{1,\dots,m+n\}$ with $\sigma(i)=m+i$ for $i=1,\dots,n$; then
$$h:=f'(X_{\sigma(1)},\dots,X_{\sigma(m+n)})\in \k\lfloor X_1,\dots,X_{m+n}\rfloor$$ by (W2).
So by Lemma~\ref{lem:DL subst} applied to $h$ in place of $f$ and $Y=(X_{m+1},\dots,X_{m+n})$ we get
$f(g(X))=h(X,g(X))\in  \k\lfloor X\rfloor$.
\end{proof}

\noindent
By~\cite[Remarks~1.3(1),(7),(8)]{DL}, we have:

\begin{lemma}\label{lem:W-Lemma}
The ring
$\k\lfloor X\rfloor$ is   noetherian regular  local; its unique maximal ideal  is~$(X)$, and  its  completion is  $\k[[X]]$.
\end{lemma}


\noindent
Together with \cite[Theorem~8.14]{Matsumura} this yields:

\begin{cor}\label{cor:fflat}
The ring $\k[[X]]$ is a faithfully flat   $\k\lfloor X\rfloor$-module.
\end{cor}

\noindent
The following is \cite[Remarks~1.3(4)]{DL}:

\begin{lemma}[Implicit Function Theorem]\label{lem:IFT}
Let $f=(f_1,\dots,f_n)\in \k\lfloor X,Y\rfloor^n$ be
such that 
$$f\equiv 0\mod (X,Y)\quad\text{and}\quad \det \left( \partial f_i/\partial Y_j \right)\not\equiv 0 \mod (X,Y).$$ 
Then there
is some $y=(y_1,\dots,y_n)\in (X)^n$  such that $f(X,y)=0$.
\end{lemma}

\noindent
Using a special case of  the previous lemma we obtain:

\begin{lemma}\label{lem:henselian}
The local ring $R=\k\lfloor X\rfloor$ is henselian.
\end{lemma}
\begin{proof}
Let $Y$ be a single indeterminate over $R$ and
$$P(Y)=1+Y+b_2Y^2+\cdots+b_dY^d\in R[Y]$$
where $d\geq 2$, $b_2,\dots,b_d\in (X)$.
It suffices to show that there is some~$y\in R$ such that $P(y)=0$.
To see this consider $f(X,Y):=P(Y-1)\in  R[Y]\subseteq \k\lfloor X,Y\rfloor$.
Then~$f\equiv 0\bmod (X,Y)$ and $\partial f/\partial Y\equiv 1\bmod (X,Y)$. Now use Lemma~\ref{lem:IFT}.
\end{proof}

\noindent
 The local ring $\k[[X]]^{\operatorname{a}}$ is the henselization of its local subring $\k[X]_{\fm}$ where~$\fm$ is the
 ideal of $\k[X]$ generated by $X_1,\dots,X_m$. (See, e.g., \cite[Th\'eor\`eme~4]{Lafon63}.)
  From this fact in combination with the universal property of henselizations
 and Lemma~\ref{lem:henselian}, we obtain $\k[[X]]^{\operatorname{a}}\subseteq \k\lfloor X\rfloor$. This is not used later,
 in the contrast to the following, more  immediate consequence of Lemma~\ref{lem:henselian}:
 
\begin{cor}\label{cor:roots}
Let $f\in\k \lfloor X \rfloor$, $k\geq 1$, and suppose $\operatorname{char}\k=0$ or $\operatorname{char}\k=p>0$ and~$p\!\not|k$. Then
$$\text{$f=g^k$ for some $g\in \k \lfloor X \rfloor^\times$}\quad\Longleftrightarrow\quad\text{$f(0)=a^k$ for some $a\in\k^\times$.}$$
In particular, the group of $1$-units of $\k\lfloor X\rfloor$ is $k$-divisible.
\end{cor}

\noindent 
Recall that   $\k$  is called \textit{euclidean}\/ if $a^2+b^2\neq -1$ for all $a,b\in\k$, and
for all $a\in\k$ there is some $b\in\k$ with $a=b^2$ or $a=-b^2$.
In this case $\k$ has a unique ordering making it an ordered field, given by $a\geq 0\Leftrightarrow a=b^2$ for some $b\in\k$, and below we then always
view $\k$ as an ordered field this way.
(For example,  each real closed field is euclidean, as is the field of real constructible numbers; but neither   $\Q$ nor any algebraic number field is euclidean~\cite[Corollary~5.12]{Lam}.) 
 
\begin{lemma}\label{lem:u>0}
Suppose $\k$ is euclidean.
Let $K$ be an ordered field and $\sigma\colon  \k\lfloor X\rfloor\to K$ be a ring morphism, and
let $u\in  \k\lfloor X\rfloor$. Then 
$$u(0)>0\quad\Longleftrightarrow\quad u\in  \k\lfloor X\rfloor^\times\text{ and } \sigma(u)>0.$$
\end{lemma}
\begin{proof}
Suppose $u(0)>0$. Then   $u\in\k\lfloor X\rfloor^\times$, and since $u(0)$ is a square in $\k$, $u$ is a square in 
$\k\lfloor X\rfloor$ by Corollary~\ref{cor:roots}, so $\sigma(u)>0$. Conversely, 
suppose~$u\in  \k\lfloor X\rfloor^\times$ and
$\sigma(u)>0$. Then $u(0)\neq 0$; if $u(0)<0$ then $\sigma(u)<0$, by what we just showed applied to~$-u$ in place of $u$, a contradiction. Thus $u(0)>0$.
\end{proof}

\begin{lemma}[Weierstrass Preparation]\label{lem:WP}
Suppose $m\geq 1$.
Let $g\in  \k\lfloor X\rfloor$ be regular in $X_m$ of order~$d$, and set $X':=(X_1,\dots,X_{m-1})$. Then there are a unit $u$ of $\k\lfloor X\rfloor$
and a polynomial
$$w=X_m^d+w_1 X_m^{d-1} + \cdots + w_d\qquad\text{where $w_1,\dots,w_d\in (X')$}$$
such that $g=uw$.
\end{lemma}

\noindent
This is   \cite[Remarks~1.3(9)]{DL}.
Finally, we note:

\begin{lemma}\label{lem:Nagata}
Suppose $\k$ is perfect.
If  $P$ is a prime ideal of $\k\lfloor X\rfloor$, then the ideal~$P\k[[X]]$ of $\k[[X]]$ generated by $P$ is also prime.  
\end{lemma}
\begin{proof}
By \cite[Theorem~2.1]{DL} in combination with Lemmas~\ref{lem:W-Lemma} and~\ref{lem:henselian}, $\k\lfloor X\rfloor$ is a Weierstrass ring in the sense of Nagata \cite[{\S{}45}]{Nagata}, so we can use~\cite[(45.1)]{Nagata}.
\end{proof}

\subsection{Infinitesimal $W$-struc\-tures}
{\it In this subsection we let $K$ be  a  valued field with valuation ring $\mathcal O=\mathcal O_K$ and maximal ideal~$\smallo=\smallo_K$.}\/ 
The following definition is modeled on \cite[axioms C1)--C3) in (2.1)]{DMM}:

\begin{definition} \label{def:W-structure}
An {\bf infinitesimal $W$-struc\-ture}  on
$K$ is  a family
$(\phi_m)_{m\geq 0}$ of ring morphisms
$$\phi_m\colon W_m \to\big\{\text{ring of  functions   $\smallo^m\to \mathcal O$}\big\}$$
such that:
\begin{itemize}
\item[(A1)] $\phi_m(X_i)=\text{$i$-th coordinate
function on $\smallo^m$}$, for each $i=1,\dots,m$;
\item[(A2)] the map $\phi_{m+1}$ extends $\phi_{m}$, if we identify in the 
obvious way functions on~$\smallo^m$ with functions on $\smallo^{m+1}$ that do not depend on the last
coordinate;
\item[(A3)] for all $f\in W_m$, permutations $\sigma$ of $\{1,\dots,m\}$, and $a=(a_1,\dots,a_m)\in\smallo^m$:
$$\phi_m\bigl(f(X_{\sigma(1)},\dots,X_{\sigma(m)})\bigr)(a)=\phi_m(f)\bigl(a_{\sigma(1)},\dots,a_{\sigma(m)}\bigr).$$
\end{itemize}
\end{definition}

\noindent
Note   that  $\smallo^0=\{\text{pt}\}$ is a singleton, and the map which sends     $\alpha\colon \smallo^0\to \mathcal O$ to the element~$\alpha(\text{pt})$ of $\mathcal O$
is an isomorphism  from
the ring of functions $\smallo^0\to \mathcal O$ to the ring~$\mathcal O$. We identify this ring with $\mathcal O$ via the isomorphism $\alpha\mapsto \alpha(\text{pt})$. Then~$\phi_0$ is a ring embedding $\k\to \mathcal O$. Hence the valued field $K$ is necessarily
of equicharacteristic $\operatorname{char}\k$. Moreover,
if $f\in\k[X]\subseteq W_m$, then $\phi_m(f)\colon\smallo^m\to\mathcal O$ is given by $a\mapsto f(a)$.

\begin{lemma}\label{lem:phin(f)(0)}
For an infinitesimal $W$-structure  $(\phi_m)$ on  $K$ and $f\in \k\lfloor X\rfloor$:
$$\phi_m(f) = \phi_0(f(0))+\phi_m(g)\quad\text{ where $g\in (X)$;}$$
in particular $\phi_m(f)(\smallo^m)\subseteq \phi_0(f(0)) + \smallo$ and
$\phi_m(f)(0)=\phi_0(f(0))$.
\end{lemma}

\noindent
This follows immediately from (A1), (A2), and the fact that $f-f(0)\in(X)$.
 
\begin{cor}\label{cor:phin(f)(0)}
Suppose $\mathcal O=K$, and let~$\iota\colon\k\to K$ be an embedding. Then there is a unique infinitesimal  $W$-structure~$(\phi_m)$ on $K$ such that~$\iota=\phi_0$.  
\end{cor}
\begin{proof}
We have $\smallo=\{0\}$, therefore
the family $(\phi_m)$, where $\phi_m(f)$ is the function~$\{0\}^m\to K$ with value $\iota(f(0))\in K$,  for $f\in W_m$, is an infinitesimal
$W$-structure on $K$ with~$\iota=\phi_0$. By the previous lemma this is also the only infinitesimal $W$-structure~$(\phi_m)$ on $K$ with~$\iota=\phi_0$.
\end{proof}

\noindent
Let $K_0$ be a valued subfield of $K$, with valuation ring $\mathcal O_0$ and maximal ideal $\smallo_0$; so~$\mathcal O_0=\mathcal O\cap K_0$ and $\smallo_0=\smallo\cap K_0$.
We say that the infinitesimal $W$-struc\-ture~$(\phi_m)$ on~$K$ {\it restricts}\/ to an infinitesimal $W$-structure on $K_0$
if $\phi_m(f)(\smallo_0^m) \subseteq \mathcal O_0$ for each~$m$ and $f\in W_m$. In this case, the family $(\psi_m)$, where
for each $m$ and $f\in W_m$ we let~$\psi_m(f)$ be the restriction of $\phi_m(f)\colon\smallo^m\to\mathcal O$ to a map $\smallo_0^m\to\mathcal O_0$,
is an infinitesimal $W$-structure on~$K_0$. We call $(\psi_m)$ the {\it restriction}\/ of the $W$-structure on $K$ to $K_0$.

\begin{examples} Here are a few examples of infinitesimal $W$-struc\-tures:
\begin{enumerate}
\item Let $\Gamma$ be an ordered abelian group, written additively, and   let $t^\Gamma$ be a multiplicative copy
of $\Gamma$, with isomorphism $\gamma\mapsto t^\gamma\colon \Gamma\to t^\Gamma$.
 Let~$\k(\!(t^\Gamma)\!)$ be the    field of Hahn series with coefficients in~$\k$
 and monomials in~$t^\Gamma$. Its elements are the formal series 
 $$f=\sum_{\gamma\in\Gamma} f_{\gamma}t^\gamma \qquad (f_\gamma\in\k)$$
 whose support $\operatorname{supp} f:=\{\gamma\in\Gamma:f_\gamma\neq 0\}$ is a well-ordered subset of~$\Gamma$, added and
 multiplied in the natural way. The field $\k(\!(t^\Gamma)\!)$ carries a valuation~$v\colon \k(\!(t^\Gamma)\!)^\times\to\Gamma$ given by
 $$f\mapsto v(f) := \min \supp f,$$
 which we call the $t$-adic valuation on $\k(\!(t^\Gamma)\!)$.
 Its valuation ring is
 $$\mathcal O = \big\{ f\in \k(\!(t^\Gamma)\!): \supp f \geq 0 \big\}.$$
 The kernel of the surjective ring morphism $f\mapsto f_0\colon\mathcal O\to\k$   is the maximal ideal
 $$\smallo = \big\{ f\in \k(\!(t^\Gamma)\!) : \supp f >0 \big\}$$
 of $\mathcal O$, so this morphism induces an isomorphism from the residue field~$\mathcal O/\smallo$ of $\mathcal O$
 onto the coefficient field $\k$.
For $f\in\k[[X]]$ and~$a\in\smallo^m$, the series~$f(a)$ makes sense in $\k(\!(t^\Gamma)\!)$.
For~$f\in W_m$ let $\phi_m(f)$  be the function~$a\mapsto f(a)\colon\smallo^m\to \mathcal O$.
Then the family $(\phi_m)$  is
an infinitesimal $W$-struc\-ture on~$\k(\!(t^\Gamma)\!)$ where~$\phi_0\colon\k\to \mathcal O$ is the natural embedding.
\item Equip the field
$$\k(\!(t^*)\!):=\bigcup_{d\geq 1} \k(\!(t^{1/d})\!)$$ 
of Puiseux series over $\k$ with the $t$-adic  valuation ring   $$ \k[[t^*]]:=\bigcup_{d\geq 1} \k[[t^{1/d}]],$$
with maximal ideal  
$$  \bigcup_{d\geq 1} t^{1/d}\k[[t^{1/d}]].$$
This is a valued subfield of the Hahn field $\k(\!(t^{\Q})\!)$.
The infinitesimal $W$-structure on $\k(\!(t^{\Q})\!)$ described in (1) restricts to one on~$\k(\!(t^*)\!)$. 
\item Suppose $\Gamma$ is archimedean. Then the completion of the $\k$-subalgebra $\k[t^\Gamma]$ of~$\k(\!(t^\Gamma)\!)$ under the $t$-adic valuation is the subfield
$$\hskip5em \operatorname{cl}(\k[t^\Gamma]) := \left\{ f\in \k(\!(t^\Gamma)\!)\, :\, \parbox{18em}{$\operatorname{supp}f$ is finite or $\operatorname{supp}f=\{\gamma_0,\gamma_1,\dots\}$ with $\gamma_n\to\infty$ as $n\to\infty$}\right\}$$
of $\k(\!(t^\Gamma)\!)$. (See \cite[Example~3.2.19]{ADH}.) We have $\k(\!(t^*)\!) \subseteq \operatorname{cl}(\k[t^\Q])$.
The infinitesimal $W$-structure on $\k(\!(t^{\Q})\!)$ from (1) restricts to one on $\operatorname{cl}(\k[t^\Gamma])$.
\item Suppose $\k$ comes equipped with a complete absolute value, and let
 $K:=\k\{\!\{t^*\}\!\}$ be the   subfield of $\k(\!(t^*)\!)$ consisting of the convergent Puiseux series over $\k$, equipped with the 
 valuation ring $\mathcal O_t:=\k[[t^*]]\cap K$, with maximal ideal $\smallo_t$.  Suppose $W=\bigl(\k\{X_1,\dots,X_m\}\bigr)$ is the Weierstrass system of convergent power series rings. Then for $f\in\k\{X_1,\dots,X_m\}$ and $a\in\smallo_t^m$ we have~$f(a)\in K$, so the infinitesimal $W$-structure on~$\k(\!(t^*)\!)$ from (2) above restricts to an
 infinitesimal    $W$-structure on $K$.
\end{enumerate}
\end{examples}

\noindent
{\it In the rest of this subsection we let $(\phi_m)$ be an infinitesimal $W$-struc\-ture on $K$.}\/ (The existence of such a $(\phi_m)$
is a strong assumption on $K$: see Lemma~\ref{lem:K henselian} below.)  We  always
identify~$\k$ with a subfield of~$\mathcal O$ via the embedding $\phi_0\colon\k\to \mathcal O$. Note that then each $\phi_m$ is a $\k$-algebra morphism from $W_m=\k \lfloor X\rfloor$ to the $\k$-algebra of  functions~$\smallo^m\to \mathcal O$.
Thus by  Lemma~\ref{lem:phin(f)(0)}, for $f\in W_m$ we have $\phi_m(f)(0)=f(0)$.
If~$Y=(Y_1,\dots,Y_m)$ is an arbitrary tuple of distinct indeterminates, then  
for~$f=f(Y)\in \k\lfloor Y\rfloor$ we let $\phi_m(f):=\phi_m(f(X))$. In the next lemma and its corollary
$Y=(Y_1,\dots,Y_n)$ is a tuple of distinct indeterminates disjoint from $X=(X_1,\dots,X_m)$.

\begin{lemma}[Substitution]\label{lem:W-structure subst}
Let $f\in \k\lfloor X,Y\rfloor$ and $g=(g_1,\dots,g_n)\in \k\lfloor X\rfloor^n$ with~$g(0)=0$. Then  $f(X,g(X))\in \k\lfloor X\rfloor$ and
$$\phi_m\bigl(f(X,g(X))\bigr)(a)=\phi_{m+n}(f)\bigl(a,\phi_m(g)(a)\bigr)\qquad\text{for all $a\in\smallo^m$.}$$
\end{lemma}
\begin{proof}
Lemma~\ref{lem:DL subst} shows $f(X,g(X))\in \k\lfloor X\rfloor$.
As in the proof of that lemma take~$u_1,\dots,u_n\in \k\lfloor X,Y\rfloor$ and $R\in \k\lfloor X\rfloor$ such that
$$f(X,Y) = u_n(X,Y) (Y_n-g_n(X))+\cdots+ u_1(X,Y) (Y_1-g_1(X)) + R(X).$$
Then $f(X,g(X))=R(X)$ and so $$\phi_{m}\big(f(X,g(X))\big)=\phi_m(R)=\phi_{m+n}(R)$$ by (A2).
Let $a\in\smallo^m$ and set $b:=\phi_m(g)(a)$; by (A2) we also have $b=\phi_{m+n}(g)(a,b)$ and so using (A1),
\[\phi_{m+n}(f)\bigl(a,\phi_m(g)(a)\bigr)=\phi_{m+n}(R)(a,b)=\phi_m\bigl(f(X,g(X))\bigr)(a).\qedhere\] 
\end{proof}

\begin{cor}\label{cor:W-structure subst}
Let $f\in \k\lfloor Y\rfloor$ and $g=(g_1,\dots,g_n)\in \k\lfloor X\rfloor^n$ with~$g(0)=0$. Then
$$\phi_m\bigl(f(g(X))\bigr) =\phi_{n}(f)\circ \phi_m(g).$$
\end{cor}
\begin{proof}
Let $f'$,   $\sigma$, $h$ be as in the proof of Corollary~\ref{cor:DL subst}. Then
for $a\in\smallo^m$ we have
\begin{align*}
\phi_m\bigl(f(g(X))\bigr)(a)	&= \phi_m\bigl(h(X,g(X))\bigr)(a) \\
								&= \phi_{m+n}(h)\big(a,\phi_m(g)(a)\big) \\
								&= \phi_{m+n}(f')\big(\phi_m(g)(a),b\big) \quad\text{for some $b\in\smallo^m$}\\
								&= \phi_n(f)\big(\phi_m(g)(a)\big),
\end{align*}
where we used Lemma~\ref{lem:W-structure subst} for the second, (A3) for the third, and (A2) for the fourth equality.
\end{proof}

\noindent
Let now  $d\in\N$, $d\geq 1$.
By Corollary~\ref{cor:DL subst}, the  $\k$-algebra automorphism $\tau_d$  of $\k[[X]]$ restricts to a $\k$-algebra automorphism of $\k\lfloor X\rfloor$.
Now also consider the automorphism of the $\k$-linear space $\smallo^m$, also  denoted by~$\tau_d$,
given by 
$$(a_1,\dots,a_m)\mapsto\big(a_1+a_m^{d^{m-1}},\dots,a_{m-1}+a_m^d,a_m\big).$$
By (A1) and the previous corollary, we obtain:

\begin{cor}\label{cor:shear}
For each $f\in \k\lfloor X\rfloor$ and   $d\in\N$, $d\geq1$: $\phi_m(\tau_d(f))=\phi_m(f)\circ\tau_d$.
\end{cor}

\noindent 
To illustrate Corollary~\ref{cor:shear} we prove 
that if $\mathcal O\neq K$, then each ring morphism~$\phi_m$ is an embedding; more precisely:
 
\begin{lemma}\label{lem:phim embedding}
Suppose the valuation of $K$ is non-trivial. Then for each non-zero~$f\in \k\lfloor X\rfloor$ and   non-empty open $U\subseteq\smallo^m$ 
there is some $a\in U$ with $\phi_m(f)(a)\neq 0$.
\end{lemma}
\begin{proof}
The case $m=0$ was already observed, so suppose $m\geq 1$.  Take $\tau=\tau_d$ (where $d\in\N$, $d\geq 1$) such that $\tau(f)$ is regular in $X_m$.
Replacing $f$, $U$ by $\tau(f)$, $\tau^{-1}(U)$ and using Corollary~\ref{cor:shear} we arrange that~$f$ is regular in $X_m$.
Lemma~\ref{lem:WP} yields~$u\in  \k\lfloor X\rfloor^\times$ and a  monic~$w\in  \k\lfloor X'\rfloor[X_m]$, where
$X'=(X_1,\dots,X_{m-1})$, such that~$f=uw$. 
Take~$a'\in\smallo^{m-1}$ with $U\cap (\{a'\}\times\smallo)\neq\emptyset$.
Since~$\smallo$ is infinite, we get an $a_m\in\smallo$ such that $a:=(a',a_m)\in U$ and $\phi_m(w)(a)\neq 0$,
and then also~$\phi_m(f)(a)\neq 0$.
\end{proof}

\noindent
From now on we often drop $\phi_m$ from the notation: for $a\in\smallo^m$, $f\in\k\lfloor X\rfloor$  we write~$f(a)$
instead of $\phi_m(f)(a)$. By (A1) we have $g(\smallo^m)\subseteq\smallo$ for   $g\in (X)$ and
 hence
$f(\smallo^m)\subseteq f(0)+\smallo\subseteq\mathcal O$ for   $f\in\k\lfloor X\rfloor$, by Lemma~\ref{lem:phin(f)(0)}.
Hence given~$a\in\smallo^m$, the map~$f\mapsto f(a)\colon \k\lfloor X\rfloor\to\mathcal O$ is a local ring morphism.
The presence of an infinitesimal $W$-structure on a valued field imposes a significant constraint:

\begin{lemma}\label{lem:K henselian}
The valued field $K$ is henselian.
\end{lemma}
\begin{proof}
Let   $a_1,\dots,a_m\in\smallo$ and $P(Y):=1+Y+a_1Y^2+\cdots+a_{m}Y^{m+1}$ where~$Y$ is a single new indeterminate. It suffices to show that $P$ has a zero in $\mathcal O$.
Put
$$Q(X,Y) := 1+Y+X_1Y^2+\cdots+X_mY^{m+1} \in \k\lfloor X\rfloor[Y].$$
Since $\k\lfloor X\rfloor$ is henselian (Lemma~\ref{lem:henselian}), there is some $f\in \k\lfloor X\rfloor$ with $Q(X,f(X))=0$. Then $P(b)=Q(a,b)=0$
for $a:=(a_1,\dots,a_m)$ and $b:=f(a)\in\mathcal O$.
\end{proof}

\noindent
This way we reprove the well-known fact that the valued field 
$\k(\!(t^\Gamma)\!)$ and its valued subfield
$\operatorname{cl}(\k[t^\Gamma])$ (when $\Gamma$ is archimedean) are henselian. Similarly, so are the valued fields $\k(\!(t^*)\!)$ and~$\k\{\!\{t^*\}\!\}$ (when $\k$ is equipped with a complete absolute value).

\begin{lemma}\label{lem:W-restrict}
Let $L$ be a valued field extension of $K$ equipped with an infinitesimal $W$-structure which restricts to that of $K$, and let $K_0$ be a valued
subfield of $L$ containing $K$ which is algebraic over $K$. Then the  infinitesimal $W$-structure of $L$ restricts to an
infinitesimal $W$-structure on $K_0$.
\end{lemma}
\begin{proof}
By induction on $n$ we show:
if $f\in \k\lfloor X,Y\rfloor$ where $Y=(Y_1,\dots,Y_n)$ is a tuple of distinct indeterminates disjoint from $X$ and $b\in\smallo^m$, $c\in\smallo_{K_0}^n$, then
 $f(b,c)\preceq 1$.
For $n=0$ this is clear, so suppose $n\geq 1$.  
Let $$P=Z^d+a_{1}Z^{d-1}+\cdots+a_d\in K[Z]\qquad (a_1,\dots,a_d\in K)$$ be the minimum polynomial of $c_n$ over $K$.
Then $a:=(a_1,\dots,a_d)\in\smallo^d$ since $K$ is henselian and~$c_n\prec 1$. 
(See, e.g., \cite[1.3.12, 3.3.11, 3.3.15]{ADH}.)
We have $g(a,c_n)=0$ where
$$g(U,Y_n) := Y_n^d+U_{1}Y_n^{d-1}+\cdots+U_d \in \Z[U,Y_n]\subseteq \k\lfloor U,Y\rfloor \quad (U=(U_1,\dots,U_d)).$$
Let $Y'=(Y_1,\dots,Y_{n-1})$.
Weierstrass Division in $\k\lfloor U, X, Y\rfloor$ yields $q\in \k\lfloor U,X,Y\rfloor$ and $r\in \k\lfloor U,X,Y'\rfloor[Y_n]$ of degree~$<d$
with $f=qg+r$. Then $f(b,c)=r(a,b,c)\preceq 1$ by the inductive hypothesis applied to the coefficients of $r$ and~$(a,b)$,~$c'$ in place of~$b$,~$c$, respectively.
\end{proof}

\noindent
Let $\k^{\operatorname{a}}$ be an algebraic closure of $\k$; then  $\k^{\operatorname{a}}(\!(t^\Q)\!)$ is algebraically closed.
By the previous lemma,  
the infinitesimal $W$-structure of $\k^{\operatorname{a}}(\!(t^\Q)\!)$ restricts to an infinitesimal $W$-structure on
the algebraic closure of the field $\k(\!(t)\!)$ of Laurent series over~$\k$ inside $\k^{\operatorname{a}}(\!(t^\Q)\!)$.

\subsection{Infinitesimal $W$-struc\-tures as model-theoretic structures}\label{sec:inf W-structures as structures}
 Let $\mathcal L$ be an expansion of the language~$\mathcal L_{\preceq}= \{0,1,{-},{+},{\,\cdot\,},{\preceq}\}$ of rings augmented by
 a binary relation symbol~$\preceq$. (See Section~\ref{sec:val th}.)
We then let $\mathcal L_W$ be the expansion of~$\mathcal L$ by  a new $m$-ary function symbol, for each~$m$ and $f\in W_m$, also denoted by~$f$.
Given an $\mathcal L$-theory $T$ containing the $\mathcal L_{\preceq}$-theory of valued fields we 
let~$T_W$ be the $\mathcal L_W$-theory whose models are the $\mathcal L_W$-structures~$\boldsymbol K=(K;\dots)$ 
whose   $\mathcal L$-reduct is a model of~$T$, 
each~$f\in W_m$ is interpreted by a function~${f^{\boldsymbol K}\colon K^m\to K}$ which is identically zero on the complement of~$\smallo^m$ and
satisfies $f^{\boldsymbol K}(\smallo^m)\subseteq\mathcal O$, and such that the family~$(\phi_m)$ of maps given by~$\phi_m(f)=f^{\boldsymbol K}|{\smallo^m}$ is an infinitesimal $W$-structure on $K$. (Note: the underlying valued field of each model of $T_W$ is henselian of equicharacteristic~$\operatorname{char}\k$.)

\begin{example}
Let $q  \in \Q^>$. We then have an automorphism $a(t)\mapsto a(t^q)$ of the  valued field $K=\k(\!(t^{\Q})\!)$ over $\k$
(with inverse $b(t)\mapsto b(t^{1/q})$). 
For $\mathcal L=\mathcal L_{\preceq}$,
this is also an automorphism of the $\mathcal L_W$-structure $\boldsymbol{K}=(K;\dots)$
 with $\mathcal L$-reduct $K$ described above.
This  automorphism restricts to an automorphism
of the substructure $\big(\k(\!(t^*)\!);\dots\big)$ of $\boldsymbol{K}$.
Given $a\in \k(\!(t^*)\!)$ there is some $d\geq 1$ such that $a(t^d)\in \k(\!(t)\!)$,
and then 
$$a\preceq 1\  \Longleftrightarrow\  a(t^d)\in \k[[t]],\qquad a\prec 1\ \Longleftrightarrow\  a(t^d)\in t\k[[t]].$$
Suppose $\k$ comes equipped with a complete absolute value and $W$ is the Weierstrass system of convergent power series over $\k$. Then the automorphism
$a(t)\mapsto a(t^q)$ of the $\mathcal L_W$-structure~$\big(\k(\!(t^*)\!);\dots\big)$ further restricts to an automorphism
of its substructure $\big(\k\{\!\{t^*\}\!\};\dots\big)$.
\end{example}

\subsection{Zero sets and vanishing ideals}
Let $(\phi_m)$ be an infinitesimal $W$-structure on a valued field $K$, and let $I$ be an ideal  of $W_m=\k\lfloor X\rfloor$. We let
$$\operatorname{Z}_K(I) := \big\{ a\in\smallo^m: \text{$f(a)=0$ for all $f\in I$}\big\}$$
be the {\it zero set}\/ of $I$ in $K$.
One verifies easily that if $J$ is another ideal of $\k\lfloor X\rfloor$, then~$\operatorname{Z}_K(I)\subseteq \operatorname{Z}_K(J)$ if
$I\supseteq J$, and
$$\operatorname{Z}_K(IJ)=\operatorname{Z}_K(I\cap J)=\operatorname{Z}_K(I)\cup \operatorname{Z}_K(J), \qquad
\operatorname{Z}_K(I+J)=\operatorname{Z}_K(I)\cap \operatorname{Z}_K(J).$$
Let also $S$ be a subset of $\smallo^m$. We let
$$\operatorname{I}(S) := \big\{ f\in \k\lfloor X\rfloor: \text{$f(a)=0$ for all $a\in S$} \big\}.$$
Then $\operatorname{I}(S)$ is an ideal of $\k\lfloor X\rfloor$, called the {\it vanishing ideal}\/ of $S$. If also $T\subseteq\smallo^n$, then~$\operatorname{I}(S)\subseteq  \operatorname{I}(T)$ if $S\supseteq T$, and~$\operatorname{I}(S\cup T) = \operatorname{I}(S)\cap \operatorname{I}(T)$.
 Clearly the ideal
$\operatorname{I}(S)$ is radical, so
$\operatorname{I}\!\big(\!\operatorname{Z}_K(I)\big)\supseteq\sqrt{I}$. Also note~$\operatorname{Z}_K(I) = \operatorname{Z}_K(\sqrt{I})$.

\subsection{Germs of analytic functions}
Suppose $\k$ is equipped with a complete absolute value.
Recall that analytic  functions $f\colon U\to\k$ and $g\colon V\to\k$, where~$U,V\subseteq\k^m$ are open neighborhoods of $0$, are said to have the same {\it germ}\/ at $0$ if there is an open neighborhood~$W$ of $0$ contained in $U\cap V$ such that $f(a)=g(a)$ for all~$a\in W$.
The relation of having the same germ at~$0$ is an equivalence relation on the collection of all analytic functions $f\colon U\to\k$ 
where~$U\subseteq\k^m$ is an open neighborhood of~$0$; we denote the equivalence class of~$f$ by~$[f]$. The     equivalence classes of this equivalence relation
form a commutative ring such that $[f]+[g]=[f+g]$ and~$[f]\cdot [g]=[f\cdot g]$ for all analytic functions~$f,g\colon U\to\k$ defined on an open neighborhood $U\subseteq\k^m$ of~$0$. The map sending the germ of~$f$ to its Taylor series at $0$ is a
ring isomorphism from this ring onto $\k\{X\}=\k\{X_1,\dots,X_m\}$, via which we identify these two rings.

\section{A Proof of R\"uckert's Nullstellensatz}\label{sec:Rueckert}

\noindent
In this section $\mathcal L=\mathcal L_{\preceq}$, and we let $\operatorname{ACVF}$ be the $\mathcal L$-theory whose models are the  algebraically closed non-trivially valued fields. We recall a fundamental theorem of the model theory of algebraically closed valued fields.
(See, for example, \cite[Theorem~3.6.1]{ADH} or \cite[Theorem~4.4.2]{PrestelDelzell} for a proof.)

\begin{theorem}[{A.~Robinson~\cite{RobinsonComplete}}]\label{thm:Robinson}
$\operatorname{ACVF}$ has \textup{QE}.
\end{theorem}

\noindent
In particular, the theory $\operatorname{ACVF}$ is substructure complete, that is, if $K,L\models\operatorname{ACVF}$ and $A$ is a substructure of both $K$ and $L$, then $K_A\equiv L_A$. (Here $K_A$ is the natural expansion of $K$ to an $\mathcal L_A$-structure
where $\mathcal L_A$ extends the language $\mathcal L$ by new constant symbols for the elements of $A$, and likewise with $L_A$ in place of $K_A$: see
Section~\ref{sec:mtvf}.)
Moreover,
 $\operatorname{ACVF}$ is the model completion of the $\mathcal L$-theory
of pairs $(R,{\preceq})$ where $R$ is an integral domain and $\preceq$ is a dominance relation on $R$  \cite[Corollary~3.6.4]{ADH}.

\medskip
\noindent
In the following we let   $W$ be a Weierstrass system over a field~$\k$. 
We also let~$X=(X_1,\dots,X_m)$ be a tuple of distinct indeterminates and~$y=(y_1,\dots,y_n)$ be a tuple of distinct $\mathcal L$-variables.
Let  $K\models \operatorname{ACVF}_W$ and~$\mathcal O=\mathcal O_K$, $\smallo=\smallo_K$.
Recall that for~$a\in\smallo^m$, the map
 $f\mapsto f(a)$ is a local ring morphism~${\k\lfloor X\rfloor\to \mathcal O}$, and that we view~$\k$ as a subfield of $\mathcal O$.
 We  now use Theorem~\ref{thm:Robinson} to prove the following  specialization property:  
 
\begin{prop} \label{prop:Rueckert spec}
Let $\varphi(y)$ be an  $\mathcal L$-formula, $f=(f_1,\dots,f_n)\in \k\lfloor X\rfloor^n$,
$\Omega\models\operatorname{ACVF}$, and $\sigma \colon \k\lfloor X\rfloor\to \mathcal O_\Omega$ be a 
 local ring morphism  such that $\Omega\models\varphi(\sigma(f))$. Then
$K\models\varphi(f(a))$ for some $a\in\smallo^m$.
\end{prop}
\begin{proof}
We proceed by induction on the length $m$ of   $X$. 
Note that if we equip $\k$ with the trivial dominance relation, then the restriction of $\sigma$ to an embedding $\k\to\Omega$
as well as the natural inclusion $\k\to K$ are $\mathcal L$-embeddings. Hence
the case $m=0$ follows from substructure completeness of $\operatorname{ACVF}$. Suppose~$m\geq 1$. By
Theorem~\ref{thm:Robinson} we   arrange that $\varphi$ is quantifier-free, and we may then assume that $\varphi$ has the form
$$ \bigwedge_{i\in I} P_i(y) = 0 \wedge Q(y)\neq 0 \wedge \bigwedge_{j\in J} R_j(y)\, \square_j\, S_j(y)$$
where $I$, $J$ are finite index sets, $P_i,Q,R_j,S_j\in\Z[Y_1,\dots,Y_n]$, and each $\square_j$ is one of the symbols $\preceq$ or $\prec$.
Now note that for each~$P\in \Z[Y_1,\dots,Y_n]$ we have  $P(f)\in\k\lfloor X\rfloor$ as well as $P(\sigma(f))=\sigma(P(f))$ and $P(f(a))=P(f)(a)$ for every $a\in\smallo^m$.
Hence by suitably modifying  $\varphi$, $f$ (which may include increasing $n$) we can arrange
that the polynomials $P_i$, $Q$, $R_j$, $S_j$ are just distinct elements of~$\{Y_1,\dots,Y_n\}$.
For   example, for $n=4$, our formula
 $\varphi$ may have the form
$$y_1 = 0 \ \wedge\ y_2  \neq 0  \ \wedge\ y_3 \,\square\, y_4\quad\text{where $\square$ is either $\preceq$ or $\prec$.}$$
(The general case is only notationally more involved.) We may also clearly arrange that the formal power series $f_1,\dots,f_n$ are all non-zero. 
Take $d\in\N$, $d\geq 1$,  such that with $\tau:=\tau_d$, each~$\tau(f_j)$ is regular in $X_m$.
For each $g\in \k\lfloor X\rfloor$ and~$a\in\smallo^m$  we have $\tau(g)(a)=g(\tau(a))$.
(Corollary~\ref{cor:shear}.)
Hence
replacing each~$f_j$ by~$\tau(f_j)$  and~$\sigma$ by $\sigma\circ\tau^{-1}$ we arrange that
$f_j$ is regular in $X_m$, for~$j=1,\dots,n$.
Weierstrass Preparation in~$\k\lfloor X\rfloor$  (Lemma~\ref{lem:WP}) yields a $1$-unit~$u_j\in\k\lfloor X\rfloor$  and a   polynomial~$w_j\in\mathcal \k\lfloor X'\rfloor[X_m]$, where $X'=(X_1,\dots,X_{m-1})$, such that~$f_j=u_j w_j$, for~$j=1,\dots,n$. 
We have $\sigma(u_j)\sim 1$ since $\sigma$ is local.
Hence we may replace each~$f_j$ by $w_j$ to arrange that $f_j=w_j$ is a polynomial.
Take $e\in\N$ and $w_{jk}\in  \k\lfloor X'\rfloor$~($j=1,\dots,n$, $k=0,\dots,e$) such that
$$w_j = w_{j0} + w_{j1} X_m +\cdots + w_{je}X_m^e,$$
and set $w:=(w_{jk})$.
Let  $u:=(u_{jk})$ be a tuple of distinct new variables and $v$ be a new variable; then~$\sigma(w_j)=t_j(\sigma(w),\sigma(X_m))$ for the $\mathcal L_{\operatorname{R}}$-term
$$t_j(u,v) := u_{j0}+ u_{j1} v +\cdots + u_{je}v^e.$$
Consider
 the existential $\mathcal L$-formula
$$\psi(u) := \exists v \big(v\prec 1 \wedge \varphi\big(t_1(u,v),\dots,t_n(u,v)\big)\big).$$
Then
\begin{align*}
\Omega\models \varphi(\sigma(f)) &\quad\,\Longrightarrow\quad \Omega\models \psi(\sigma(w)) \\
& \quad\, \Longrightarrow\quad K \models\psi(w(b)) \text{ for some $b\in\smallo^{m-1}$} \\
& \quad\Longleftrightarrow\quad K \models \varphi(f(a)) \text{ for some $a\in\smallo^{m}$,}
\end{align*}
where for the second implication we used the inductive hypothesis applied to $\psi$, $w$, and the restriction of $\sigma$ to~$\k\lfloor X'\rfloor$ in place of $\varphi$, $f$, $\sigma$, respectively.
\end{proof}

\noindent
The following corollary (not used later) captures the model-theoretic essence of Proposition~\ref{prop:Rueckert spec}:

\begin{cor}\label{cor:exists-complete}
Let $\theta$ be an existential  $\mathcal L_W$-sentence. If $K\models  \theta$, then $L\models \theta$ for each~$L\models \operatorname{ACVF}_W$.  
\end{cor}
\begin{proof}
Take a tuple $x=(x_1,\dots,x_m)$ of distinct new variables, for some $m$, and a boolean combination $\psi(x)$ of 
unnested atomic $\mathcal L_W$-formulas 
with $\models\theta \leftrightarrow \exists x \psi$; here   $\exists x$ abbreviates  $\exists x_1\cdots\exists x_m$.
(See \cite[Lemmas~B.4.5 and~B.5.5]{ADH}.)
Then~$\psi(x)=\varphi(x,f(x))$ for some quantifier-free $\mathcal L$-formula
$\varphi(x,y)$ and~$f\in\k\lfloor X\rfloor^n$. 
Now 
$$\operatorname{ACVF}_W\models \theta \leftrightarrow \exists x\big( (1\preceq x_1\vee\cdots\vee 1\preceq x_m) \wedge\varphi(x,0)\big) \vee \exists x \big(x_1\prec 1\wedge\cdots\wedge x_m\prec 1\wedge 
\psi\big),$$
hence the corollary follows from Theorem~\ref{thm:Robinson} and Proposition~\ref{prop:Rueckert spec}.
\end{proof}

\begin{cor}\label{cor:spec}
Let $\varphi(y)$ be an $\mathcal L_{\operatorname{R}}$-formula, 
     $f\in \k\lfloor X\rfloor^n$, and let~$\sigma \colon \k\lfloor X\rfloor\to \Omega$ be a ring morphism  to an algebraically
  closed field $\Omega$ such that~$\Omega\models\varphi(\sigma(f))$. Then~$K\models\varphi(f(a))$ for some $a\in\smallo^m$.
\end{cor}
\begin{proof}
Suppose first that $\ker\sigma=(X)$; then $\sigma(g)=\sigma(g(0))$ for every $g\in  \k\lfloor X\rfloor$, and the corollary follows
from substructure completeness of the $\mathcal L_{\operatorname{R}}$-theory   of algebraically closed fields by taking $a:=0\in\smallo^m$. 
Now suppose $\ker\sigma\neq(X)$.
The image~$A$ of $\sigma$ is a local subring   of $\Omega$ (cf.~Lemma~\ref{lem:W-Lemma}) but not a field. Take a valuation ring of $\Omega$ lying over $A$  and equip $\Omega$ with the associated dominance relation. Then~$\Omega\models\operatorname{ACVF}$ and   $\sigma \colon \k\lfloor X\rfloor\to \mathcal O_\Omega$ is a local ring morphism, so the previous proposition applies. 
\end{proof}

\noindent
Corollary~\ref{cor:spec} now implies a    general version of R\"uckert's Nullstellensatz:

\begin{theorem}[R\"uckert's Nullstellensatz]\label{thm:Rueckert}
Let $I$ be an ideal of~$\k\lfloor X\rfloor$. Then
$$\operatorname{I}\!\big(\!\operatorname{Z}_K(I)\big)=\sqrt{I}.$$
\end{theorem}
\begin{proof}
If we have some $a\in\operatorname{Z}_K(I)$, then the kernel of the
ring morphism $f\mapsto f(a)\colon\k\lfloor X\rfloor\to K$ is a  prime ideal of~$\k\lfloor X\rfloor$ containing $I$. Hence we may assume that~$I\neq \k\lfloor X\rfloor$.
Recall: the ring~$\k\lfloor X\rfloor$ is noetherian (Lemma~\ref{lem:W-Lemma}). This yields prime ideals~$P_1,\dots,P_k$~($k\geq 1$) of $\k\lfloor X\rfloor$
such that
$\sqrt{I}=P_1\cap\cdots\cap P_k$.
Then
$$\operatorname{I}\!\big(\!\operatorname{Z}_K(I)\big)=\operatorname{I}\!\big(\!\operatorname{Z}_K(\sqrt{I})\big)=
\operatorname{I}\!\big(\!\operatorname{Z}_K(P_1)\big)\cap\cdots\cap\operatorname{I}\!\big(\!\operatorname{Z}_K(P_k)\big),$$
so it suffices to treat the case where $I$ is  prime. So from now on assume~$I$ is a prime ideal of $\k\lfloor X\rfloor$;
we need to show that then $\operatorname{I}\!\big(\!\operatorname{Z}_K(I)\big)=I$. 
Let $\Omega$ be an algebraic closure of the fraction field   of the integral domain $A:=\k\lfloor X\rfloor/I$, and let~$\sigma\colon \k\lfloor X\rfloor\to \Omega$ be the
composition of the residue morphism $f\mapsto f+I\colon\k\lfloor X\rfloor\to A$ with the natural inclusion~$A\subseteq  \Omega$.
Let $f\in \k\lfloor X\rfloor\setminus I$ and let   $g_1,\dots,g_n$ generate~$I$.
Then $\Omega\models\varphi\big(\sigma(f),\sigma(g_1),\dots,\sigma(g_n)\big)$ where $\varphi(y_0,\dots,y_{n})$ is the quantifier-free $\mathcal L_{\operatorname{R}}$-formula
$y_0\neq 0 \wedge y_{1} = \cdots = y_n=0$. Corollary~\ref{cor:spec}   yields an~$a\in\smallo^m$ with~$f(a)\neq 0$ and $g_1(a)=\cdots=g_n(a)=0$, so
$f\notin \operatorname{I}\!\big(\!\operatorname{Z}_K(I)\big)$. Thus~$\operatorname{I}\!\big(\!\operatorname{Z}_K(I)\big)\subseteq I$ as required.
\end{proof}

\noindent
Let now $\k^{\operatorname{a}}$ be an algebraic closure of $\k$ and $K=\k(\!(t)\!)^{\operatorname{a}}$ be the algebraic closure
of~$\k(\!(t)\!)$ inside $\k^{\operatorname{a}}(\!(t^\Q)\!)$, equipped with the restriction of the dominance relation of~$\k^{\operatorname{a}}(\!(t^\Q)\!)$. Expand $K$ to an $\mathcal L_W$-structure
where   $W$ is the Weierstrass system of formal power series over $\k$. 
(See the remarks after Lemma~\ref{lem:W-restrict}.)
Then $K\models\operatorname{ACVF}_W$, hence by  Theorem~\ref{thm:Rueckert} we obtain the
``formal'' version of R\"uckert's Nullstellensatz:

\begin{cor} \label{cor:Rueckert, 0}
Let $f, g_1,\dots,g_n\in \k[[X]]$ with~${f(a)=0}$ for all $a\in \smallo^m$ such that~$g_1(a)=\cdots=g_n(a)=0$. Then there are~$h_1,\dots,h_n\in \k[[X]]$ and some~$k\geq 1$ such that $f^k=g_1h_1+\cdots+g_nh_n$. 
\end{cor}

\noindent
{\it 
In the rest of this subsection we assume   $\operatorname{char}\k=0$.}\/ 
Then   $K=\k^{\operatorname{a}}(\!(t^*)\!)$   (see, e.g., \cite[Example~3.3.23]{ADH}).
Hence using the automorphisms~$a(t)\mapsto a(t^d)$ of the $\mathcal L_W$-structure~$K$ (where $d\geq 1$), from Corollary~\ref{cor:Rueckert, 0} we obtain:

\begin{cor} \label{cor:Rueckert, 1}
Let $f, g_1,\dots,g_n\in \k[[X]]$ with~${f(a)=0}$ for all $a\in t\k^{\operatorname{a}}[[t]]^m$ such that $g_1(a)=\cdots=g_n(a)=0$. Then there are power series~$h_1,\dots,h_n\in \k[[X]]$ and some $k\geq 1$ such that $f^k=g_1h_1+\cdots+g_nh_n$. 
\end{cor}
\begin{proof}
Let $a\in\smallo^m$ be such that~$g_1(a)=\cdots=g_n(a)=0$. Taking $d\geq 1$ such that~$b:=a(t^d)$ lies in $t\k^{\operatorname{a}}[[t]]^m$, we
then also have~${g_1(b)=\cdots=g_n(b)=0}$, so~${f(b)=0}$ by the hypothesis of the corollary,   thus $f(a)=0$.
Hence the corollary follows from Co\-rol\-lary~\ref{cor:Rueckert, 0}.
\end{proof}

\noindent
For $m=1$ the previous corollary holds without assuming $\operatorname{char}\k=0$, since~${\smallo\setminus\{0\}}$ is an orbit of the group of continuous automorphisms
of $\k^{\operatorname{a}}(\!(t^\Q)\!)$ over $\k^{\operatorname{a}}$ 
 \cite[Theorem~3.4]{KP}.
But we do not know whether this remains true for~$m>1$.
(See \cite{Kedlaya01,Kedlaya17} for an explicit description of the subfield~$\k(\!(t)\!)^{\operatorname{a}}$ of $\k^{\operatorname{a}}(\!(t^\Q)\!)$.)

\begin{remark}
If in the context of Corollary~\ref{cor:Rueckert, 1} we have $f,g_1,\dots,g_n\in\k[[X]]^{\operatorname{a}}$, then we can take $h_1,\dots,h_n\in \k[[X]]^{\operatorname{a}}$. 
This follows either from Theorem~\ref{thm:Rueckert} applied to the Weierstrass system of   algebraic formal power series, or the previous corollary and   Corollary~\ref{cor:fflat}.
Similarly
with $\k[[X]]^{\operatorname{a}}$ replaced by $\k[[X]]^{\operatorname{da}}$. 
\end{remark}

\noindent
{\it Now we also suppose that $\k$  is algebraically closed and comes equipped with a complete absolute value.}\/ (For example, $\k=\C$ or $\k=\C_p$ with their usual absolute values.) Then the field~$\k\{\!\{t^*\}\!\}$ of convergent Puiseux series over $\k$ is algebraically closed.
Arguing as in the proof of Corollary~\ref{cor:Rueckert, 1}, with
$\k\{\!\{t^*\}\!\}$ in place of $\k(\!(t^*)\!)$ and $W=$~the Weierstrass system of
convergent power series over $\k$, we obtain the classical version of R\"uckert's Nullstellensatz for convergent power series:

\begin{cor} \label{cor:Rueckert, 2}
Let 
$f,g_1,\dots,g_n\in \k\{X\}$ and suppose  that 
for all $a\in \k\{t\}^m$ with~${a(0)=0}$ we have $f(a)=0$ whenever~${g_1(a)=\cdots=g_n(a)=0}$. Then there are~$h_1,\dots,h_n\in \k\{X\}$ and some~$k\geq 1$ such that $f^k=g_1h_1+\cdots+g_nh_n$.
\end{cor}
 
\noindent
As usual this yields a version for germs of analytic functions.

\begin{cor}\label{cor:Rueckert fns}
Let $f,g_1,\dots,g_n\colon U\to\k$ be analytic functions, where $U$ is an open neighborhood of $0$ in $\k^m$, such that
for all $a\in U$,
$$g_1(a)=\cdots=g_n(a)=0 \quad\Longrightarrow\quad f(a)=0.$$ Then there are analytic functions $h_1,\dots,h_n\colon V\to\k$,
for some open neighborhood~$V\subseteq U$ of $0$ in $\k^m$, and some $k\geq 1$ such that $$f^k=g_1h_1+\cdots+g_nh_n\qquad\text{ on $V$.}$$
\end{cor}

\begin{proof} 
Let $f,g_1,\dots,g_n\in \k\{X\}$ also denote the germs of the corresponding analytic 
functions $U\to\k$. Then
for all $a\in \k\{t\}^m$ with~${a(0)=0}$ we have 
$g_1(a)=\cdots=g_n(a)=0\Rightarrow f(a)=0$.
(By the hypothesis and the continuity of each representative of $a$ at $0$.)
Now Corollary~\ref{cor:Rueckert, 2}  yields the existence of $h_1,\dots,h_n$ and $k$ as claimed.
\end{proof}

\section{Risler's   Nullstellensatz}\label{sec:Risler}

\noindent
Let
$\mathcal L=\mathcal L_{\operatorname{OR},\preceq}=\{0,1,{-},{+},{\,\cdot\,},{\leq},{\preceq}\}$ be the  language $\mathcal L_{\operatorname{OR}}=\mathcal L_{\operatorname{R}}\cup\{\leq\}$ of ordered rings expanded by a new binary relation symbol  $\preceq$.
Let $\operatorname{RCVF}$ be the $\mathcal L$-theory whose models are the $\mathcal L$-structures $(K,{\preceq})$ where $K$ is a
real closed ordered field and~$\preceq$ is a non-trivial convex dominance relation on $K$.
Here, a dominance relation~$\preceq$ on an ordered integral domain $R$ is said to be {\it convex}\/ if
$0\leq r\leq s\Rightarrow r\preceq s$, for all~$r,s\in R$.
The analogue of Theorem~\ref{thm:Robinson} in this context is the following fact.
(For a self-contained proof  see \cite[Theorem~3.6.6]{ADH} or \cite[Theorem~4.5.1]{PrestelDelzell}.)

\begin{theorem}[{Cherlin-Dickmann \cite{CD}}] \label{thm:CD} $\operatorname{RCVF}$ has \textup{QE.}
\end{theorem}

\noindent
Indeed, $\operatorname{RCVF}$ is the model completion of the $\mathcal L$-theory of ordered integral domains equipped with
a convex dominance relation \cite[Corollary~3.6.7]{ADH}.
Let    $W$ be a Weierstrass system over a euclidean field~$\k$  and   $K\models \operatorname{RCVF}_W$.
Also let~$X=(X_1,\dots,X_m)$ be a tuple of distinct indeterminates and $y=(y_1,\dots,y_n)$ be a tuple of distinct $\mathcal L$-variables.  Just like Theorem~\ref{thm:Robinson} gave rise to 
Proposition~\ref{prop:Rueckert spec}, Theorem~\ref{thm:CD} implies:

\begin{prop} \label{prop:Risler spec}
Let $\varphi(y)$ be an  $\mathcal L$-formula, $f=(f_1,\dots,f_n)\in \k\lfloor X\rfloor^n$, $\Omega\models\operatorname{RCVF}$,
and~$\sigma \colon \k\lfloor X\rfloor\to \mathcal O_\Omega$ 
be a local ring morphism 
such that $\Omega\models\varphi(\sigma(f))$. Then
$K\models\varphi(f(a))$ for some $a\in\smallo^m$.
\end{prop}

\begin{proof}
We proceed by induction on $m$ as in the proof of Proposition~\ref{prop:Rueckert spec}. 
Equipping~$\k$ with its unique ordering and with the trivial dominance relation,   the restriction of $\sigma$ to an embedding $\k\to\Omega$
as well as the natural inclusion $\k\to K$ are   $\mathcal L$-embeddings. 
Together with substructure completeness of $\operatorname{RCVF}$ this takes care of
the case~$m=0$. So suppose~$m\geq 1$. Using
Theorem~\ref{thm:CD}  we   arrange    $\varphi$ to have the form
$$P(y) = 0 \wedge \bigwedge_{i\in I} Q_i(y) > 0 \wedge \bigwedge_{j\in J} R_j(y)\, \square_j\, S_j(y)$$
where $I$, $J$ are finite index sets, $P,Q_i,R_j,S_j\in\Z[Y_1,\dots,Y_n]$, and each $\square_j$ is~$\preceq$ or~$\prec$.
As in  the proof of Proposition~\ref{prop:Rueckert spec} 
we also arrange
that the $P$, $Q_i$, $R_j$, $S_j$ are   distinct elements of~$\{Y_1,\dots,Y_n\}$ and 
$f_1,\dots,f_n$ are all regular in $X_m$. Take some $1$-unit $u_j\in\k\lfloor X\rfloor$ and   $w_j\in \k\lfloor X'\rfloor[X_m]$, where $X'=(X_1,\dots,X_{m-1})$, such that~$f_j=u_jw_j$. Then $\sigma(u_j)\sim 1$ since $\sigma$ is local, hence $\sigma(u_j)>0$. Thus we may arrange that each $f_j=w_j$ is a polynomial and conclude the argument as in the proof of 
Proposition~\ref{prop:Rueckert spec}.
\end{proof}

\noindent
From Proposition~\ref{prop:Risler spec} we obtain
 the completeness of the existential part of $\operatorname{RCVR}_W$ as in the proof of Corollary~\ref{cor:exists-complete}:

\begin{cor}\label{cor:exists-complete, real}
For each existential  $\mathcal L_W$-sentence $\theta$, we either   have
$\operatorname{RCVR}_W\models\theta$ or~$\operatorname{RCVR}_W\models\neg\theta$.
\end{cor}

\noindent
Next, an analogue of Corollary~\ref{cor:spec}; for this, we first note:

\begin{lemma}\label{lem:indets infinitesimal}
Let $\Omega$ be an ordered field and $\sigma\colon  \k\lfloor X\rfloor\to \Omega$ be a ring morphism.
Equip $\Omega$ with the dominance relation associated to the convex hull of $\sigma(\k)$ in $\Omega$.
Then for $g\in  \k\lfloor X\rfloor$  we have $\sigma\big(g- g(0)\big) \prec 1$.
\end{lemma}
\begin{proof}
Let $g\in  \k\lfloor X\rfloor$; replacing $g$ by $g-g(0)$ we arrange $g(0)=0$.
Let $\varepsilon\in\k^>$ and    consider $u=\varepsilon-g\in \k\lfloor X\rfloor$.
Then $\sigma(u)>0$ by Lemma~\ref{lem:u>0} and thus $\sigma(g)<\sigma(\varepsilon)$. Similarly one shows $-\sigma(\varepsilon)<\sigma(g)$.
Thus $\sigma(g)\prec 1$.
\end{proof}

\begin{cor}\label{cor:Risler spec}
Let $\varphi(y)$ be an $\mathcal L_{\operatorname{OR}}$-formula,  
     $f\in \k\lfloor X\rfloor^n$, and~$\sigma \colon \k\lfloor X\rfloor\to \Omega$ be a ring morphism  to a real
  closed field $\Omega$ such that~$\Omega\models\varphi(\sigma(f))$. Then
$K\models\varphi(f(a))$ for some $a\in\smallo^m$.
\end{cor}
\begin{proof}
Let $\mathcal O_\Omega$ be the convex hull of $\sigma(\k)$ in $\Omega$.
 By Lemma~\ref{lem:indets infinitesimal}, $\sigma$ is a local ring morphism $\k\lfloor X\rfloor\to\mathcal O_\Omega$.
If $\Omega = \mathcal O_\Omega$ then 
model completeness of the $\mathcal L_{\operatorname{OR}}$-theory  of real closed ordered fields
allows us to replace  $\Omega$ 
 by a real closed ordered field extension of $\Omega$ (such as $\Omega(\!(t^\Q)\!)$) to arrange~$\Omega \neq \mathcal O_\Omega$;
 thus  $\Omega$ equipped with its dominance relation associated to~$\mathcal O_\Omega$  is a model of $\operatorname{RCVF}$.
Now the corollary 
follows from Proposition~\ref{prop:Risler spec}.
\end{proof}

\noindent
We now quickly deduce  Risler's analytic version of Hilbert's 17th Problem:

\begin{cor}\label{cor:Risler H17}
Let $f\in \k\lfloor X\rfloor$. Then $f(a)\geq 0$ for all $a\in\smallo^m$ iff
there are a non-zero $g\in\k\lfloor X\rfloor$ and $h_1,\dots,h_k\in\k\lfloor X\rfloor$ \textup{(}for some $k$\textup{)} such that
$fg^2 =  h_1^2+\cdots+h_k^2$.
\end{cor}
\begin{proof}
Suppose there are no such $g$ and $h_j$. This yields an ordering $\leq$ on $F:=\Frac(\k\lfloor X\rfloor)$ such that $f<0$.
(See, e.g.,~\cite[Corollary~1.1.11]{BCR}.)
Let $\Omega$ be the real closure of the ordered field $(F,{\leq})$ and $\sigma\colon \k\lfloor X\rfloor\to\Omega$ be the natural inclusion, and apply Corollary~\ref{cor:Risler spec} to the $\mathcal L_{\operatorname{OR}}$-formula $y<0$ (where $y$ is a single variable).
This shows the ``only if'' direction; for ``if'' use Lemma~\ref{lem:phim embedding}.
\end{proof}

\noindent
Next we give a general version of Risler's real Nullstellensatz analogous to Theorem~\ref{thm:Rueckert}. For this, we need to recall some definitions and basic facts from real algebra.
Let   $I$ be an ideal of a ring $R$. One says that $I$ is {\it real}\/ if for every sequence~$a_1,\dots,a_k$ of elements
of $R$:
$$a_1^2+\cdots+a_k^2\in I \quad\Rightarrow\quad a_1,\dots,a_k\in I.$$
We have (cf.~\cite[Lemmas~4.1.5, 4.1.6]{BCR}):
\begin{enumerate}
\item If $I$ is real, then $I$ is radical;
\item if $I$ is real and $R$ is noetherian, then all minimal prime ideals of $R$ containing~$I$ are real; and
\item if $I$ is a prime ideal and $F:=\Frac(R/I)$, then $I$ is real iff $F$ is formally real (that is,  there is an ordering on $F$ making $F$ an ordered field).
\end{enumerate}
Next we define
$$\sqrt[\leftroot{2}\uproot{5}\operatorname{r}]{I} := \big\{ a\in R: \text{$a^{2k}+b_1^2+\cdots+b_l^2\in I$ for some $k\geq 1$  and $b_1,\dots,b_l\in R$}\big\}.$$
Then $\sqrt[\leftroot{2}\uproot{5}\operatorname{r}]{I}$ is the smallest real ideal of $R$ containing $I$, called the {\it real radical}\/ of $I$.
For $R=\k\lfloor X\rfloor$, we clearly have~$\operatorname{I}\!\big(\!\operatorname{Z}_K(I)\big)=\operatorname{I}\!\big(\!\operatorname{Z}_K(\sqrt[\leftroot{2}\uproot{5}\operatorname{r}]{I})\big)$.

\begin{theorem}[Risler's Nullstellensatz]\label{thm:Risler}
Let $I$ be an ideal of~$\k\lfloor X\rfloor$. Then
$$\operatorname{I}\!\big(\!\operatorname{Z}_K(I)\big)=\sqrt[\leftroot{2}\uproot{5}\operatorname{r}]{I}.$$
\end{theorem}
\begin{proof}
If we have some $a\in\operatorname{Z}_K(I)$, then some real prime ideal of~$\k\lfloor X\rfloor$ contains~$I$, namely the kernel of the
ring morphism $f\mapsto f(a)\colon\k\lfloor X\rfloor\to K$. Hence we may assume that some real prime ideal of $\k\lfloor X\rfloor$ contains $I$ and thus $\sqrt[\leftroot{2}\uproot{5}\operatorname{r}]{I}$. Then there are real prime ideals~$P_1,\dots,P_k$ ($k\geq 1$) of $\k\lfloor X\rfloor$
such that
$\sqrt[\leftroot{2}\uproot{5}\operatorname{r}]{I}=P_1\cap\cdots\cap P_k$.
We can arrange that   $I$ itself is a real prime ideal, and  need to show that then~$\operatorname{I}\!\big(\!\operatorname{Z}_K(I)\big)=I$. 
Equip the fraction field   of the integral domain $R:=\k\lfloor X\rfloor/I$ with some field ordering, 
let $\Omega$ be its real closure, and let $\sigma\colon \k\lfloor X\rfloor\to \Omega$ be the
composition of~$f\mapsto f+I\colon\k\lfloor X\rfloor\to R$ with the natural inclusions~$R\subseteq \Omega$.
Arguing as in the proof of Theorem~\ref{thm:Rueckert}, using Corollary~\ref{cor:Risler spec} instead of
Corollary~\ref{cor:spec},  yields~$\operatorname{I}\!\big(\!\operatorname{Z}_K(I)\big)\subseteq I$.
\end{proof}

\noindent
If $\k$ is real closed, then so is $\k(\!(t^*)\!)$; thus the
valued field $\k(\!(t^*)\!)$ is a model of~$\operatorname{RCVF}$. 
By the previous theorem we hence immediately obtain, in a similar way
as Theorem~\ref{thm:Rueckert} implied Corollary~\ref{cor:Rueckert, 1}:

\begin{cor} \label{cor:Risler, 1}
Suppose  $\k$ is real closed. Let $f, g_1,\dots,g_n\in \k[[X]]$ with~${f(a)=0}$ for all $a\in t\k[[t]]^m$ such that $g_1(a)=\cdots=g_n(a)=0$. Then there are $k\geq 1$ and~$b_1,\dots,b_l,h_1,\dots,h_n\in \k[[X]]$   such that $$f^{2k}+b_1^2+\cdots+b_l^2=g_1h_1+\cdots+g_nh_n.$$ 
\end{cor}

\begin{remark}
If in the context of the previous corollary  $f$ and each $g_j$ are in $\k_0[[X]]$, 
where $\k_0$ is a euclidean subfield of $\k$, then we can take the $b_i$, $h_j$ in~$\k_0[[X]]$.
To see this apply Theorem~\ref{thm:Risler} to the Weierstrass system $(\k_0[[X]])$ instead of
the  Weierstrass system $(\k[[X]])$.
Similarly with 
$\k_0[[X]]^{\operatorname{a}}$ or
  $\k_0[[X]]^{\operatorname{da}}$ in place of $\k_0[[X]]$.
\end{remark}

\noindent
The field $\R\{\!\{t^*\}\!\}$ is  also real closed; hence we   obtain, similarly to Corollary~\ref{cor:Rueckert, 2}:

\begin{cor}  \label{cor:Risler, 2}
Let 
$f,g_1,\dots,g_n\in \R\{X\}$, and suppose    
 that for each $a\in \R\{t\}^m$ with~$a(0)=0$ we have $f(a)=0$ whenever $g_1(a)=\cdots=g_n(a)=0$. Then  there are~$k\geq 1$ and~$b_1,\dots,b_l,h_1,\dots,h_n\in \R\{X\}$   such that $$f^{2k}+b_1^2+\cdots+b_l^2=g_1h_1+\cdots+g_nh_n.$$ 
\end{cor}

\noindent
Finally, the previous corollary yields a version for germs of real analytic functions:

\begin{cor}\label{cor:Risler fns}
Let $f,g_1,\dots,g_n\colon U\to\R$ be real analytic functions, where $U$ is an open neighborhood of $0$ in $\R^m$, such that
  for all $a\in U$,
  $$g_1(a)=\cdots=g_n(a) \quad\Longrightarrow\quad f(a)=0.$$ Then there are real analytic functions $b_1,\dots,b_l,h_1,\dots,h_n\colon V\to\R$,
for some open neighborhood $V\subseteq U$ of $0$ in $\R^m$, and some $k\geq 1$ such that 
$$f^{2k}+b_1^2+\cdots+b_l^2=g_1h_1+\cdots+g_nh_n\qquad\text{ on $V$.}$$
\end{cor}

\begin{remark}
In Corollary~\ref{cor:Risler fns},
if the germs of $f$ and of the $g_j$ at $0$ are algebraic (over~$\R[X]$), then we can choose the $b_i$, $h_j$ such that
their germs at $0$ are also algebraic. Similarly with ``differentially algebraic'' in place of ``algebraic''. (Cf.~the remark following
Corollary~\ref{cor:Risler, 1}.)
\end{remark}

\section{The Theory $p\operatorname{CVF}$}\label{sec:pCVF}

\noindent
Our aim here is to formulate and prove a theorem analogous to Theorem~\ref{thm:CD} in the $p$-adic context.
We first recall the basic algebraic and model-theoretic facts about $p$-adically closed fields.
{\it Throughout this section $K$ is a field of characteristic zero.}\/

\subsection{$p$-valued fields}
Recall that a valuation  on $K$ with residue field   $\F_p$  such that~$p$ is an element of smallest positive valuation is
called a {\it $p$-valuation}\/ on~$K$. The valuation ring of a $p$-valuation on $K$ is said to be a {\it $p$-valuation ring}\/ of $K$, and
a   field of characteristic zero equipped with one of its $p$-valuation rings is a {\it $p$-valued field.}\/
If $K$ is $p$-valued, then so is each valued subfield of $K$.
We also note:

\begin{lemma}\label{lem:p-val spec}
Let $K$ be a valued field and   $\Delta\neq\{0\}$ be a convex subgroup of its value group. Then
$K$ is $p$-valued iff the $\Delta$-specialization $\dot K$ of $K$ is $p$-valued, and
in this case the $\Delta$-coarsening of $K$ is not $p$-valued.
\end{lemma}
\begin{proof}
The valued field
$\dot K$ has value group $\Delta$, and its residue field is isomorphic to~$\res(K)$.
\end{proof}


\subsection{$p$-adically closed fields}
One says that a $p$-valued field  is {\it $p$-adically closed}
if it  has no proper algebraic $p$-valued field extension.
The   fundamental result about $p$-adically closed fields is the following (see \cite[Theorem~3.1]{PR}): 

\begin{theorem}\label{thm:AKE}
A $p$-valued field is $p$-adically closed iff it is henselian and its value group is a $\Z$-group.
\end{theorem}

\noindent
Using the   Ax-Kochen/Er\v{s}ov Theorem for unramified henselian valued fields of mixed characteristic
(cf., e.g.,~\cite[Theorem~7.1]{vdD}) this implies:

\begin{cor}\label{cor:pCF complete}
Any two $p$-adically closed valued fields \textup{(}viewed as $\mathcal L_{\preceq}$-structures as usual\textup{)} are elementarily equivalent.
\end{cor}

\noindent
We also recall a well-known fact:

\begin{lemma}\label{lem:def val}
Let $\mathcal O$ be a $p$-valuation ring of $K$, and set $\varepsilon:=3$ if~$p=2$ and~$\varepsilon:=2$ otherwise. Then 
$$\mathcal O\ \supseteq\ \mathcal O_0:=\big\{ a\in K: \text{there is some $b\in K$ with $1+pa^\varepsilon=b^\varepsilon$}\big\},$$ 
with $\mathcal O=\mathcal O_0$ if the valued field $(K,\mathcal O)$ is henselian.  
\end{lemma}

\noindent
In particular, if $(K,\mathcal O)$ is $p$-adically closed, then   $\mathcal O$ is the unique $p$-valuation ring of~$K$.
Theorem~\ref{thm:AKE} in combination with
Lemmas~\ref{lem:Z-gps},~\ref{lem:p-val spec}, and  \cite[Lemma~3.4.2]{ADH}
also yields a specialization result for $p$-adic closedness:

\begin{lemma}\label{lem:p-ad closed spec}
Let $K$, $\Delta$ be as in Lemma~\ref{lem:p-val spec}.
Then $K$ is $p$-adically closed iff the $\Delta$-coarsening  of $K$ is henselian with divisible value group   and the $\Delta$-specialization~$\dot K$ of $K$ is $p$-adically closed.
\end{lemma}
 
\noindent
In the next lemma and its corollary we let $L\supseteq K$ be an extension of $p$-valued fields. 

\begin{lemma}\label{lem:alg closure in p-adically closed}
If $L$ is henselian, then
$$\text{$K$ is algebraically closed in $L$}\quad \Longleftrightarrow\quad\text{$K$ is henselian and $\Gamma_L/\Gamma$ is torsion-free.}$$
\end{lemma}

\noindent
For a proof of this lemma see  \cite[Lemma~4.2]{PR}.
Lemma~\ref{lem:alg closure in p-adically closed} in combination with Lemma~\ref{lem:Z-subgp} and Theorem~\ref{thm:AKE} implies:

\begin{cor}\label{cor:alg closure in p-adically closed}
If $L$ is $p$-adically closed,
then the following are equivalent:
\begin{enumerate}
\item $K$ is $p$-adically closed;
\item $K$ is algebraically closed in $L$;
\item $K$ is henselian and $\Gamma_L/\Gamma$ is torsion-free.
\end{enumerate}
\end{cor}
 
\noindent
A  field which carries at least one $p$-valuation is said to be {\it formally $p$-adic.}\/ 
Every formally $p$-adic field has characteristic zero.
If $K$ is formally $p$-adic and satisfies one of the equivalent conditions in the next lemma, then $K$ is called {\it $p$-adically closed.}\/

\begin{lemma}
For a  formally $p$-adic field $K$, the following statements   are equivalent:
\begin{enumerate}
\item $K$   has no proper algebraic formally $p$-adic   field extension;
\item  each $p$-valuation ring of $K$ makes $K$ a $p$-adically closed valued field;
\item some valuation ring of $K$ makes $K$ a $p$-adically closed valued field.
\end{enumerate}
\end{lemma}
\begin{proof}
The implications (1)~$\Rightarrow$~(2)~$\Rightarrow$~(3) are clear. Let $\mathcal O$ be a valuation ring of~$K$ making
$K$ a $p$-adically closed valued field. Let $F$ be an algebraic  field extension of~$K$, and let
$\mathcal O_F$ be a $p$-valuation ring of $F$. Then~$\mathcal O_F\cap K$ is a $p$-valuation ring of~$K$, so $\mathcal O_F\cap K=\mathcal O$  and thus $F=K$ since~$(K,\mathcal O)$ is $p$-adically closed.
\end{proof}


\noindent
Below we will have occasion to discuss field equipped with two dependent valuations, one of which is a $p$-valuation.
To facilitate keeping them apart we now generally denote $p$-valuation rings of $K$ by $\mathcal O_p$, 
 with maximal ideal~$\smallo_p$,
valuation~$v_p\colon K^\times\to\Gamma_p$,   and associated dominance relation $\preceq_p$.
If we want to make the dependence on~$K$ explicit, we write $\mathcal O_{K,p}$ instead of $\mathcal O_p$, etc.
(Below, $K$ will often have a unique $p$-valuation ring, making these notations
unambiguous in this case.)

\begin{examples}
Let $\k$ be a $p$-adically
closed field. Let  $v_{\k}\colon\k^\times\to\Gamma_{\k}$ be the $p$-valuation of~$\k$,
with valuation ring $\mathcal O_{\k}$ and  maximal ideal $\smallo_{\k}$.
Consider the  Hahn field $K=\k(\!(t^\Gamma)\!)$ where
$\Gamma\neq\{0\}$ is divisible.
Then the $t$-adic valuation $v\colon K^\times\to\Gamma$ of $K$ is not a $p$-valuation.
However, the field $K$  is $p$-adically closed: 
Let  $\Gamma_p:=\Gamma\times \Gamma_{\k}$  equipped with the lexicographic ordering, and
for  $f=\sum_{\gamma\in\Gamma} f_\gamma\gamma\in K^\times$ ($f_\gamma\in\k$) set 
$$v_p(f) := \big(\delta,v_{\k}(f_{\delta})\big)  \text{ where $\delta=vf$;}$$
then $v_p\colon K^\times\to\Gamma_p$ is 
the   $p$-valuation of $K$. (Cf.~Lem\-ma~\ref{lem:p-ad closed spec}.)
The valuation ring of~$v_p$ is 
$\mathcal O_p = \mathcal O_{\k} + \smallo$, with maximal ideal~$\smallo_{\k}+\smallo$,
where~$\smallo$ is the maximal ideal of the valuation ring $\mathcal O=\k+\smallo$ of $v$.
For $\Gamma=\Q$,  the $p$-valuation on $\k(\!(t^{\Q})\!)$ restricts to a $p$-valuation
on its subfield~$\k(\!(t^*)\!)$, making $\k(\!(t^*)\!)$  a $p$-adically closed valued field.
For $\k=\Q_p$,   the $p$-valuation on
$\Q_p(\!(t^*)\!)$ further restricts to a $p$-valuation on its
subfield $\Q_p\{\!\{t^*\}\!\}$, with valuation ring $\Z_p+\smallo_t$ having maximal ideal $p\Z_p+\smallo_t$.
\end{examples}

\noindent
The subgroup $(K^\times)^n={\{a^n:a\in K^\times\}}$ of the multiplicative group of a $p$-adically closed field $K$ consisting of the $n$th powers of non-zero elements of $K$ plays an important role in the study of $K$. 
For $K=\Q_p$ we  recall a consequence of   compactness of $\Z_p$ and density of $\Z$ in $\Z_p$:

\begin{lemma}\label{lem:coset repres}
Let $n\geq 1$; then $\big[\Q_p^\times:(\Q_p^\times)^n\big]<\infty$, and each
coset of the subgroup~$(\Q_p^\times)^n$ of $\Q_p^\times$ has a representative in  $\Z$.
\end{lemma}

\noindent
We now let $\mathcal L_{\operatorname{Mac}}$ be the expansion of the language $\mathcal L_{\operatorname{R}}$ of rings by distinct unary predicate
symbols~$\operatorname{P}_n$~(${n\geq 1}$), a binary relation symbol $\preceq_p$, and a unary function symbol  ${}^{-1}$.
In the following we construe a $p$-valued   field $K$  as an $\mathcal L_{\operatorname{Mac}}$-structure by interpreting $\operatorname{P}_n$ by the set $(K^\times)^n$, the binary relation symbol $\preceq_p$ by the valuation of~$K$, and 
 ${}^{-1}$ by the multiplicative inverse function $a\mapsto a^{-1}$ extended to a total
function~${K\to K}$ via $0^{-1}:=0$. 
Note that the class of  $p$-adically closed valued fields is $\mathcal L_{\operatorname{Mac}}$-axiomatizable thanks to
Theorem~\ref{thm:AKE}. 
We let~$p\operatorname{CF}$ be the $\mathcal L_{\operatorname{Mac}}$-theory of $p$-adically closed  fields.

\begin{theorem}[Macintyre]\label{thm:Macintyre}
$p\operatorname{CF}$  admits~\textup{QE}.
\end{theorem}

\noindent
We refer to \cite{Macintyre} or \cite[Theorem~5.6]{PR} for    proofs of this theorem. 
For later use we note that the symbols $\preceq_p$ and ${}^{-1}$ are not actually necessary to achieve QE.
For this let $\mathcal L^*_{\operatorname{Mac}}$ be the reduct $\mathcal L_{\operatorname{R}}\cup\{\operatorname{P}_n:n\geq 1\}$
of $\mathcal L_{\operatorname{Mac}}$. The following lemma has already been noted, e.g., in \cite[Theorem~2.2]{Denef}. (Note that
the proof relies on the specific structure of terms in  $\mathcal L_{\operatorname{Mac}}$, and would break down, e.g.,
for the $(\mathcal L_{\operatorname{Mac}})_W$-theory~$p\operatorname{CF}_W$, where $W$ is the Weierstrass system of convergent $p$-adic power series, and the  sublanguage of
  $(\mathcal L_{\operatorname{Mac}})_W$ obtained by removing the symbols $\preceq_p$ and  ${}^{-1}$.)

\begin{lemma}\label{lem:Macintyre}
Every quantifier-free $\mathcal L_{\operatorname{Mac}}$-formula is $p\operatorname{CF}$-equivalent to a 
quan\-ti\-fier-free $\mathcal L^*_{\operatorname{Mac}}$-formula.
\end{lemma}
\begin{proof}
Take   representatives~$\lambda_1,\dots,\lambda_N\in\Z$ ($N\geq 1$)    for the cosets of the subgroup~$(\Q_p^\times)^n$   of~$\Q_p^\times$.  
Let $K$ be a $p$-adically closed valued field.
Then $\lambda_1,\dots,\lambda_N$ are also representatives for the cosets of the subgroup~$(K^\times)^n$ of~$K^\times$, by completeness
of~$p\operatorname{CF}$ (Corollary~\ref{cor:pCF complete}), hence for $a,b\in K$, $b\neq 0$:
$$a/b \in (K^\times)^n \quad\Longleftrightarrow\quad  a\lambda_i, b\lambda_i\in  (K^\times)^n\text{ for some $i\in\{1,\dots,N\}$.}$$
Also, let $\varepsilon$ be as in Lemma~\ref{lem:def val}; then for  $a,b\in K$ of  we have
$$a \preceq_p b\quad\Longleftrightarrow\quad \text{$b=0$, or $b\neq 0$ and $\big(1+p(a/b)^\varepsilon\big)\in (K^\times)^\varepsilon$.}$$
This easily yields that $p\operatorname{CF}$ has closures of  $\mathcal L^*_{\operatorname{Mac}}$-substructures in the sense
of the remarks following~\cite[B.11.4]{ADH}, so the claim follows from loc.~cit.
\end{proof}

\noindent
Call a field $K$   {\it $p$-euclidean}\/ if it is formally $p$-adic  and $\big[K^\times:(K^\times)^n\big]=\big[\Q_p^\times:(\Q_p^\times)^n\big]$ for each $n\geq 1$.

\begin{lemma}\label{lem:p-euclidean}
Suppose $K$ is $p$-euclidean. Then $K$ has  
a unique expansion
to an $\mathcal L_{\operatorname{Mac}}$-structure which is a substructure of a model of $p\operatorname{CF}$.
\textup{(}In particular, $K$ has a unique $p$-valuation ring.\textup{)}
\end{lemma}
\begin{proof}
Since $K$ is formally $p$-adic, it has a $p$-adically closed extension $F$, and this gives rise to an
expansion of $K$ to an $\mathcal L_{\operatorname{Mac}}$-substructure of $F$ viewed as an $\mathcal L_{\operatorname{Mac}}$-structure.
By  Lemma~\ref{lem:coset repres} and $\big[K^\times:(K^\times)^n\big]=\big[\Q_p^\times:(\Q_p^\times)^n\big]$ we have~${(F^\times)^n\cap K^\times}=(K^\times)^n$ for each $n\geq 1$.
Let $\mathcal O_p$ be the valuation ring of the $\mathcal L_{\operatorname{Mac}}$-structure $K$, and let
$\varepsilon$ and $\mathcal O_0$ be as in Lemma~\ref{lem:def val}. Then $\mathcal O_0\subseteq\mathcal O_p$ by that lemma, and we
claim $\mathcal O_p=\mathcal O_0$. For this, let $a\in\mathcal O_p$;
then   $1+pa^\varepsilon\in (F^\times)^\varepsilon\cap K^\times=(K^\times)^\varepsilon$,
so we get $a\in\mathcal O_0$.
\end{proof}

\noindent
A field which satisfies the conclusion of the last lemma called {\it weakly $p$-euclidean.}\/ We always construe a weakly
$p$-euclidean field as a valued field via its unique $p$-valuation ring.
Every $p$-adically closed field is $p$-euclidean; but~$\Q$ is not. However:

\begin{lemma}\label{lem:Q strongly p-euclidean}
The field $\Q$ is weakly $p$-euclidean.
\end{lemma}
\begin{proof}
Clearly $\Q$ has only one $p$-valuation ring, namely 
$$\Z_{(p)}=\left\{\frac{a}{b}:a,b\in\Z,\ b\notin p\Z\right\},$$
 and thanks to
Corollary~\ref{cor:pCF complete}  we also have
$(\Q_p^\times)^n\cap\Q^\times = (K^\times)^n\cap\Q^\times$ for each~$K\models p\operatorname{CF}$ and~$n\geq 1$.
\end{proof}

\begin{question}
Are there   $p$-euclidean algebraic number fields?
\end{question}

\subsection{$p$-convex valuations}
Let $(K,\mathcal O_p)$ be a $p$-valued field.  A subring $\mathcal O$ of $K$ is said to be {\it $p$-convex}\/ if it is convex
in $(K,\mathcal O_p)$, that is, if $\mathcal O_p\subseteq \mathcal O$. (See Section~\ref{sec:convexity}.)  Every such $p$-convex subring of $K$ is a valuation ring of $K$. If $\preceq$ is the dominance relation on $K$ associated to a $p$-convex valuation ring  
of $K$, then 
$$a\preceq_p b \ \Longrightarrow\  a\preceq b,\qquad a\prec b   \ \Longrightarrow\  a\prec_p b \qquad \text{for all $a,b\in K$.}$$
If $K\subseteq L$ is an extension of $p$-valued fields and $\mathcal O_L$ is a $p$-convex subring of $L$, then~$\mathcal O_L\cap K$ is
a $p$-convex subring of $K$.
In the next lemma and its corollary we let $(K,\mathcal O)$ be an arbitrary valued field with value group~$\Gamma$ and residue field  $\k$.
Note that ``$\subset$'' means ``proper subset''.

\begin{lemma}\label{lem:char p-adically closed p-convex}
\begin{multline*}
\text{$K$ is a $p$-adically closed field and~$\mathcal O_p\subset \mathcal O$}\quad\Longleftrightarrow\quad \\
\text{$(K,\mathcal O)$ is henselian, $\Gamma$ is divisible, and  $\k$ is $p$-adically closed.}
\end{multline*}
\end{lemma}
\begin{proof}
Suppose the field $K$ is $p$-adically closed and~$\mathcal O_p\subset \mathcal O$. Let $\Delta$ be the convex subgroup of $\Gamma_p$
such that the $\Delta$-coarsening of $(K,\mathcal O_p)$ is $(K,\mathcal O)$.
Then $\Delta\neq\{0\}$,  and by Lemma~\ref{lem:p-ad closed spec}, $(K,\mathcal O)$ is henselian, its value group   is divisible, and 
the  $\Delta$-specialization  of $(K,\mathcal O_p)$ is $p$-adically closed; the underlying field of the latter is~$\k=\mathcal O/\smallo$.
Conversely, suppose $(K,\mathcal O)$ is henselian with divisible value group,
and let~$\mathcal O_{\k}$ be a $p$-valuation ring of  $\k$ making $\k$ a $p$-adically closed valued field.
Take a valuation ring~$\mathcal O_0$ of $K$ contained in $\mathcal O$ and a convex subgroup $\Delta_0$ of the value group of
$\mathcal O_0$ such that the $\Delta_0$-coarsening of $(K,\mathcal O_0)$ is $(K,\dot{\mathcal O_0})=(K,\mathcal O)$ and the $\Delta_0$-specialization of~$(K,\mathcal O_0)$ is~$(\k,\mathcal O_{\k})$.
Note that $\Delta_0\neq\{0\}$, so $\mathcal O_0\subset\mathcal O$. Also,  $(K,\dot{\mathcal O_0})=(K,\mathcal O)$ is henselian, its value group~$\Gamma_0/\Delta_0\cong\Gamma$ is divisible, and 
the
 $\Delta_0$-specialization of~$(K,\mathcal O_0)$ is $p$-adically closed; hence $(K,\mathcal O_0)$ is $p$-adically closed by Lemma~\ref{lem:p-ad closed spec}.
\end{proof}

\begin{cor}\label{cor:char p-adically closed p-convex}
 \begin{multline*}
\text{$K$ is   $p$-adically closed and $\mathcal O$ is a non-trivial $p$-convex subring of $K$}\ \Longleftrightarrow\ \\
\text{$(K,\mathcal O)$ is henselian, $\Gamma\neq\{0\}$ is divisible, and  $\k$ is $p$-adically closed.}
\end{multline*}
\end{cor}

\begin{example}
Let $\Gamma\neq\{0\}$ be a divisible ordered abelian group, written additively, and let $\k$ be a $p$-adically
closed field. Then the valuation ring of the $t$-adic valuation on the 
Hahn field $K=\k(\!(t^\Gamma)\!)$ is   a non-trivial $p$-convex subring of $K$.
Similarly, the $t$-adic valuation ring $\k[[t^*]]$ of  field $\k(\!(t^*)\!)$ of Puiseux series  over $\k$
 is   a non-trivial $p$-convex subring of $\k(\!(t^*)\!)$.
\end{example}

\noindent
If   $K$ is a $p$-valued field, we define the {\it $p$-convex hull}\/ (in $K$) of a subring $R$ of $K$ to be the convex hull of $R$ in the
valued field $K$.

\begin{lemma}\label{lem:p-convex hull}
Let $(K,\mathcal O_p)$ be a $p$-valued field and let $F$ be a subfield   of $K$, with $p$-convex hull $R$ in $K$.
Then $R\neq\mathcal O_p$, hence if some   $p$-convex subring $\mathcal O\neq K$ of $K$ contains~$F$, then $R$ is
a non-trivial $p$-convex subring of $K$.
\end{lemma}
\begin{proof}
For the first statement use that $\mathcal O_p$ does not contain a subfield of $K$; the rest follows from this.
\end{proof}
 
\noindent
We now let $\mathcal L=\mathcal L_{\operatorname{Mac},{\preceq}}$ be the expansion of the language $\mathcal L_{\operatorname{Mac}}$
of Macintyre by a single binary relation symbol $\preceq$ (not to be confused with the symbol $\preceq_p$ in $\mathcal L$, intended for the dominance
relation associated to a $p$-valuation).
Let  
$p\operatorname{CVF}$ be the $\mathcal L$-theory
expanding $p\operatorname{CF}$ by axioms which state that $\preceq$ is interpreted by
a dominance relation associated to a non-trivial  $p$-convex valuation ring.

\begin{theorem}\label{thm:Belair}
$p\operatorname{CVF}$ has \textup{QE}.
\end{theorem}

\noindent
This $p$-adic analogue of Theorem~\ref{thm:CD} is stated without proof in \cite[Theorem~5]{CD}; see also \cite[Corollaire~2.3(1)]{Belair} and~\cite[Lem\-ma~2.13]{Guzy}. For the proof, these sources  quote the unpublished thesis~\cite{Delon} or  the notationally dense and long paper~\cite{Weispfenning84} (which also yields
a primitive recursive algorithm for QE in $p\operatorname{CVF}$). For the convenience of the reader we include a deduction of the important Theorem~\ref{thm:Belair} below from   standard facts in the literature.  
We start with a somewhat weaker result:

\begin{prop}\label{prop:Belair}
$p\operatorname{CVF}$ is model complete.
\end{prop}

\noindent
This follows from the Ax-Kochen/Er\v{s}ov Theorem for elementary substructures. Here, valued fields
are viewed as $\mathcal L_{\preceq}$-structures where $\mathcal L_{\preceq}$ is the language described in Section~\ref{sec:mtvf}.

\begin{theorem}\label{thm:AKE preceq}
Let $K\subseteq E$ be an extension of henselian valued fields of equicharacteristic zero. Then 
$$K\preceq E \quad\Longleftrightarrow\quad \text{$\res(K)\preceq\res(E)$ and $\Gamma\preceq\Gamma_E$.}$$
\end{theorem}

\noindent
For a version of this fact in a $3$-sorted context, for valued fields equipped with an angular component map,
see \cite[Corollary~5.23]{vdD}. One may reduce to this setting by replacing~$(E,K)$ by an $\aleph_1$-saturated elementary
extension, after which $E$ can be equipped with a cross-section which restricts to a cross-section of $K$
(cf.~\cite[Lemmas~3.3.39, 3.3.40]{ADH}).

\begin{proof}[Proof of Proposition~\ref{prop:Belair}]
Let $K\subseteq E$ be    models of $p\operatorname{CVF}$. 
By Corollary~\ref{cor:char p-adically closed p-convex}, the valued field $E$ is henselian, $\Gamma_E\neq\{0\}$ is divisible, and $\res(E)$ is $p$-adically closed; similarly
for $K$, $\Gamma$, $\res(K)$ in place of~$E$,~$\Gamma_E$,~$\res(E)$, respectively.
We have~$\res(K)\preceq\res(E)$ by Theorem~\ref{thm:Macintyre}, and also~$\Gamma\preceq\Gamma_E$ (see, e.g.,~\cite[Example~B.11.12]{ADH}). Hence~$K$ is an elementary $\mathcal L_{\preceq}$-substructure of $E$, by Theorem~\ref{thm:AKE preceq}.
By Corollary~\ref{cor:alg closure in p-adically closed}, $K$ is algebraically closed in $E$, so
 $(E^\times)^n\cap K=(K^\times)^n$ for each $n\geq 1$.
Hence also using Lemma~\ref{lem:def val} we obtain that
$K$ is  also
an elementary $\mathcal L$-substructure of $E$.
\end{proof}

\noindent
By Proposition~\ref{prop:Belair} we have 
$$\Q_p\{\!\{t^*\}\!\}\ \preceq\ \Q_p(\!(t^*)\!)\ \preceq\ \ \operatorname{cl}(\Q_p[t^{\Q}])\ \preceq\  \Q_p(\!(t^{\Q})\!)$$
as $\mathcal L$-structures.
In order to finish the proof of  Theorem~\ref{thm:Belair}, we recall a well-known fact about $n$th powers
in  henselian valued fields of equicharacteristic zero:

\begin{lemma}\label{lem:nth powers residue field}
Let $K$ be a henselian valued field of equicharacteristic zero, and suppose~$n\geq 1$. Then 
$$(\mathcal O^\times)^n = (K^\times)^n\cap\mathcal O^\times=\big\{a\in\mathcal O^\times:\overline{a}\in (\k^\times)^n\big\}.$$
Hence for $a,b\in K^\times$ with $a\sim b$ we have $a\in (K^\times)^n \Leftrightarrow b\in (K^\times)^n$.
\end{lemma}

\begin{proof}[Proof of  Theorem~\ref{thm:Belair}] 
Let  $E, F \models p\operatorname{CVF}$ and  let~$K$ be a  common substructure of~$E$ and of~$F$;
we need to show $E_K \equiv F_K$; cf.~\cite[Corollary~B.11.6]{ADH}. The valuation ring $\mathcal O$ of $K$ is   $p$-convex. 
Note that $\res(K)=\mathcal O/\smallo\subseteq \res(E)$ has characteristic zero, thus~$\mathcal O_p\subset\mathcal O$.
We reduce to the case $\Gamma\neq\{0\}$.
Suppose $\Gamma=\{0\}$.
Replacing~$E$,~$F$ by elementary extensions we first arrange that
$E$, $F$ are $\aleph_1$-saturated, hence we have cross-sections
$s_E\colon \Gamma_E\to E^\times$ of $E$ and $s_F\colon \Gamma_F\to F^\times$ of~$F$.
Take any $\alpha\in\Gamma_E^>$ and set~$x:=s_E(\alpha)\in E^\times$; 
then $x$ is transcendental over~$K$ by~\cite[Co\-rol\-lary~3.1.8]{ADH}.
Since~$\Gamma_E$ is divisible, we have~$x\in (E^\times)^n$ for each $n\geq 1$. 
Similarly we take any~$\beta\in\Gamma_F^>$ and set~$y:=s_F(\beta)\in F^\times$;
then $y$ is transcendental over~$K$ and~$y\in (F^\times)^n$ for each~$n\geq 1$.
The field isomorphism $\sigma\colon K(x)\to K(y)$ over $K$ with~$\sigma(x)=y$ is a valued field isomorphism
\cite[Lemma~3.1.30]{ADH}.
Let $0\neq f\in K[T]$, where $T$ is a single variable over $K$, and $n\geq 1$. 
Let $a\in K^\times$, $g\in K[T]$, and $m$ be such that~$f=aT^m(1+gT)$. Then
$f(x) \sim ax^m$, hence by Lemma~\ref{lem:nth powers residue field}:
$$f(x)\in (E^\times)^n\quad\Longleftrightarrow\quad ax^m \in (E^\times)^n
\quad\Longleftrightarrow\quad a\in (E^\times)^n.$$
Similarly we obtain $f(y)\in (F^\times)^n \Leftrightarrow a\in (F^\times)^n$.
Since $K$ is an $\mathcal L_{\operatorname{Mac}}$-sub\-structure of both $E$ and $F$, this yields
$f(x)\in (E^\times)^n\Leftrightarrow f(y)\in (F^\times)^n$. 
If also~$0\neq g\in K[T]$ and $h:=f/g\in K(T)$, then 
$$h(x)\in (E^\times)^n \ \Longleftrightarrow\  (fg^{n-1})(x) / g(x)^n \in (E^\times)^n \ \Longleftrightarrow\   
(fg^{n-1})(x)\in (E^\times)^n,$$
and similarly with $F$ and $y$ in place of $E$ and $x$, respectively. Hence by the above applied to $fg^{n-1}\in K[T]$ in place of $f$ we obtain
$h(x)\in (E^\times)^n \Leftrightarrow h(y)\in (F^\times)^n$.
Using Lemma~\ref{lem:def val}, this also implies
$h(x)\preceq_p 1 \Leftrightarrow h(y) \preceq_p 1$.
Hence $\sigma$ is   an isomorphism $K(x)\to K(y)$ between   $\mathcal L_{\operatorname{Mac}}$-sub\-structures of $E$
and of $F$, respectively. Now identify $K(x)$ with its image under~$\sigma$ and replace~$K$ by $K(x)$ to arrange $\Gamma\neq\{0\}$.

Let~$L$ be the algebraic closure of $K$ inside $E$. Then $L$  
is $p$-adically closed, by Corollary~\ref{cor:alg closure in p-adically closed}. 
Now \cite[Theorem~3.10]{PR} yields a $K$-embedding~$j$ of the $p$-valued field~$L$ into $F$, and 
 $j$ is also an embedding of  
 $\mathcal L_{\operatorname{Mac}}$-structures. 
By Lemma~\ref{lem:unique convex ext}, there is a unique
$p$-convex subring of $L$ lying over~$\mathcal O$, hence $j$ is also an embedding of $\mathcal L$-structures.
Identifying $L$ with its image under $j$ we first arrange that~$L$ is a common $\mathcal L$-substructure of $E$ and   $F$, and
replacing $K$ by $L$ we  then arrange that~$K\models p\operatorname{CVF}$. Then
$K\preceq E$ and $K\preceq F$ by Proposition~\ref{prop:Belair} and thus   $E\equiv_K F$.
\end{proof}

\noindent
Let $\mathcal L^*:=\mathcal L^*_{\operatorname{Mac}}\cup\{\preceq\}$ be the expansion of the language of rings
$\mathcal L_{\operatorname{R}}$ by
the unary relation symbols~$\operatorname{P}_n$~({$n\geq 1$}) and the binary relation symbol $\preceq$ (a reduct of $\mathcal L$).

\begin{cor}\label{cor:Belair}
Each   $\mathcal L$-formula is $p\operatorname{CVF}$-equivalent to a 
quan\-ti\-fier-free $\mathcal L^*$-for\-mu\-la.
\end{cor}

\noindent
This follows from Lemma~\ref{lem:Macintyre} and Theorem~\ref{thm:Belair}.

\begin{cor}
$p\operatorname{CVF}$ is complete:  $K\equiv\Q_p(\!(t^*)\!)$ for each $K\models p\operatorname{CVF}$.
\end{cor} 
\begin{proof}
Expand the field $\Q$ to an $\mathcal L$-structure by interpreting $\preceq_p$ by the $p$-adic dominance relation on $\Q$, $\preceq$ by the trivial dominance relation on $\Q$,
and~$\operatorname{P}_n$ by~$(\Q_p^\times)^n\cap\Q^\times$  ($n\geq 1$). 
Note that $(\Q_p^\times)^n\cap\Q = (K^\times)^n\cap\Q$ for each~$K\models p\operatorname{CF}$ and~$n\geq 1$, thanks to
Corollary~\ref{cor:pCF complete}. This $\mathcal L$-structure embeds into each model of~$p\operatorname{CVF}$
(see Lemma~\ref{lem:Q strongly p-euclidean}), 
  so the corollary follows from Theorem~\ref{thm:Belair}.
\end{proof}

\noindent
In \cite{Belair88}, B\'elair gives an explicit axiomatization of the universal part
$p\operatorname{CF}_\forall$ of the $\mathcal L_{\operatorname{Mac}}$-theory 
$p\operatorname{CF}$. We note:

\begin{cor}
The $\mathcal L$-theory $p\operatorname{CVF}_\forall$ is axiomatized by 
$$p\operatorname{CF}_\forall \cup 
\big\{ 1\preceq n\cdot 1  :n\geq 1 \big\} \cup
\big\{\forall x\forall y(x\preceq_p y\rightarrow x\preceq y)\big\}.$$
\end{cor}
\begin{proof}
Clearly $p\operatorname{CVF}_\forall$ contains the displayed $\mathcal L$-theory.
Conversely, let $K$ be a model of the latter. Then the valuation ring $\mathcal O$
associated to the dominance relation~$\preceq$ on $K$ is $p$-convex, and its residue field has characteristic zero.
Take an extension $K_1$ of the $\mathcal L_{\operatorname{Mac}}$-reduct of $K$ to a model of
$p\operatorname{CF}$, and expand it to an $\mathcal L$-structure by
interpreting $\preceq$ by the dominance relation associated to the $p$-convex hull
of $\mathcal O$ in $K_1$. Then $K\subseteq K_1$ as $\mathcal L$-structures, so
replacing~$K$ by~$K_1$ we can   arrange that $K\models p\operatorname{CF}$. 
If the $p$-convex subring $\mathcal O$ of $K$ is non-trivial, then we are done.
Suppose otherwise; then $\mathcal O=K$.  
Then  by the example following Corollary~\ref{cor:char p-adically closed p-convex},
the valuation ring of the $t$-adic valuation on the 
Hahn field~$L:=K(\!(t^\Q)\!)$ is   a non-trivial $p$-convex subring of $L$, and
accordingly expanding~$L$ to an $\mathcal L$-structure
yields
 an extension of $K$ to  a model of $p\operatorname{CVF}$ as required. 
\end{proof}

\section{The $p$-adic Analytic Nullstellensatz}\label{sec:p-adic NS}

\noindent
{\it In this section we 
let $K$ be a field of characteristic zero.}\/  
We first recall the definition and a few basic facts about the $p$-adic Kochen operator $\gamma=\gamma_p$,
the $p$-adic Kochen ring~$\Lambda_K$, and generalize some auxiliary results about $p$-adic ideals from~\cite{Srhir01, Srhir03}.  
In Proposition~\ref{prop:p-adic spec} we   establish a specialization result analogous
to Propositions~\ref{prop:Rueckert spec} and~\ref{prop:Risler spec}, from which a general form of the 
$p$-adic analytic version of Hilbert's 17th Problem (Corollary~\ref{cor:p-adic H17}) and the
$p$-adic analytic Nullstellensatz
(Theorem~\ref{thm:p-adic NSS}) follow. Theorems A,~B, and~C from the introduction are then easy consequences.
(See Corollaries~\ref{cor:p-adic NSS, 1}--\ref{cor:p-adic H17 fns}.)

\subsection{The Kochen ring}
We set~$K_\infty:=K\cup\{\infty\}$ where $\infty\notin K$.  
For $f\in K$ we put $\wp(f):=f^p-f$ and
$$\gamma(f) := \frac{1}{p} \frac{\wp(f)}{\wp(f)^2-1} \in K_\infty\quad\text{where $\gamma(f):=\infty$ if $\wp(f)=\pm 1$.}$$ 
For the following see \cite[Lemma~6.1]{PR}:

\begin{lemma}\label{lem:p-val gamma}
Let $\mathcal O$ be a valuation ring of $K$. Then
$\mathcal O$ is a  $p$-valuation ring of $K$ iff $p\notin\mathcal O^\times$ and $\gamma(f) \in \mathcal O$ for all $f\in K$.
\end{lemma}

\noindent
By convention, for $R\subseteq K$ we set 
$\gamma(R):=\big\{\gamma(f): f\in R\big\}\setminus\{\infty\}$.
If $K$ is formally $p$-adic, then $\gamma(f)\neq\infty$ for each $f\in K$, by the previous lemma, so in this case $\gamma(R)$ indeed agrees with the image of the restriction of the map $\gamma$ to a map $R\to K_\infty$. For a proof of the following see \cite[Theorem~6.4]{PR}:

\begin{prop}\label{prop:6.4}
Let $(K,\mathcal O)$ be $p$-valued field and $L$ be a field extension of $K$.
Then there is a $p$-valuation ring  of $L$ lying over $\mathcal O$ iff $p\notin\mathcal O[\gamma(L)]^\times$.
\end{prop}

\noindent
Since $K$ is formally $p$-adic iff there is a $p$-valuation ring of $K$ which lies over the unique $p$-valuation ring 
$\Z_{(p)}$
of 
the subfield $\Q$ of $K$,
Proposition~\ref{prop:6.4}    implies:

\begin{cor}\label{cor:6.4}
 $K$ is formally $p$-adic $\Longleftrightarrow$~$p\notin \Z_{(p)}[\gamma(K)]^\times$.
\end{cor}

\noindent
The {\em $p$-adic Kochen ring}\/ of~$K$ is the subring 
$$\Lambda_K := \left\{ \frac{f}{1-pg}:f,g\in \Z[\gamma(K)],\ pg\neq 1 \right\}$$
of $K$. 
It is easy to see that
$$\Lambda_K = \left\{ \frac{f}{1-pg}:f,g\in \Z_{(p)}[\gamma(K)],\ pg\neq 1 \right\}.$$
If  $K$ is formally $p$-adic then by the previous corollary we can omit here the condition~``$pg\neq 1$''.
The next result is shown in \cite[Theorem~6.14]{PR}:

\begin{theorem}\label{thm:holomorphy}
$\Lambda_K = \bigcap \big\{\mathcal O: \text{$\mathcal O$ is a $p$-valuation ring of $K$} \big\}$.
\end{theorem}

\noindent
By Theorem~\ref{thm:holomorphy} and the  proof of Lemma~\ref{lem:p-euclidean}, with $\mathcal O_0$ as in Lemma~\ref{lem:def val}, we have~$\mathcal O_0=\Lambda_K$ whenever $K$ is $p$-euclidean. Also,
$K$ is formally $p$-adic iff $\Lambda_K\neq K$;
moreover:

\begin{cor}\label{cor:formally p-adic p nonunit}
$K$ is   formally $p$-adic $\Longleftrightarrow$  $p\notin\Lambda_K^\times$.
\end{cor}
\begin{proof}
Suppose $K$ is formally $p$-adic, and towards a contradiction assume $p\in\Lambda_K^\times$. 
Take $f,g\in\Z[\gamma(K)]$ with  $p^{-1}=f/(1-pg)$; then $p^{-1}=f+g\in\Z[\gamma(K)]$, contradicting
Corollary~\ref{cor:6.4}.
\end{proof}

\noindent
Theorem~\ref{thm:holomorphy} implies that $\Lambda_K$ is a Pr\"ufer domain (every localization of $\Lambda_K$ at a maximal ideal
of $\Lambda_K$ is a valuation ring of $K$); see \cite[Corollary~6.16]{PR}.
In fact, $\Lambda_K$ is even a B\'ezout domain (every finitely generated ideal of $\Lambda_K$ is principal) with fraction field $K$; cf.
 \cite[Theorem~6.17]{PR}. Hence every subring of $K$ containing $\Lambda_K$ is also a B\'ezout domain. 
In particular, given a subring $R$ of $K$, the subring
$$\Lambda_KR := \left\{ \frac{f}{1-pg}:f\in R[\gamma(K)],\ g\in \Z[\gamma(K)] , \ pg\neq 1 \right\}.$$ 
  of $K$ generated by~$\Lambda_K$ and~$R$ is a B\'ezout domain. (We will not use this fact below.)
  
\subsection{$p$-adic ideals and $p$-adic radical}
{\it In this subsection we let  $R$ be an integral domain with fraction field $K$.}\/ 
We say that an ideal $I$ of $R$ is   {\em $p$-adic}\/  
if 
$$I = \sqrt{I\,  \Lambda_KR} \cap R,$$
that is, if $I$ contains each $r\in R$ such that $r^k\in I\,\Lambda_KR$ for some $k\geq 1$.
(This notion was introduced in \cite{Srhir01} and further studied in \cite{FS, Srhir03}.)
If~$I$ is $p$-adic, then~$I$ is radical. 
The trivial ideals $I=\{0\}$ and $I=R$ of $R$ clearly are   $p$-adic.
This notion of $p$-adic ideal is mainly of interest if $K$ is formally $p$-adic: 
otherwise we have $\Lambda_K=K$, hence~$R$ has no non-trivial  $p$-adic ideals.

\begin{lemma}\label{lem:p-adic down}
Let $R^*$ be an integral domain extending $R$  such that
$R^*$ is  faithfully flat over $R$, and 
let $I$ be an ideal of $R$ such that the ideal $IR^*$ of $R^*$ is   $p$-adic. Then~$I$ is $p$-adic.
\end{lemma}

\begin{proof}
Let $f\in \sqrt{I\,  \Lambda_KR} \cap R$. This yields $k\geq 1$ with
$f^k\in I\,  \Lambda_KR$. Now with $K^*:=\Frac(R^*)$ we have $\Lambda_K\subseteq\Lambda_{K^*}$, hence
$f^k\in I\, \Lambda_{K^*}R^*$. Since $IR^*$ is $p$-adic we get~$f\in IR^*$ and thus $f\in I$ because $R^*$ is faithfully flat over $R$ \cite[Theorem~7.5(ii)]{Matsumura}.
\end{proof}

\begin{lemma}\label{lem:p-adic min primes}
If $R$ is noetherian, then all minimal prime ideals of~$R$ containing a given $p$-adic ideal of $R$ are $p$-adic.
\end{lemma}
  
\noindent
This follows from an easy ring-theoretic lemma, applied to $A=R$, $B=\Lambda_KR$:

\begin{lemma}\label{lem:rad min prime}
Let $A\subseteq B$ be a ring extension   where $A$ is noetherian. Let $I$ be an ideal of $A$ such that $I=\sqrt{IB}\cap A$. Then for each minimal prime ideal $P$ of $A$ containing $I$ we also have $P=\sqrt{PB}\cap A$.
\end{lemma}

\noindent
In the next lemma and its corollary we assume that  $R$ is  local with maximal ideal~$\fm$ and residue field~$\k:=R/\fm$ of characteristic zero,
and we let $r\mapsto \bar{r}\colon R\to \k$ be the residue morphism.  

\begin{lemma}
Let $\alpha\in\gamma(\k)$; then there is some $a\in\gamma(R)\cap R$ such that $\alpha=\overline{a}$.
\end{lemma}

\begin{proof}
Take $\beta\in\k$ such that $\alpha=\gamma(\beta)$, so 
$$\alpha p  (\delta^2-1) = \delta \qquad\text{where $\delta:=\wp(\beta)\in\k\setminus\{\pm 1\}$.}$$
Next let $a,b\in R$ with $\overline{a}=\alpha$, $\overline{b}=\beta$. Then $\overline{d}=\delta$ for $d:=\wp(b)$ and so $d\neq\pm 1$, and
$ap(d^2-1)-d\in\fm$. Now $p(d^2-1)\in R^\times$ and thus $a-\gamma(b)=a- \frac{d}{p(d^2-1)}\in\fm$. Hence $\gamma(b)\in R$ and $\alpha=\overline{\gamma(b)}$.
\end{proof}

\begin{cor} \label{cor:local p-adic}
Suppose $K$ is formally $p$-adic. If $\fm$ is $p$-adic, then $\k$ is formally $p$-adic.
\end{cor}
\begin{proof}
Suppose $\k$ is not formally $p$-adic. 
Corollary~\ref{cor:formally p-adic p nonunit}
then yields~$\alpha,\beta\in \Z[\gamma(\k)]$,
$p\beta\neq 1$, such that
$$p^{-1}=\frac{\alpha}{1-p\beta}\quad\text{and hence}\quad 1-p(\alpha+\beta)=0.$$
The  lemma above yields $a,b\in \Z[\gamma(K)] \cap R$ with $\alpha=\bar{a}$, $\beta=\bar{b}$.
We have $1\neq p(a+b)$ since $K$ is formally $p$-adic.
Thus~$1-p(a+b)\in \fm\cap\Lambda_K^\times$ and so $1\in \fm \Lambda_K R\cap R$.  
Hence $\fm$ is not $p$-adic. \end{proof}
 
\noindent
Next we let $P$ be a prime ideal of $R$ and $F$ be the fraction field of $R/P$.
We let 
$$R_P:=\big\{r/s:r,s\in R,\ s\notin P\big\}\subseteq K$$ be the localization of $R$ at $P$.
The residue morphism $R\to R/P$ extends uniquely to a surjective ring morphism $R_P\to F$ with kernel $PR_P$.

\begin{lemma}\label{lem:prime p-adic}
$P$ is $p$-adic iff the maximal ideal $PR_P$ of $R_P$ is also $p$-adic.
\end{lemma}
\begin{proof}
Suppose $P$ is $p$-adic.
Let $f\in R_P$ and $k\geq 1$ be such that
$f^k\in P\Lambda_K R_P$; we need to show that~$f\in PR_P$. Multiplying $f$ by a suitable unit of $R_P$ we arrange~$f\in R$.
Take~$a\in R\setminus P$ such that $af^k\in P\Lambda_K R$.
Then~$af\in \sqrt{P\Lambda_KR}\cap R=P$, thus~$f\in PR_P$ as claimed.
Conversely, if $PR_P$ is $p$-adic, and   $f\in R$, $k\geq 1$ with~$f^k\in P\Lambda_K R$, then $f\in PR_P\cap R=P$, showing that
$P$ is $p$-adic.
\end{proof}

\noindent
Combining Corollary~\ref{cor:local p-adic} and Lemma~\ref{lem:prime p-adic} yields:

\begin{cor}\label{cor:P p-adic => F form p-adic}
If $K$ is formally $p$-adic and $P$ is $p$-adic, then  $F$ is formally $p$-adic.
\end{cor}
 
\noindent
Here is a partial converse of the implication in
the previous corollary:

\begin{prop}\label{prop:p-adic prime ideals}
Suppose $R$ is regular and $F=\Frac(R/P)$ is formally $p$-adic. Then $K$ is formally $p$-adic and  $P$ is $p$-adic.
\end{prop}

\begin{proof}
The local integral domain $R_P$ is   regular  and its
  residue field is isomorphic to~$F$.
Hence localizing at $P$ we can arrange that~$R$ is local with maximal ideal~$P$,
by Lemma~\ref{lem:prime p-adic}.
Lemma~\ref{lem:existence of places} yields a valuation ring $\mathcal O$ of~$K$ lying over $R$
such that the natural inclusion $R\to\mathcal O$ induces an isomorphism~$F=R/P\to\k=\mathcal O/\smallo$; thus~$\k$ is formally $p$-adic. Let~$\mathcal O_{\k}$ be a $p$-valuation ring of $\k$ and
take a valuation ring~$\mathcal O_0$ of~$K$ with~$\mathcal O_0\subseteq\mathcal O$ and a convex subgroup~$\Delta_0$
of its value group such that the $\Delta_0$-specialization of $(K,\mathcal O_0)$ is~$(\k,\mathcal O_{\k})$. (See Lemma~\ref{lem:lift val spec}.)
Then $\mathcal O_0$ is a $p$-valuation ring of $K$ by Lemma~\ref{lem:p-val spec},
hence~$\Lambda_{K}\subseteq\mathcal O_0\subseteq\mathcal O$ by Theorem~\ref{thm:holomorphy}. This yields  $P\Lambda_KR \cap R \subseteq P\mathcal O\cap R=P$ and thus~$\sqrt{P\Lambda_KR}\cap R = P$  as required.
\end{proof}

\noindent
It would be interesting to weaken the regularity hypothesis in Proposition~\ref{prop:p-adic prime ideals}.
Let now~$\sigma\colon R\to L$ be a ring morphism. The previous proposition immediately implies:






\begin{cor}\label{cor:p-adic prime ideals}
If $R$  is regular and $L$ is  formally $p$-adic, then $K$ is formally $p$-adic and $\ker\sigma$ is a $p$-adic prime ideal of $R$.
\end{cor}

\noindent
We also note another artifact of the proof of Proposition~\ref{prop:p-adic prime ideals}.

\begin{lemma}\label{lem:p-adic prime ideals}
Suppose $R$ is regular, and let $f,g\in R$. Then $$f\in g\Lambda_K\ \Longrightarrow\ \sigma(f)\in \sigma(g)\Lambda_L.$$
\end{lemma}
\begin{proof}
We may arrange that $L=F=\Frac(R/P)$ and $\sigma$ is the composition of the residue morphism
$R\to R/P$ with the natural inclusion $R/P\to F$. Replacing~$R$,~$P$ by $R_P$, $PR_P$, respectively, we can also arrange that $R$ is local with maximal ideal~$P$ and residue field $L=R/P$, and $\sigma\colon R\to L$ is the residue morphism.
Suppose~$f\in g\Lambda_K$, and let $\mathcal O_L$ be a $p$-valuation ring of~$L$, with associated dominance relation~$\preceq_L$ on $L$; we need to show $\sigma(f)\preceq_L \sigma(g)$, by Theorem~\ref{thm:holomorphy}.
Take $\mathcal O$
as in the proof of Proposition~\ref{prop:p-adic prime ideals} and let $\mathcal O_{\k}$ be the image of~$\mathcal O_L$ under the isomorphism $F\to\k$. Next take $\mathcal O_0$ as in the proof of Proposition~\ref{prop:p-adic prime ideals},
so~$\mathcal O_0=\big\{h\in\mathcal O:\sigma(h)\preceq_L 1\big\}$. If~$g\in P$, then $f\in P$ since $P\Lambda_KR\cap R=P$, and hence clearly  $\sigma(f) \preceq_L \sigma(g)$. Suppose $g\notin P$; then $g\in R^\times$, hence we may replace~$f$,~$g$ by $f/g$, $1$ to arrange~$g=1$. Then $f\in\Lambda_K$ and so $f\in\mathcal O_0$ by Theorem~\ref{thm:holomorphy}, thus $\sigma(f)\preceq_L 1$.
\end{proof}

\noindent
Let $I$ be an ideal of $R$. We define the {\em $p$-radical of $I$} as
$$\sqrt[\leftroot{2}\uproot{5}p]{I} := \sqrt{I\Lambda_KR}\cap R=\big\{ r\in R: \text{$r^k\in I\Lambda_KR$ for some $k\geq 1$} \big\}.$$
Clearly $\sqrt[\leftroot{2}\uproot{5}p]{I}$ is the smallest $p$-adic ideal of $R$ containing $I$.
If $J$ is an ideal of $R$, then~$I\subseteq J\Rightarrow  \sqrt[\leftroot{2}\uproot{5}p]{I} \subseteq\sqrt[\leftroot{2}\uproot{5}p]{J}$.

\subsection{$p$-adic ideals under completion}
In this subsection we show a $p$-adic  analogue of \cite[Proposition~6.3]{Risler76}. Let  $W$ be a Weierstrass system over a formally $p$-adic field~$\k$.
Let $I$ be an ideal of~$R:=W_m=\k\lfloor X\rfloor$. If the ideal $I\hat R$ of $\hat R:=\k[[X]]$ generated by $I$
is $p$-adic, then so is $I$, by Corollary~\ref{cor:fflat} and Lemma~\ref{lem:p-adic down}. The converse also holds:

\begin{prop}\label{prop:IhatR}
Suppose $I$ is $p$-adic; then so is $I\hat R$.
\end{prop}
\begin{proof}
We may assume $I\neq R$.
Note that since we have a ring morphism
$f\mapsto f(0)\colon \hat R\to\k$ from the regular local ring $\hat R$ to the formally $p$-adic field $\k$,
the field~$\hat K:=\Frac(\hat R)$ is formally $p$-adic by Corollary~\ref{cor:p-adic prime ideals}. Let $P_1,\dots,P_k$ ($k\geq 1$) be the minimal prime ideals of $I$. Then $P_1,\dots,P_k$ are $p$-adic by Lemma~\ref{lem:p-adic min primes},
and~$I\hat R=P_1\hat R\cap\cdots\cap P_k\hat R$ by \cite[Theorem~7.4(ii)]{Matsumura}.
Hence it is enough to show the proposition in the case where $I$ is prime. In this case,
$I\hat R$ is then also prime,
by Lemma~\ref{lem:Nagata}. Since we have~$I\hat R\cap R=I$~\cite[Theorem~7.5(ii)]{Matsumura},
  we may view~$F:=\Frac(R/I)$ as a subfield of~$\hat F:=\Frac(\hat R/I\hat R)$.
Now by Corollary~\ref{cor:6.4}, $F$ is formally $p$-adic iff~$p^{-1}\notin \Z_{(p)}[\gamma(F)]$,
and similarly with~$\hat F$ in place of $F$.
By Corollary~\ref{cor:P p-adic => F form p-adic}, $F$ is formally $p$-adic, hence so is $\hat F$ by
the version of the Artin Approximation Theorem for $R\subseteq\hat R$ from~\cite[Theorem~1.1]{DL}. Thus $I\hat R$ is $p$-adic by Proposition~\ref{prop:p-adic prime ideals}.
\end{proof}

\begin{cor}\label{cor:IhatR}
$\sqrt[\leftroot{2}\uproot{5}p]{I\hat R}=\sqrt[\leftroot{2}\uproot{5}p]{I}\hat R$. 
\end{cor}

\subsection{The $p$-adic analytic Nullstellensatz}
{\it In this subsection we let   $\mathcal L=\mathcal L_{\operatorname{Mac},{\preceq}}$. We also let  $W$ be a Weierstrass system over a weakly $p$-euclidean field~$\k$, and we let~${K\models p\operatorname{CVF}_W}$.}\/ (Note that the infinitesimal $W$-structure on $K$ is with respect
to the distinguished $p$-convex valuation ring of $K$,   not with respect to the $p$-valuation ring of $K$.)
We also let~$X=(X_1,\dots,X_m)$ be a tuple of distinct indeterminates over~$\k$ and~$y=(y_1,\dots,y_n)$ be a tuple of distinct $\mathcal L$-variables.  

\begin{prop}\label{prop:p-adic spec}
For each $\mathcal L$-formula $\varphi(y)$ and each $f\in \k\lfloor X\rfloor^n$,
if $\Omega\models\varphi(\sigma(f))$ for some $\Omega\models p\operatorname{CVF}$ and local ring morphism~${\sigma \colon \k\lfloor X\rfloor\to \mathcal O_\Omega}$, then
$K\models\varphi(f(a))$ for some $a\in\smallo^m$.
\end{prop}
\begin{proof}
The proof is completely analogous to that of Proposition~\ref{prop:Risler spec}. Using Corollary~\ref{cor:Belair}  instead of Theorem~\ref{thm:CD} we   first arrange that $\varphi$ is a quantifier-free $\mathcal L^*$-formula.
If $k\geq 1$ and $\lambda_1,\dots,\lambda_N\in\Z$ ($N\geq 1$) are representatives for the cosets~$\neq \Q_p^\times$ of the subgroup
$(\Q_p^\times)^k$ of $\Q_p^\times$  (Lemma~\ref{lem:coset repres}), then for each   single variable $v$ we have
$$p\operatorname{CF}\models \neg \operatorname{P}_k(v) \leftrightarrow \operatorname{P}_k(\lambda_1 v) \vee \cdots \vee \operatorname{P}_k(\lambda_N v).$$
Hence we  can further arrange that~$\varphi$ has the form
$$P(y) = 0 \wedge \bigwedge_{i\in I} \operatorname{P}_{k_i}\!\big(Q_i(y)\big)  \wedge \bigwedge_{j\in J} R_j(y)\, \square_j\, S_j(y)$$
where $I$, $J$ are finite, $P,Q_i,R_j,S_j\in\Z[Y_1,\dots,Y_n]$,  $k_i\geq 1$, and each symbol~$\square_j$ is~$\preceq$ or $\prec$.
For each $k\geq 1$, every $1$-unit of $\k\lfloor X\rfloor$ is a $k$th power of a $1$-unit of  $\k\lfloor X\rfloor$, by Corollary~\ref{cor:roots}; using
this fact   instead of Lemma~\ref{lem:u>0} we now finish the argument as in the proof of Proposition~\ref{prop:Risler spec}.
\end{proof}

\noindent
As before, the  $p$-adic version of Corollary~\ref{cor:exists-complete, real} follows:

\begin{cor}\label{cor:exists-complete, p-adic}
For each   existential  $\mathcal L_W$-sentence $\theta$, we either  have
$p\operatorname{CVF}_W\models\theta$ or~$p\operatorname{CVF}_W\models\neg\theta$.
\end{cor}

\noindent
Next an analogue of Lemma~\ref{lem:indets infinitesimal}:

\begin{lemma}\label{lem:indets infinitesimal, p-adic}
Let $\Omega$ be a   $p$-valued  field and $\sigma\colon  \k\lfloor X\rfloor\to \Omega$ be a ring morphism.
Let $\preceq$ be the dominance relation on $\Omega$ whose valuation ring is the $p$-convex hull of~$\sigma(\k)$ in
$\Omega$. Then for each
$g\in  \k\lfloor X\rfloor$  we have $\sigma\big(g- g(0)\big) \prec 1$.
\end{lemma}
\begin{proof}
Replacing $g$ by $g-g(0)$ we arrange $g(0)=0$.
Then $1+pg^2$ is a square in~$\k\lfloor X\rfloor$, by Corollary~\ref{cor:roots};
hence $1+p\sigma(g)^2$ is a square in~$\Omega$, so necessarily $\sigma(g)\preceq_p 1$.  Since this also holds for $gc^{-1}$ in place of $g$, for each~$c\in\k^\times$,
we get~$\sigma(g) \preceq_p \sigma(c)$ for all $c\in\k^\times$. 
Thus $\sigma(g)\prec 1$ by Lemma~\ref{lem:convex hull}.
\end{proof}

\begin{cor}\label{cor:p-adic spec}
Let $\varphi(y)$ be an $\mathcal L_{\operatorname{Mac}}$-formula, $f\in  \k\lfloor X\rfloor^n$, and
$\sigma\colon  \k\lfloor X\rfloor\to\Omega$ be a ring morphism to a  $p$-adically closed field $\Omega$
such that~$\Omega\models\varphi(\sigma(f))$. Then~$K\models\varphi(f(a))$ for some $a\in\smallo^m$.
\end{cor}
\begin{proof}
The $p$-convex hull of $\sigma(\k)$ in the $p$-adically closed field extension
$\Omega(\!(t^\Q)\!)$ of~$\Omega$ is a non-trivial $p$-convex subring of $K$, by Lemma~\ref{lem:p-convex hull}.
Hence  using model completeness of
the  $\mathcal L_{\operatorname{Mac}}$-theory $p\operatorname{CF}$ we may replace $\Omega$ by $\Omega(\!(t^\Q)\!)$ 
to arrange that  the $p$-convex hull
 $\mathcal O_\Omega$   of $\sigma(\k)$ in $\Omega$ satisfies $\Omega\neq\mathcal O_\Omega$. 
  Then $\Omega$ equipped with the dominance relation associated to $\mathcal O_\Omega$
is a model of $p\operatorname{CVF}$.
 By Lemma~\ref{lem:indets infinitesimal, p-adic},
$\sigma$ is a local ring morphism  $\k\lfloor X\rfloor\to\mathcal O_\Omega$.
Hence the corollary follows from Proposition~\ref{prop:p-adic spec}.
\end{proof}

\noindent
Recall that the  ring~$\k\lfloor X\rfloor$ is regular (Lemma~\ref{lem:W-Lemma}),
hence $F:=\Frac(\k\lfloor X\rfloor)$ is formally $p$-adic.
We  obtain an analytic version of Kochen's $p$-adic formulation of Hilbert's 17th Problem.

\begin{cor}\label{cor:p-adic H17}
$$\Lambda_F = \left\{ \frac{f}{g}:f,g\in \k\lfloor X\rfloor,\ g\neq 0,\ f(a) \preceq_p g(a) \text{ for all $a\in\smallo^m$} \right\}.$$
\end{cor}
\begin{proof}
Let $f,g\in \k\lfloor X\rfloor$, $g\neq 0$, with $f/g\notin \Lambda_F$. By
Theorem~\ref{thm:holomorphy} take a $p$-valuation ring $\mathcal O_p$ of $F$ with $f/g\notin\mathcal O_p$.
Let $\Omega$ be a $p$-adically closed valued field extension of~$(F,\mathcal O_p)$, with natural inclusion $\sigma\colon \k\lfloor X\rfloor\to\Omega$. Then Corollary~\ref{cor:p-adic spec} yields some~$a\in\smallo^m$ with $f(a) \not\preceq_p g(a)$. This shows the inclusion $\supseteq$.
The reverse inclusion  follows from Lemma~\ref{lem:p-adic prime ideals}.
\end{proof}

\noindent
Next we prove the $p$-adic analytic Nullstellensatz. First, a preliminary observation:
 
\begin{lemma}\label{lem:zeroes p-adic radical}
Let $I$ be an ideal of $\k\lfloor X\rfloor$. Then $\operatorname{Z}_K(I) = \operatorname{Z}_K(\sqrt[\leftroot{2}\uproot{5}p]{I})$,
and if $\operatorname{Z}_K(I)\neq\emptyset$ then some $p$-adic prime ideal of $\k\lfloor X\rfloor$ contains $I$.
\end{lemma}
\begin{proof} For each $a\in\smallo^m$, the
  kernel of the ring morphism $f\mapsto f(a)\colon \k\lfloor X\rfloor \to K$ is a $p$-adic prime ideal $P_a$ of $\k\lfloor X\rfloor$,  by Corollary~\ref{cor:p-adic prime ideals}.
Let $f\in\sqrt[\leftroot{2}\uproot{5}p]{I}$ and
  $a\in\operatorname{Z}_K(I)$. 
  Then $P_a$   contains~$I$ and thus
also $\sqrt[\leftroot{2}\uproot{5}p]{I}$; in particular $f\in P_a$ and so $f(a)=0$.
\end{proof}

\begin{theorem}[$p$-adic analytic Nullstellensatz]\label{thm:p-adic NSS}
Let $I$ be an ideal of $\k\lfloor X\rfloor$. Then
$$\operatorname{I}\!\big(\!\operatorname{Z}_K(I)\big)=\sqrt[\leftroot{2}\uproot{5}p]{I}.$$
\end{theorem}
\begin{proof}
If no $p$-adic prime ideal of  $\k\lfloor X\rfloor$ contains $I$, then 
 $\operatorname{Z}_K(I)=\emptyset$ by Lemma~\ref{lem:zeroes p-adic radical} and thus
 $\operatorname{I}\!\big(\!\operatorname{Z}_K(I)\big)=\k\lfloor X\rfloor=
\sqrt[\leftroot{2}\uproot{5}p]{I}$
by Lemma~\ref{lem:p-adic min primes}.
Hence we can assume
that some $p$-adic prime ideal of $\k\lfloor X\rfloor$ contains $I$.
Lemma~\ref{lem:p-adic min primes} then yields $p$-adic prime ideals~$P_1,\dots,P_k$ ($k\geq 1$) of $\k\lfloor X\rfloor$
with
$\sqrt[\leftroot{2}\uproot{5}p]{I}=P_1\cap\cdots\cap P_k$.
Using Lemma~\ref{lem:zeroes p-adic radical} we     arrange that   $I$ itself is a $p$-adic prime ideal;
we need to show that then~$\operatorname{I}\!\big(\!\operatorname{Z}_K(I)\big)=I$. 
The fraction field $F$ of~$R:=\k\lfloor X\rfloor/I$ 
is formally $p$-adic, by Corollary~\ref{cor:P p-adic => F form p-adic}.
Equip  $F$   with some  $p$-valuation,
let $\Omega$ be a $p$-adically closed extension of this $p$-valued field, and let $\sigma\colon \k\lfloor X\rfloor\to \Omega$ be the
composition of~$f\mapsto f+I\colon\k\lfloor X\rfloor\to R$ with the natural inclusion~$R\subseteq \Omega$.
Arguing as in the proof of Theorem~\ref{thm:Rueckert}, using Corollary~\ref{cor:p-adic spec} instead of
Corollary~\ref{cor:spec},  then yields~$\operatorname{I}\!\big(\!\operatorname{Z}_K(I)\big)\subseteq I$.
\end{proof}

\noindent
Recall that if $\k$ is $p$-adically closed, then the $t$-adically valued field $\k(\!(t^*)\!)$ is a model of $p\operatorname{CVF}$. Hence in a similar way as Theorem~\ref{thm:Rueckert} implied Corollary~\ref{cor:Rueckert, 1} or
  Theorem~\ref{thm:Risler} implied Corollary~\ref{cor:Risler, 1}, from Theorem~\ref{thm:p-adic NSS} we get the Nullstellensatz for formal power series over $p$-adically closed fields (Theorem~C from the introduction):
  
\begin{cor} \label{cor:p-adic NSS, 1}
Suppose $\k$ is $p$-adically closed, and let $R:=\k [[ X ]]$ and~$F:=\Frac(\k [[ X ]])$. 
Let $f, g_1,\dots,g_n\in \k[[X]]$ be such that for all $a\in t\k[[t]]^m$:
$$ g_1(a)=\cdots=g_n(a)=0 \quad\Longrightarrow\quad f(a)=0.$$ Then there are $k\geq 1$, $g\in\Z[\gamma(F)]$, and~$h_1,\dots,h_n\in R[\gamma(F)]$,    such that 
$$f^{k}(1-pg)=g_1h_1+\cdots+g_nh_n.$$ 
\end{cor}
 
\begin{remarks} \mbox{}
Let $\k$, $R$, $F$ be as in the previous corollary.
\begin{enumerate}
\item Let $\k_0$ be a weakly $p$-euclidean subfield of $\k$ (such as $\k_0=\Q$).
If $f$ and each~$g_j$ 
in Corollary~\ref{cor:p-adic NSS, 1}
are in $\k_0[[X]]$, 
then we can take $g\in\Z[\gamma(F_0)]$ and~$h_j\in R_0[\gamma(F_0)]$, where $R_0:= \k_0[[X]]$ and $F_0:=\Frac(R_0)$.
Similarly with other Weierstrass systems over $\k_0$  in place of $\k_0[[X]]$, like
$\k_0[[X]]^{\operatorname{a}}$ or~$\k_0[[X]]^{\operatorname{da}}$.
\item With $\operatorname{Hom}_{\k}(\k[[X]],\k[[t]])$ denoting the set of $\k$-algebra morphisms $\k[[X]]\to\k[[t]]$,
 Corollary~\ref{cor:p-adic NSS, 1} can also be phrased succinctly as follows: for each ideal~$I$ of~$\k[[X]]$,
$$\sqrt[\leftroot{2}\uproot{5}p]{I} = \bigcap \big\{ \ker \lambda: \lambda\in\operatorname{Hom}_{\k}(\k[[X]],\k[[t]]),\ I\subseteq\ker\lambda\big\}.$$
\end{enumerate}
\end{remarks}

\noindent
In the rest of this subsection we let  $A:=\Q_p \{ X \}$,~$F:=\Frac(\Q_p \{ X \})$.  
Combining Corollaries~\ref{cor:IhatR} and \ref{cor:p-adic NSS, 1} yields:

\begin{cor}
Let   $g_1,\dots,g_n\in A$ and 
$$Z:=\big\{ a\in t\Q_p[[t]]^m:g_1(a)=\cdots=g_n(a)=0\big\}.$$ 
Then the ideal of all $f\in\Q_p[[X]]$ such that $f(a)=0$ for all $a\in Z$
is generated by   power series in $A=\Q_p\{X\}$.
\end{cor}

\noindent
Since the $t$-adically valued field $\Q_p\{\!\{t^*\}\!\}$ is a model of $p\operatorname{CVF}$, arguing as in the proof of Corollary~\ref{cor:Rueckert, 1} we also obtain:

\begin{cor}  \label{cor:p-adic NSS, 2}
Let 
$f,g_1,\dots,g_n\in A$, and suppose    
 that for each $a\in \Q_p\{t\}^m$ with~$a(0)=0$ we have $$ g_1(a)=\cdots=g_n(a)=0 \quad\Longrightarrow\quad f(a)=0.$$ 
Then
  there are~$k\geq 1$, $g\in\Z[\gamma(F)]$,~$h_1,\dots,h_n\in A[\gamma(F)]$,    such that 
$$f^{k}(1-pg)=g_1h_1+\cdots+g_nh_n.$$ 
\end{cor}

\noindent
In the same way that Corollary~\ref{cor:Rueckert, 1} gave rise to Corollary~\ref{cor:Rueckert fns},
the last corollary also yields a version for $p$-adic analytic functions near $0$; this is Theorem~A in the introduction:

\begin{cor}\label{cor:p-adic NSS fns}
Let $f,g_1,\dots,g_n\colon U\to\Q_p$ be  analytic functions, where $U$ is an open neighborhood of $0$ in $\Q_p^m$, such that
for all $a\in U$:
$$ g_1(a)=\cdots=g_n(a)=0 \quad\Longrightarrow\quad f(a)=0.$$  
Then there are $k\geq 1$, $g\in\Z[\gamma(F)]$, $h_1,\dots,h_n\in A[\gamma(F)]$, such that in $F$,
with the germs of~$f,g_1,\dots,g_n$ at $0$ denoted by the same symbols:
\begin{equation}\label{eq:p-adic NSS fns}
f^{k}(1-pg)=g_1h_1+\cdots+g_nh_n.
\end{equation}
\end{cor}

\begin{remark}  
Let   $f,g_1,\dots,g_n\colon U\to\Q_p$ be as in Corollary~\ref{cor:p-adic NSS fns}.
If $r,s\colon U\to\Q_p$ are analytic functions and
 $Z:=s^{-1}(0)$,
 then $\gamma\big(r(a)/s(a)\big)\in\Z_p$ for each~$a\in U\setminus Z$, by Lemma~\ref{lem:p-val gamma}, so we have a   continuous  function
 $$\gamma(r/s)\colon U\setminus Z \to\Z_p\quad\text{with $\gamma(r/s)(a) = \gamma\big(r(a)/s(a)\big)$ for $a\in U\setminus Z$.}$$
Hence if $k\geq 1$, $g\in\Z[\gamma(F)]$, and $h_1,\dots,h_n\in A[\gamma(F)]$ are as in
  the previous corollary, then   $g,h_1,\dots,h_n$
give rise to functions $g,h_1,\dots,h_n\colon U\setminus Z\to\Q_p$, where~$Z=s^{-1}(0)$   for some analytic $s\colon U\to\Q_p$ with non-zero germ at $0$,
such that the identity~\eqref{eq:p-adic NSS fns} above holds in the the ring of continuous functions  $U\setminus Z\to\Q_p$.
\end{remark}

\noindent
Finally, we obtain Theorem~B, a $p$-adic analytic version of Hilbert's 17th Problem:

\begin{cor}\label{cor:p-adic H17 fns}
Let $f,g\colon U\to\Q_p$ be  analytic functions, where $U$ is an open neighborhood of $0$ in $\Q_p^m$, such that
$\abs{f(a)}_p \leq \abs{g(a)}_p$ for all $a\in U$.
Then, with the germs of~$f$,~$g$ at~$0$ denoted by the same symbols, we have $f\in g\Lambda_F$.
\end{cor}
\begin{proof}
If $g=0$,  then also $f=0$, hence we may assume $g\neq 0$.
The hypothesis yields $\abs{f(a)}_p \leq \abs{g(a)}_p$ for all $a\in t\Q_p\{t\}^m$  and so
$\abs{f(a)}_p \leq \abs{g(a)}_p$ for all $a\in\smallo^m$ where $\smallo=\bigcup_{d\geq 1}t^{1/d}\Q_p\{\!\{t^{1/d}\}\!\}$, the maximal ideal of $\Q_p\{\!\{t^*\}\!\}$. Hence $f/g\in\Lambda_F$ by Corollary~\ref{cor:p-adic H17}.
\end{proof}

\noindent
In the next section we show how to improve the preceding   corollaries \ref{cor:p-adic NSS fns} and~\ref{cor:p-adic H17 fns} by replacing     $A=\Q_p\{X\}$   above
  with the ring $\Q_p\langle X\rangle$ of restricted power series.

\section{Nullstellens\"atze for Restricted Power Series} \label{sec:rest}

\noindent
{\it In this section we let $\k$ be a field and  
$\abs{\,\cdot\,}\colon\k\to\R^{\geq}$ be a complete ultrametric absolute value
on~$\k$.}\/ We
let $R:=\{a\in\k:\abs{a}\leq 1\}$ and~$\fm:={\{a\in\k:\abs{a}<1\}}$ be the valuation ring
of $\abs{\,\cdot\,}$  and its maximal ideal, respectively. The residue field of~$R$ is denoted by
$\bar{R}:=R/\fm$, with residue morphism
$a\mapsto\bar{a}\colon R\to\bar{R}$.  Whenever appropriate we view $\k$ as an $\mathcal L_{\preceq}$-structure
by equipping it with the dominance relation $\preceq$ associated to $R$: $a\preceq b:\Leftrightarrow \abs{a}\leq\abs{b}$,
for $a,b\in\k$. 

\subsection{Restricted power series}
The set of all formal power series
$$f = \sum_\alpha f_\alpha X^\alpha \in \k[[X]] \quad \text{such that  
    $\abs{f_\alpha}\to 0$ as $\abs{\alpha}\to\infty$}$$
forms a $\k[X]$-subalgebra of
$\k[[X]]=\k[[X_1,\dots,X_m]]$, called the ring of {\it
  restricted power series} with coefficients in $\k$.  Here, as earlier, $\alpha=(\alpha_1,\dots,\alpha_m)$ ranges over all multiindices in~$\N^m$, and~$\abs{\alpha}=\alpha_1+\cdots+\alpha_m$.
 The Gauss norm on
$\k[X]$ extends to an ultrametric absolute value on the domain $\kX$  by setting
$$\abs{f}:=\max_\alpha\ \abs{f_\alpha} \qquad\text{for   $f\neq 0$ as above.}$$
(See \cite[\S{}1.5.3, Corollary~2]{BGR}.) The extension of this ultrametric absolute value on~$\kX$ 
to an  ultrametric absolute value on $\Frac(\kX)$ is also denoted by $\abs{\,\cdot\,}$.

The   $f\in\kX$ with $\abs{f}\leq 1$ form an $R[X]$-sub\-al\-ge\-bra~$\RX$ of
$\kX$, the ring of restricted power series with coefficients in~$R$.
This is the completion of its subring~$R[X]$ with respect to the $\fm R[X]$-adic topology on $R[X]$. (Cf., e.g., \cite[\S{}3]{GS}.)
The natural inclusion $R[X]\to\RX$ induces an embedding $\overline{R}[X]\to \RX/\frak m\RX$, via which we
identify $\overline{R}[X]$ with a subring of~$\RX/\frak m\RX$.
For each $f\in\RX$ there is some~$d\in\N$ such that $\abs{f_\alpha}<1$ for~$\abs{\alpha}>d$; hence~$\overline{R}[X] = \RX/\frak m\RX$. We  denote the image of $f$ under the natural
surjection~$\RX\to \bar{R}[X]$ by $\bar{f}$.

\subsection{Substitution}
Let $Y=(Y_1,\dots,Y_n)$ be a tuple of distinct indeterminates and~$g=(g_1,\dots,g_m)\in R\langle Y\rangle^m$. We set $g^\alpha:=g_1^{\alpha_1}\cdots g_m^{\alpha_m}\in R\langle Y\rangle$  for each~$\alpha$. Then for $f\in\k\langle X\rangle$ as above we have $f_\alpha g^\alpha\to 0$ as $\abs{\alpha}\to\infty$, hence we may define
\begin{equation}\label{eq:comp}
f(g) = f(g_1,\dots,g_m) := \sum_\alpha f_\alpha g^\alpha\in\k\langle Y\rangle.
\end{equation}
The map $f\mapsto f(g)$ is the unique $\k$-algebra morphism $\kX\to\k\langle Y\rangle$ with $X_i\mapsto g_i$ for $i=1,\dots,m$;
it restricts to an $R$-algebra morphism $\RX\to R\langle Y\rangle$.
In particular, for $a\in R^m$ we have a $\k$-algebra morphism $f\mapsto f(a)\colon\k\langle X\rangle\to \k$, which restricts to
an $R$-algebra morphism $f\mapsto f(a)\colon \RX\to R$. 
Let $f\in\k\langle X\rangle$; then~$\abs{f(a)} \leq \abs{f}$ for each  $a\in R^m$, and moreover:

\begin{lemma}\label{lem:max principle}
Suppose $\bar{R}$ is infinite; then  
$\abs{f}=\max_{a\in R^m}  \abs{f(a)}$.
\end{lemma}
\begin{proof}
This is clear if $f=0$, so suppose $f\neq 0$.
Taking $b\in\k$ with $\abs{b}=\abs{f}$ and replacing $f$ by $b^{-1}f$ we arrange $\abs{f}=1$, so $\overline{f}\in\overline{R}[X]$ is
non-zero. Since $\bar{R}$ is infinite, we obtain an $a\in R^m$ such that $\overline{f(a)}=\overline{f}(\overline{a})\neq 0$, so $\abs{f(a)}=1=\abs{f}$.
\end{proof}

\noindent
Hence if $\bar{R}$ is infinite and $f,g\in\kX$ satisfy $\abs{f(a)}\leq\abs{g(a)}$ for every $a\in R^m$, then~$\abs{f}\leq\abs{g}$; similarly with $<$ in place of $\leq$.

\subsection{Weierstrass Division in $\kX$} This works a bit differently   than in $\k[[X]]$. To explain this,
suppose $m\geq 1$ and let $X'=(X_1,\dots,X_{m-1})$.
Every   $f\in\RX$
can be written uniquely in the form 
\begin{equation}\label{eq:f as power series over DXprime}
f=\sum_{i=0}^\infty f_iX_m^i
\quad\text{with $f_i(X')\in\RXprime$ for all $i\in\N$,}
\end{equation}
where the infinite sum
converges with respect to the  Gauss norm on $\kX$. An element 
$f$ of $\RX$, expressed as in \eqref{eq:f as power series over DXprime},  is called {\it regular in $X_m$ of degree $d\in\N$}
if~$c:=\bar{f_d}\in \overline{R}{}^\times$ and 
$\overline{f_i}=0$ for all $i>d$, so
$c^{-1}\overline{f}\in\overline{R}[X]$ is monic in $X_m$ of degree~$d$.
If~$f=f_1\cdots f_n$ where
$f_1,\dots,f_n\in \RX$, then $f$ is regular in $X_m$ of degree $d$ iff   $f_i$ is regular in $X_m$ of degree $d_i$ for $i=1,\dots,n$,
with $d_1+\cdots+d_n=d$.
Let~$d\in\N$, $d\geq 1$. The $\k$-algebra automorphism $\tau_d$ of $\k[[X]]$ with
$X_i \mapsto X_i+X_m^{d^{m-i}}$ for~$1\leq i<m$ and $X_m \mapsto X_m$
then restricts to an $R$-algebra automorphism of~$\RX$.
  
\begin{lemma}\label{NoetherZpX}
Let $d>1$ and  $f\in\RX$ be such that $\bar{f}\in\bar{R}[X]$ is non-zero of 
degree~$<d$. 
Then for some $e<d^m$ and some $u\in R^\times$,
$$\tau_d(f)\equiv uX_m^e + \text{terms of lower degree in $X_m$} \mod \fm\RX.$$
\textup{(}In particular,  
$\tau_d(f)$ is regular in $X_m$ of degree $<d^m$.\textup{)}
\end{lemma}

\noindent
For a proof of this and the following  standard facts  see, e.g.,~\cite{BGR}.
Here now is the fundamental property of  $\RX$:

\begin{theorem}[Weierstrass Division Theorem for $\RX$]\label{WD}
Let $g\in\RX$ be regular in $X_m$ of degree $d$. Then for each $f\in\RX$
there are uniquely determined elements~$q\in \RX$ and $r\in\RXprime[X_m]$
with $\deg_{X_m} r<d$ such that $f=qg+r$.
\end{theorem}

\noindent
Applying this theorem with $f=X_m^d$, we obtain:

\begin{cor}[Weierstrass Preparation Theorem for $\RX$]\label{WP}
Let $g\in\RX$ be regular in $X_m$ of degree $d$. There are a unique
monic polynomial   $w\in\RXprime[X_m]$ of degree $d$ and a unique  
$u\in \RX^\times$ such that $g=u  w$.
\end{cor}

\noindent
We remark here that the group of units of $\RX$ is  
$$
\RX^\times=R^\times\big(1+\fm\RX\big)=\big\{ f\in \RX :\bar{f}\in \bar{R}[X]^\times=\bar{R}{}^\times \big\},
$$
hence $\kX^\times = \k^\times\big(1+\fm\RX\big)$.
The elements of the subgroup $1+\fm\RX$ of~$\RX^\times$ are called the {\it $1$-units}\/ of $\RX$.
For each $1$-unit $u$ of $\RX$ and  $a\in R^m$ we have $u(a)\sim 1$ (in the valued field $\k$).

\medskip
\noindent 
A non-zero element $f\in\kX$ is called {\em regular in $X_m$ of degree
  $d\in\N$} if there exists some $a\in \k$ such that $af\in \RX$ and $af$
is regular in $X_m$ of degree $d$ as defined above. (Note that then necessarily $\abs{a}=\abs{f}^{-1}$.) 
We denote the  unique extension of the  $R$-algebra automorphism $\tau_d$ of $\RX$  from Lem\-ma~\ref{NoetherZpX}
to a $\k$-algebra automorphism of $\kX$ also  by~$\tau_d$. By that lemma, for each non-zero $f\in\kX$  there is some   $d\in\N$, $d>1$, such that $\tau_d(f)$ is regular in $X_m$ of degree $<d^m$.

\medskip
\noindent 
From Theorem~\ref{WD} and Corollary~\ref{WP} we get:

\begin{cor}\label{WDforFX}\textup{(Weierstrass Division and Preparation 
      for $\kX$.)}  Let~$g\in\kX$ be regular in $X_m$ of degree
  $d$. Then every $f\in\kX$ can be uniquely written as~$f=qg+r$ with
  $q\in\kX$ and $r\in \k\langle X'\rangle[X_m]$, $\deg_{X_m} r < d$. In
  particular, there are a unique monic polynomial
  $w\in\RXprime[X_m]$ of degree $d$ and a unique   $u\in\kX^\times$ such
  that $g=u  w$.
\end{cor}

\noindent
Corollary~\ref{WDforFX} and the remark preceding it imply that $\kX$ is
noetherian. 
(If~$R$ is a discrete valuation ring, then $\RX$ is
also noetherian.) 
Euclidean Division in the polynomial ring~$\kXprime[X_m]$ and
the uniqueness part of Weierstrass Division  also  have another useful consequence:

\begin{cor}\label{cor:ED}
Let $w\in \RXprime[X_m]$ be monic; then the inclusion $\kXprime[X_m]\subseteq\kX$ induces a $\kXprime$-algebra isomorphism
$$\kXprime[X_m]/w \kXprime[X_m] \xrightarrow{\ \cong\ } \kX/w\kX.$$
\end{cor}

\noindent
For future use we  note an application of the previous corollary, where we assume~${m=1}$ and write $X=X_1$.

\begin{cor}\label{cor:div poly}
Let $f,g\in \k[X]$ with $f\in g\RX$. Then $f\in gR[X](1+\fm R[X])^{-1}$.
\end{cor}
\begin{proof}
This is clear if $g=0$, so we may assume $g\neq 0$.  Weierstrass Preparation then yields
$u\in R^\times (1+\fm\RX)$ and a monic $w\in R[X]$ such that $g=uw$. Then~$u\in R[X]$ by Corollary~\ref{cor:ED}.
Let $h\in\RX$ with $f=gh$. 
Then Corollary~\ref{cor:ED} also yields~$uh\in R[X]$.
Thus $h\in R[X](1+\fm R[X])^{-1}$ as required.
\end{proof}

\subsection{Restricted Weierstrass systems}
We now adapt Definition~\ref{def:DL} to the setting of restricted power series:

\begin{definition}\label{def:restricted W-system}
 A {\em restricted Weierstrass system over $\k$}   
is a family of rings~$(W_m)$ 
such that for all $m$ we have, with $X=(X_1,\dots,X_m)$:
\begin{itemize}
\item[(RW1)] $\k[X] \subseteq  W_m \subseteq \kX$;
\item[(RW2)] for each
permutation $\sigma$ of $\{1,\dots,m\}$,   the $\k$-algebra automorphism
$f(X) \mapsto  f(X_{\sigma(1)},\dots,X_{\sigma(m)})$ of $\kX$ maps $W_m$ onto itself;
\item[(RW3)] 
$W_m \cap \k\langle X'\rangle = W_{m-1}$ for $X'=(X_1,\dots,X_{m-1})$, $m\geq 1$;
\item[(RW4)] if $g\in W_m$ is regular of degree $d$ in $X_m$ ($m\geq 1$),
 then for every $f\in 
W_m$ there are~$q\in W_m$ and
  $r\in W_{m-1}[X_{m}]$ of degree $<d$ (in $X_{m}$) with
$f=qg+r$.  
\end{itemize}
\end{definition}

\noindent
Clearly the system 
$\bigl(\kX\bigr)$
consisting of all restricted power series rings with coefficients in $\k$ is a restricted Weierstrass system over $\k$.
In \cite{ideal4} it is shown that when $\operatorname{char}\k=0$, then
the rings~$\kX^{\operatorname{a}} := \kX\cap \k[[X]]^{\operatorname{a}}$ also
form a restricted Weierstrass system over $\k$.

\medskip
\noindent
In the rest of this section we let $W=(W_m)$ be a restricted Weierstrass system over~$\k$.
For an arbitrary tuple~$Y=(Y_1,\dots,Y_m)$   of distinct indeterminates,   we
let 
$$\k\lfloor Y\rfloor:= \big\{f\in \k[[ Y ]]: f(X)\in W_m\big\} \subseteq \k\langle Y\rangle, \qquad R\lfloor Y\rfloor:=R\langle Y\rangle\cap \k\lfloor Y\rfloor.$$
In the following we let  $Y=(Y_1,\dots,Y_n)$ be a tuple of distinct indeterminates disjoint from~$X=(X_1,\dots,X_m)$.
The elements of the subgroup $1+\fm R\lfloor X\rfloor$ of~$R\lfloor X\rfloor$ are called {\em $1$-units}\/ of
$\k\lfloor X\rfloor$. We have ${W_m \cap \kX^\times =W_m^\times}$. (Similar argument as in  the proof of Lemma~\ref{lem:units}.) Also using (RW1) this yields~$R\lfloor X\rfloor^\times=R^\times({1+\fm R\lfloor X\rfloor})$, 
 so~$\k\lfloor X\rfloor^\times=\k^\times(1+\fm R\lfloor X\rfloor)$.

\medskip
\noindent
As in the proof of Lemma~\ref{lem:DL subst} one shows, using (RW4) instead of (W4):

\begin{lemma}\label{lem:DL rest subst}
Let $f\in \k\lfloor X,Y\rfloor$ and $g \in R\lfloor X\rfloor^n$.
Then $f(X,g(X))\in \k\lfloor X\rfloor$.
\end{lemma}
 
\noindent
Using (RW2), (RW3) in place of (W2), (W3), respectively, this yields:
 
\begin{cor}\label{cor:DL rest subst}
For all $f\in \k\lfloor Y\rfloor$ and $g\in R\lfloor X\rfloor^n$ we have
$f(g(X))\in \k\lfloor X\rfloor$.
\end{cor}
 
\noindent
In particular, the $\k$-algebra automorphism $\tau_d$ of $\kX$  maps the subrings~$\k\lfloor X\rfloor$ and~$R\lfloor X\rfloor$ of~$\kX$ onto themselves.

\begin{lemma}\label{lem:W-Lemma, rest}
The integral domain
$\k\lfloor X\rfloor$ is   noetherian and factorial.
Moreover, suppose $R$ is a DVR;  then
$R\lfloor X\rfloor$ is also  noetherian and factorial,
 $\RX$ is faithfully flat over $R\lfloor X\rfloor$, and thus $R\lfloor X\rfloor$ is regular.
\end{lemma}

\begin{proof}
The noetherianity and factoriality statements follow as in  \cite[\S\S{}5.2.5, 5.2.6]{BGR} using Lemma~\ref{NoetherZpX} and~(RW4).
Suppose $R$ is a DVR. Then $R$ is regular, hence so is~$\RX$, by~\cite[Theorem~8.10]{GS}.
By \cite[Theorem~4.9]{GS}, $\RX$ is faithfully flat over~$R\lfloor X\rfloor$;
thus $R\lfloor X\rfloor$ is also regular by \cite[Theorem~8.7]{GS}.
 \end{proof}

\begin{remark}
The ring $\kX$ is regular; cf., e.g., \cite[\S{}7.1.1, Proposition~3]{BGR}.
We will not try to prove here that $\k\lfloor X\rfloor$ is regular in general (though this seems plausible).
For this it would be enough to see that  $\kX$ is faithfully flat over $\k\lfloor X\rfloor$; this can probably be shown
using Hermann's method in $\kX$ as described in \cite{ideal}.
\end{remark}

\noindent
As usual, (RW4) implies: 

\begin{lemma}[Weierstrass Preparation]\label{lem:WP rest}
Suppose $m\geq 1$.
Let $g\in  \k\lfloor X\rfloor$ be regular in $X_m$ of order~$d$, and set $X':=(X_1,\dots,X_{m-1})$. Then there are a unit $u$ of $\k\lfloor X\rfloor$
and a polynomial
$$w=X_m^d+w_1 X_m^{d-1} + \cdots + w_d\qquad\text{where $w_1,\dots,w_d\in R\lfloor X'\rfloor$}$$
such that $g=uw$.
\end{lemma}

\noindent
Consider now pairs $(A,M)$ where $A$ is a ring and $M$ is a proper ideal of $A$.
Let~$Y$ be a single indeterminate over $A$.
A {\em henselian polynomial}\/ in $(A,M)$ is one of the form~$1+Y+c_2Y^2+\cdots+c_dY^d$
where $c_2,\dots,c_d\in M$, $d\geq 2$.
 We say that $(A,M)$ is {\em henselian}\/  if every henselian polynomial in $(A,M)$ has a zero in $A$.  
 A valued field~$K$   is henselian iff the pair $(\mathcal O,\smallo)$ is  henselian.
If $A$ is complete and separated in its $M$-adic topology (i.e., the natural morphism $A\to \lim\limits_{\longleftarrow} A/M^n$ is an isomorphism), then $A$ is henselian with respect to~$M$ 
(Hensel's Lemma). See, e.g.,~\cite[(2.9)]{vdD81} for a proof of this fact;
we also recall from \cite[(2.10)]{vdD81}:

\begin{lemma}\label{lem:hens=>newt}
Suppose $(A,M)$ is henselian. Then for   
each $P\in A[Y]$, $z\in A$ and~$h\in M$ such that
$P(z) = h P'(z)^2$,  there is some $y\in A$ such that $P(y)=0$ and~$y\equiv z\bmod P'(z)hA$.
\end{lemma}

\noindent
We now show:

\begin{lemma}
The pair $\big(R\lfloor X\rfloor, \fm R\lfloor X\rfloor\big)$ is henselian.
\end{lemma}
\begin{proof}
Let $P(X,Y)\in R\lfloor X\rfloor[Y]$ be henselian in $\big(R\lfloor X\rfloor, \fm R\lfloor X\rfloor\big)$; we need to find a $y \in R\lfloor X\rfloor$ with $P(X,y) = 0$.
By our assumption, $P$ is regular in~$Y$ of degree~$1$, so   
Lemma~\ref{lem:WP rest}
gives a $u\in R\lfloor X,Y\rfloor^\times$ and a monic~$w\in R\lfloor X\rfloor[Y]$ of degree $1$  with~$P = uw$. Now $w = Y - y$ where~$y\in R\lfloor X\rfloor$, and $y$ has the required property.
\end{proof}

\begin{cor}\label{cor:roots, rest}
Let $f\in\k\lfloor X\rfloor$ and $k\geq 1$ such that $k\notin\fm$. Then 
$$\text{$f=g^k$ for some $g\in\k\lfloor X\rfloor^\times$} \quad\Longleftrightarrow\quad f\in (\k^\times)^k(1+\fm R\lfloor X\rfloor).$$
In particular, the group of $1$-units of $\k\lfloor X\rfloor$ is $k$-divisible.
\end{cor}
\begin{proof}
The forward direction is clear, and for the backward direction we may assume~$f\in 1+\fm R\lfloor X\rfloor$, and we need to show that then $f\in (R\lfloor X\rfloor^\times)^k$.
To see this consider the polynomial $P=Y^k-f$ and $z=1$ in Lemma~\ref{lem:hens=>newt} applied to
the henselian pair $\big(R\lfloor X\rfloor, \fm R\lfloor X\rfloor\big)$.
\end{proof}

\subsection{Zero sets and vanishing ideals}
Let $I$ be an ideal of $\k\lfloor X\rfloor$. We then let
$$\operatorname{Z}(I) := \big\{ a\in R^m: \text{$f(a)=0$ for all $f\in I$} \big\},$$
the zero set of $I$.  If $J$ is another ideal of $\k\lfloor X\rfloor$, then~$\operatorname{Z}(I)\subseteq \operatorname{Z}(J)$ if
$I\supseteq J$, and
$$\operatorname{Z}(IJ)=\operatorname{Z}(I\cap J)=\operatorname{Z}(I)\cup \operatorname{Z}(J), \qquad
\operatorname{Z}(I+J)=\operatorname{Z}(I)\cap \operatorname{Z}(J).$$
Let also $S$ be a subset of $R^m$. The {\it vanishing ideal}\/ of $S$ is the ideal
$$\operatorname{I}(S) := \big\{ f\in \k\lfloor X\rfloor: \text{$f(a)=0$ for all $a\in S$} \big\}$$
 of $\k\lfloor X\rfloor$. If also $T\subseteq R^m$, then~$\operatorname{I}(S)\subseteq  \operatorname{I}(T)$ if $S\supseteq T$, and~$\operatorname{I}(S\cup T) = \operatorname{I}(S)\cap \operatorname{I}(T)$.
 Clearly 
$\operatorname{I}\!\big(\!\operatorname{Z}_K(I)\big)\supseteq\sqrt{I}$ and~$\operatorname{Z}_K(I) = \operatorname{Z}_K(\sqrt{I})$.
(Since we   only need to consider zeros of power series $f\in \k\lfloor X\rfloor$ in $R^m$, and not in the valuation ring of an extension
of $\k$, we do not require  ``restricted $W$-structures'' in analogy to Definition~\ref{def:W-structure}.)

\subsection{The Nullstellensatz for the Tate algebra}
{\it In this subsection we let $\mathcal L=\mathcal L_{\preceq}$, and we
assume that~$\k$ is algebraically closed unless noted otherwise.}\/ Thus $\k$ (viewed as an $\mathcal L$-structure) is a model
of $\operatorname{ACVF}$.
Let~$y=(y_1,\dots,y_n)$ be a tuple of distinct $\mathcal L$-variables.
Here is an analogue of Proposition~\ref{prop:Rueckert spec}:
 
\begin{prop} \label{prop:Rueckert spec, rest}
Let $\varphi(y)$ be an $\mathcal L$-formula,  $f=(f_1,\dots,f_n)\in \k\lfloor X\rfloor^n$,
$\Omega\models\operatorname{ACVF}$,
and~$\sigma \colon \k\lfloor X\rfloor\to \Omega$ be a   ring morphism
such that $\sigma( R\lfloor X\rfloor)\preceq 1$, $\sigma(\fm)\prec   1$,  and~$\Omega\models\varphi(\sigma(f))$. Then
$\k\models\varphi(f(a))$ for some $a\in R^m$.
\end{prop}
\begin{proof}
We proceed by induction on $m$. 
The case $m=0$ follows from model completeness of $\operatorname{ACVF}$
since $\sigma$ restricts to an $\mathcal L$-embedding  $\k\to\Omega$. Suppose~$m\geq 1$. Using
Theorem~\ref{thm:Robinson}, as in the proof of Proposition~\ref{prop:Rueckert spec} we first arrange that $\varphi$ is quantifier-free of the form
$$ \bigwedge_{i\in I} P_i(y) = 0 \wedge Q(y)\neq 0 \wedge \bigwedge_{j\in J} R_j(y)\, \square_j\, S_j(y)$$
where $I$, $J$ are finite index sets, $P_i,Q,R_j,S_j\in\Z[Y_1,\dots,Y_n]$, and each $\square_j$ is one of the symbols $\preceq$ or $\prec$.
We also arrange
that~$P_i$, $Q$, $R_j$, $S_j$ are   distinct elements of~$\{Y_1,\dots,Y_n\}$, and  $f_1,\dots,f_n\neq 0$. Take some $d\in\N$, $d>1$ such that  $\tau_d(f_j)$ is regular in $X_m$, for $j=1,\dots,n$, and
replace~$f_j$ by $\tau_d(f_j)$ ($j=1,\dots,n$) and~$\sigma$ by~$\sigma\circ\tau_d^{-1}$ to further arrange that
$f_j$ is regular in $X_m$, for~$j=1,\dots,n$.
As usual let~$X'=(X_1,\dots,X_{m-1})$. Lem\-ma~\ref{lem:WP rest}  then yields some $1$-unit $u_j$   in~$\k\lfloor X\rfloor$ as well as a   polynomial~$w_j\in\mathcal \k\lfloor X'\rfloor[X_m]$ such that~$f_j=u_j w_j$, for~$j=1,\dots,n$. 
In~$\Omega$ we have $\sigma(u_j)\sim 1$  since $\sigma(\fm R\lfloor X\rfloor)\prec 1$, and in $\k$ we have $u_j(a)\sim 1$  for each $a\in R^m$.
Hence we may replace each~$f_j$ by $w_j$ to arrange that $f_j=w_j$ is a polynomial. Now
argue as in the proof of Proposition~\ref{prop:Rueckert spec} with ``$v\prec 1$'' in the definition of $\psi$ replaced by~``$v\preceq 1$''
and $K$, $\smallo$ replaced by $\k$, $R$, respectively.
\end{proof}

\begin{cor}\label{cor:spec, rest}
Let $\varphi(y)$ be an $\mathcal L_{\operatorname{R}}$-formula, 
     $f\in \k\lfloor X\rfloor^n$, and~$\sigma \colon \k\lfloor X\rfloor\to \Omega$ be a ring morphism  to an algebraically
  closed field  with~$\Omega\models\varphi(\sigma(f))$. Then~$\k\models\varphi(f(a))$ for some $a\in R^m$.
\end{cor}
\begin{proof}
Note that $\sigma$ restricts to a field embedding~$\k\to\Omega$;   identify~$\k$ with its image under $\sigma$.
Let $A:=\sigma(R\lfloor X\rfloor)$; then $R\subseteq A$ and
$1\notin\fm A$. Let~$M\supseteq \fm A$ be a maximal ideal of $A$; then the local subring $A_M$ of $\Omega$ lies over~$R$.
Take a valuation ring of $\Omega$ which lies over $A_M$, and hence over~$R$,
and equip $\Omega$ with the associated dominance relation. Then $\sigma(\RX)\preceq 1$, $\sigma(\fm)\prec 1$. 
Model completeness of
 the $\mathcal L_{\operatorname{R}}$-theory of algebraically closed fields allows us to replace~$\Omega$ by a suitable extension  
 to arrange~$\Omega\models\operatorname{ACVF}$. Now use Proposition~\ref{prop:Rueckert spec, rest}. 
\end{proof}

\noindent
Arguing as in the proof of Theorem~\ref{thm:Rueckert}, using
Corollary~\ref{cor:spec, rest} in place of Corollary~\ref{cor:spec}, we obtain the familiar Nullstellensatz for $\kX$. (Cf.~\cite[\S{}7.1.2, Theorem~3]{BGR}. This is a special case of   a more general Nullstellensatz for power series rings over algebraically closed valued fields from~\cite[Proposition~5.2.7]{CL}.)

\begin{cor}\label{cor:Rueckert, rest}
Let $I$ be an ideal of~$\k\lfloor X\rfloor$; then
$\operatorname{I}\!\big(\!\operatorname{Z}(I)\big)=\sqrt{I}$.
\end{cor}

\noindent 
We also note that Proposition~\ref{prop:Rueckert spec, rest} easily yields a fact about rational domains (cf.~\cite[\S{}7.2.3]{BGR})
which is in analogy to Corollary~\ref{cor:p-adic H17} above:

\begin{prop}\label{prop:int valued}
Let $f,g\in \k\lfloor X\rfloor$. Then   
$$\text{$\abs{f(a)} \leq \abs{g(a)}$ for all $a\in R^m$}\quad\Longleftrightarrow\quad f\in g R\lfloor X\rfloor.$$
\end{prop}
\begin{proof}
We only need to show ``$\Rightarrow$''. Thus suppose  $\abs{f(a)} \leq \abs{g(a)}$ for all $a\in R^m$; then $\abs{f}\leq\abs{g}$
by Lemma~\ref{lem:max principle}.
To show $f\in g R\lfloor X\rfloor$ we can assume~$g\neq 0$, and it then suffices to show
$f\in g\k\lfloor X\rfloor$. Suppose otherwise, and
let $A:=\k\lfloor X\rfloor$, $F:=\Frac(A)$; then~$h:=f/g\in F\setminus A$. Now the integral domain $A$ is factorial by Lemma~\ref{lem:W-Lemma, rest} and
so integrally closed in $F$ (see, e.g., \cite[Lemma~1.3.11]{ADH}). Hence~\cite[Theorem~10.4]{Matsumura} yields a valuation ring~$\mathcal O_F$ of $F$ containing $A$ but not $h$.
The residue morphism~$\pi\colon\mathcal O_F\to\k_F:=\mathcal O_F/\smallo_F$
restricts to a field embedding $\k\to\k_F$, via which we identify $\k$ with a subfield of~$\k_F$.
Let $B:=\pi(R\lfloor X\rfloor)$; then $R \subseteq B$ and~$1\notin\fm B$.
Let $M$ be a maximal ideal of~$B$ containing~$\fm B$; then $B_M$ is a local subring of $\k_F$ lying over
$R$. Now \cite[Pro\-po\-si\-tion~3.1.13]{ADH}  (applied to $B_M$, $\k_F$ in place of~$A$,~$K$)   and Lemma~\ref{lem:lift val spec}       yield
a valuation ring $\mathcal O$ of  $F$ with $R\lfloor X\rfloor\subseteq\mathcal O\subseteq\mathcal O_F$ which
lies over the valuation ring $R$ of $\k$. Let $\Omega$ be   some algebraically closed  field extension  of~$F$ equipped with the dominance relation associated to a valuation ring of $\Omega$
lying over~$\mathcal O$; then $\Omega\models\operatorname{ACVF}$ and $R\lfloor X\rfloor\preceq 1$, $\fm\prec 1$.
Proposition~\ref{prop:Rueckert spec, rest} now yields some~$a\in R^m$ with~$\abs{f(a)} > \abs{g(a)}$, a contradiction. 
\end{proof}

\begin{remark}
The analogue of the previous proposition with $f,g\in\k[X]$ was studied in~\cite{PrRi}.
If the   equivalence in Proposition~\ref{prop:int valued} holds for all $f,g\in\k[X]$, even just when~$m=1$,
then $\k$   necessarily is algebraically closed.
This follows from \cite[Theorem~3.1]{PrRi} and Corollary~\ref{cor:div poly}.
\end{remark}

\begin{cor}\label{cor:int valued}
Let $f,g\in \k\lfloor X\rfloor$. Then   
$$\text{$\abs{f(a)} < \abs{g(a)}$ for all $a\in R^m$}\quad\Longleftrightarrow\quad f\in g \fm R\lfloor X\rfloor.$$
\end{cor}

\noindent
This follows from the remark after Lemma~\ref{lem:max principle} combined with Proposition~\ref{prop:int valued}.
We finish with another application of Proposition~\ref{prop:Rueckert spec, rest}, a generalization of 
\cite[Proposition~5.1]{HY}.
Here it is convenient to switch from the absolute value $\abs{\,\cdot\,}$ of $\k$ to the valuation $$a\mapsto va:=-\log\,\abs{a}\colon\k^\times\to\Gamma=v(\k^\times)\,\subseteq\,\R$$ on $\k$. For $a=(a_1,\dots,a_m)\in\k^m$ we set $va:=\big(v(a_1),\dots,v(a_m)\big)\in \Gamma^m$.
 
\begin{prop}\label{prop:closure}
Let $\varphi(y)$ be an $\mathcal L_{\operatorname{R}}$-formula, $f\in\k\lfloor X\rfloor^n$, and
$$S:=\big\{ a\in R^m:  \k\models\varphi(f(a)) \big\}.$$
Then $vS :=\{ va:a\in S \}\subseteq \Gamma^m$ is closed.
\end{prop}

\noindent
For the proof,
 similarly to Section~\ref{sec:inf W-structures as structures}, we let $\mathcal L_W$ be the expansion of $\mathcal L$ by a new $n$-ary function symbol $g$
 for each $n$ and $g\in W_n$, and expand $\k$ to an $\mathcal L_W$-structure by interpreting each such function symbol $g$
 by the function $\k^n\to\k$ which agrees with $a\mapsto g(a)$ on $R^n$ and is identically zero on $\k^n\setminus R^n$.
 
 \begin{proof}[Proof of Proposition~\ref{prop:closure}]
Let $\beta$ be an element of the closure of $vS$.  
 Let $\k^*$ be an elementary extension of the $\mathcal L_W$-structure $\k$,
 with valuation $v^*\colon (\k^*)^\times\to\Gamma^*$, such that $\Gamma^*$ contains
 some positive element   smaller  than every positive element of $\Gamma$.
 This yields a non-zero convex subgroup $\Delta^*$ of $\Gamma^*$ such that $\Gamma\cap\Delta^*=\{0\}$.
 Now $\beta$ is contained in the closure of $v^*(S^*)$ where $S^*$ is defined like $S$ with $\k$, $R$ replaced by~$\k^*$ and its valuation ring $R^*$, respectively. Thus we obtain some $a^*\in S^*$ such that~$v^*(a^*)-\beta\in(\Delta^*)^m$.
 Let $\Omega$ be the   field $\k^*$ equipped with the valuation ring of the $\Delta^*$-coarsening  of the
 valuation of $\k^*$.
Then $\Omega\models\operatorname{ACVF}$, and the   $\mathcal L$-reduct of~$\k$ is a substructure of
$\Omega$ with  $v_\Omega(a^*)=\beta$. Let $\sigma\colon \k\lfloor X\rfloor\to\Omega$ be the ring morphism~$g\mapsto g^*(a^*)$, where $g^*$ denotes   the interpretation of the
function symbol~$g$ in the $\mathcal L_W$-structure $\k^*$. Then $\sigma(R\lfloor X\rfloor)\preceq 1$, $\sigma(\fm)\prec 1$.
Let $b\in\k^m$ with $vb=\beta$.
Then  the tuple $(X,b,f)\in \k\lfloor X\rfloor^{m+m+n}$ satisfies~$\sigma(X,b,f)=\big(a^*,b,f^*(a^*)\big)$,
and hence
$\Omega\models\psi\big(\sigma(X,b,f)\big)$ where $\psi=\psi(w,x,y)$ is the $\mathcal L$-formula
$$\varphi(y)\wedge w_1 \asymp x_1 \wedge\cdots\wedge w_m\asymp x_m.$$
Proposition~\ref{prop:Rueckert spec, rest} now yields $a\in R^m$ with $\k\models\psi\big(a,b,f(a)\big)$,
so $\beta=va\in vS$.
 \end{proof}
 
\subsection{The real holomorphy ring}
For use in the next subsection we recall the analogue of the Kochen ring    in
the real setting. {\it Throughout this section we let $K$ be a field.}\/
A valuation ring of $K$ is said to be {\it real}\/
if its residue field is formally real. If~$\mathcal O$ is a real valuation ring of $K$ then some ordering of $K$ makes $K$ an ordered field such that $\mathcal O$ is convex. (Baer-Krull  \cite[Theorem~10.1.10]{BCR}.) Hence   every
subring of~$K$ containing  a real valuation ring of $K$ is also a real valuation ring of $K$. Moreover: 

\begin{lemma}\label{lem:real val}
Let $\mathcal O$ be a valuation ring of $K$  and $\Delta$ be a convex subgroup of its value group. Then $\mathcal O$ is real iff the
valuation ring  of the $\Delta$-specialization   of~$(K,\mathcal O)$ is real.
\end{lemma}

\noindent
Suppose $K$ is formally real (equivalently,  has a real valuation ring).
Let $H_K$ be the {\it real holomorphy ring}\/ of $K$, that is, the intersection of all real valuation rings of $K$. 
 Sch\"ul\-ting~\cite{Schuelting} showed that
 $H_K$ is the subring of~$K$   generated by the elements~$\frac{1}{1+f}$ where~$f$ is a sum of squares in $K$: $f=g_1^2+\cdots+g_n^2$~($g_i\in K$);
 alternatively, $H_K$ consists of all $h\in K$ such that for some $n$, both $n+h$ and~$n-h$ are a sum of squares in $K$.
 (Here as everywhere in this paper, $n$ denotes a natural number.)
For a subring $A$ of $K$,   the subring $H_KA$ of $K$ generated by $H_K$ and~$A$ is the intersection of all real valuation 
rings of $K$ containing~$A$~\cite[Lem\-ma~1.12]{Becker}.

\begin{lemma}\label{lem:HK}
Let $\sigma\colon A\to L$ be a ring morphism to
a formally real field where $A$ is a regular integral domain with fraction field $K$. Let $f,g\in A$. Then 
$$f\in gH_K\ \Longrightarrow\ \sigma(f)\in \sigma(g)H_L.$$
\end{lemma}
\begin{proof}
This is similar to the proof of Lemma~\ref{lem:p-adic prime ideals}, with $A$, $H_K$, $H_L$ in place of $R$, $\Lambda_K$, $\Lambda_L$, respectively.
Let $P:=\ker\sigma$, a real prime ideal of $A$. 
As in the proof of that lemma we   arrange that $A$ is local with maximal ideal~$P$,~$L=A/P$,
and $\sigma\colon A\to L$ is the residue morphism.
 Suppose $f\in gH_K$. Let~$\mathcal O_L$ be a real valuation ring of $L$, with associated dominance relation $\preceq_L$ on $L$;
we need to show that~$\sigma(f)\preceq_L \sigma(g)$. Take a valuation ring~$\mathcal O$ of~$K$ lying over $R$ such that the
natural inclusion~$A\to\mathcal O$ induces an isomorphism~$L=A/P\to\k=\mathcal O/\smallo$.  (Lemma~\ref{lem:existence of places}.)
 Let $\mathcal O_{\k}$ be the image of $\mathcal O_L$ under the isomorphism~$L\to\k$. By Lemma~\ref{lem:lift val spec} take 
 a valuation ring~$\mathcal O_0$ of~$K$ with~$\mathcal O_0\subseteq\mathcal O$ and a convex subgroup~$\Delta_0$
of its value group such that the $\Delta_0$-specialization of~$(K,\mathcal O_0)$ is~$(\k,\mathcal O_{\k})$. 
So~$\mathcal O_0=\big\{h\in\mathcal O:\sigma(h)\preceq_L 1\big\}$, and $\mathcal O_0$ is real by Lemma~\ref{lem:real val}, hence~$H_K\subseteq\mathcal O_0\subseteq\mathcal O$.
Arguing as at the end of the proof of Lemma~\ref{lem:p-adic prime ideals} with~$H_K$ in place of  $\Lambda_K$,
we conclude $\sigma(f)\preceq_L\sigma(g)$ as required.
\end{proof}

\subsection{The Nullstellensatz for real restricted analytic functions}
{\it In this subsection we assume  $\k$ to be real closed.}\/ Examples include each Hahn field~${\k=\R(\!(t^\Gamma)\!)}$ where $\Gamma$ is a divisible ordered subgroup
of the ordered additive group of reals, equipped with the absolute value
$f\mapsto \abs{f}\colon \k\to\R^{\geq}$ with $\abs{f}=\operatorname{e}^{-vf}$ for~$f\in\k^\times$,
or the completion~$\operatorname{cl}(\R[t^\Gamma])$
of the $\R$-subalgebra $\R[t^\Gamma]$ of $\k$, equipped with the restriction of the absolute value of $\k$.
(See the examples after Corollary~\ref{cor:phin(f)(0)}; here  $\operatorname{cl}(\R[t^\Q])$ is the Levi-Civita field mentioned in the introduction.) Note that since~$R$ is henselian, it is a proper convex subring of $\k$. 
(See \cite[Corollary~3.3.5 and Lemma~3.5.14]{ADH}.)
Thus equipping~$\k$ with the dominance relation associated to $R$ makes~$\k$ a model of~$\operatorname{RCVF}$.
Let $\mathcal L=\mathcal L_{\operatorname{OR},\preceq}$ and $y=(y_1,\dots,y_n)$ be a tuple of distinct $\mathcal L$-variables.

\begin{prop} \label{prop:Risler spec, rest}
Let $\varphi(y)$ be an $\mathcal L$-formula, $f\in \k\lfloor X\rfloor^n$, and 
 $\sigma \colon \k\lfloor X\rfloor\to   \Omega$ be a ring morphism, where 
 $\Omega\models\operatorname{RCVF}$, such that   $\sigma(R\lfloor X\rfloor)\preceq 1$, $\sigma(\fm)\prec 1$, and~$\Omega\models\varphi(\sigma(f))$. Then
$\k\models\varphi(f(a))$ for some $a\in R^m$.
\end{prop}

\begin{proof} 
Induction on $m$ as in the proof of Proposition~\ref{prop:Rueckert spec, rest}. 
The case $m=0$ follows from model completeness of $\operatorname{RCVF}$.
Suppose~$m\geq 1$.  By
Theorem~\ref{thm:CD} we   arrange that  $\varphi$ has the form
$$P(y) = 0 \wedge \bigwedge_{i\in I} Q_i(y) > 0 \wedge \bigwedge_{j\in J} R_j(y)\, \square_j\, S_j(y)$$
where $I$, $J$ are finite, $P,Q_i,R_j,S_j\in\Z[Y_1,\dots,Y_n]$, and each $\square_j$ is~$\preceq$ or~$\prec$.
As in  the proof of Proposition~\ref{prop:Rueckert spec, rest} 
we also arrange
that $P$, $Q_i$, $R_j$, $S_j$ are   distinct elements of~$\{Y_1,\dots,Y_n\}$ and 
$f_1,\dots,f_n$ are all regular in $X_m$. Take a $1$-unit $u_j\in R\lfloor X\rfloor$ and some $w_j\in \k\lfloor X'\rfloor[X_m]$, where $X'=(X_1,\dots,X_{m-1})$, such that~$f_j=u_jw_j$. Then~$\sigma(u_j)\sim 1$ since $\sigma(\fm R\lfloor X\rfloor)\prec 1$, hence
also $\sigma(u_j)>0$. Thus we may arrange that each $f_j=w_j$ is a polynomial and argue as in the proof of 
Proposition~\ref{prop:Rueckert spec, rest}.
\end{proof}

\noindent
In Lemma~\ref{lem:sigma(fm)prec1} and its corollary, $\Omega$ is an ordered field and $\sigma\colon \k\lfloor X\rfloor\to\Omega$ is a ring morphism.

\begin{lemma}\label{lem:sigma(fm)prec1}
Equip $\Omega$ with the dominance relation associated to the convex hull of 
  $\sigma(R\lfloor X\rfloor)$ in~$\Omega$. Then $\sigma(\fm)\prec 1$.
\end{lemma}
\begin{proof}
Let $a\in\fm$, $a>0$, and $f\in  R\lfloor X\rfloor$. Then $g:=1-af\in 1+\fm R\lfloor X\rfloor$ is 
a square in $R\lfloor X\rfloor$ by Corollary~\ref{cor:roots, rest}, thus $\sigma(g)>0$
and so $\sigma(f)<\sigma(a^{-1})$. 
Since this holds for all $f\in  R\lfloor X\rfloor$, we obtain $1\prec\sigma(a^{-1})$ and thus $\sigma(a)\prec 1$ as required.
\end{proof}

\begin{cor}\label{cor:Risler spec, rest}
Let $\varphi(y)$ be an $\mathcal L_{\operatorname{OR}}$-formula, 
     $f\in \k\lfloor X\rfloor^n$, and~$\Omega$ be real closed with~$\Omega\models\varphi(\sigma(f))$.
Then~$\k\models\varphi(f(a))$ for some $a\in R^m$.
\end{cor}
\begin{proof}
Equip $\Omega$ with the   dominance relation associated to the  convex hull of the subring~$\sigma(R\lfloor X\rfloor)$ in $\Omega$. Then  $\sigma(\fm)\prec 1$ by the   lemma above. Replace~$\Omega$ by a suitable extension 
(using model completeness of the $\mathcal L_{\operatorname{OR}}$-theory  of real closed fields) to arrange 
$\Omega\models\operatorname{RCVF}$.   Now the corollary follows from Proposition~\ref{prop:Risler spec, rest}. 
\end{proof}

\noindent
As in the proof of Corollary~\ref{cor:Risler H17} respectively Theorem~\ref{thm:Risler},
using Corollary~\ref{cor:Risler spec, rest} instead of 
Corollary~\ref{cor:Risler spec}, 
we now obtain  Theorem~D from the introduction:

\begin{cor}[real restricted analytic Positivstellensatz]
For each $f\in \k\lfloor X\rfloor$ we have $f(a)\geq 0$ for all $a\in R^m$ iff
there are  $g\in\k\lfloor X\rfloor\setminus\{0\}$ and $h_1,\dots,h_k\in\k\lfloor X\rfloor$   such that
$fg^2 =  h_1^2+\cdots+h_k^2$.
\end{cor}

\begin{cor}[real restricted analytic Nullstellensatz]
Let $I$ be an ideal of $\k\lfloor X\rfloor$; then $\operatorname{I}\!\big(\!\operatorname{Z}(I)\big)=\sqrt[\leftroot{2}\uproot{5}\operatorname{r}]{I}$. 
\end{cor}

\noindent
As in \cite[Proposition~6.4]{Risler76}, the previous corollary implies a \L{}ojasiewicz Inequality:

\begin{cor}[]
Let $f,g\in\k\lfloor X\rfloor$ such that $\operatorname{Z}(f)\supseteq\operatorname{Z}(g)$.
Then there are $c\in\R^{\geq}$ and $k\geq 1$ such that $\abs{f(a)}^k\leq c\abs{g(a)}$ for each $a\in R^m$.
\end{cor}
\begin{proof}
The previous corollary applied to $I=g\k\lfloor X\rfloor$ yields $b_1,\dots,b_l,h\in \k\lfloor X\rfloor$ and~$k\geq 1$ such that
$f^{2k}+b_1^2+\cdots+b_l^2=gh$. So  for each $a\in R^m$ we have~$g(a)h(a)\geq f^{2k}(a)\geq 0$ and thus $c\abs{g(a)}\geq\abs{g(a) h(a)}\geq\abs{f(a)}^{2k}$  for $c:=\abs{h}$.
\end{proof}

\noindent
To show Theorem~E, we restrict to the case $\k\lfloor X\rfloor=\kX$. We let  $F:=\Frac(\kX)$ and consider the real holomorphy ring $H_F$ of $F$.
The value group of $\k$ being archimedean, we have $R=H_{\k}\subseteq H_F$. We also have $\fm\RX \subseteq H_F$, since
for~$f\in\fm\RX$, $1\pm f$ is a square in $\RX$ by  Corollary~\ref{cor:roots, rest}. Hence $\RX^\times\subseteq H_F^\times$.

\medskip
\noindent
Let also $\mathcal O_F:=\{f\in F:  \abs{f}\leq 1 \}$ be the valuation ring of the Gauss norm on~$F$,
with maximal ideal $\smallo_F=\{f\in F:\abs{f}<1\}$.  The residue field of $\mathcal O_F$ is the formally real field $\overline{R}(X)$;
hence $H_F\subseteq\mathcal O_F$. The following proposition characterizes the elements of the subring~$H_F\RX$
of $\mathcal O_F$:

\begin{prop}\label{prop:int valued, kX}
Let $f,g\in \kX$. Then   
$$\text{$\abs{f(a)} \leq \abs{g(a)}$ for all $a\in R^m$}\quad\Longleftrightarrow\quad   f\in gH_F\RX.$$
\end{prop}
\begin{proof}
Since   $\kX$ is regular (see the remark after Lemma~\ref{lem:W-Lemma, rest}),  
the backward direction follows from 
Lemma~\ref{lem:HK} applied to the ring morphisms $h\mapsto h(a)\colon \kX\to\k$ ($a\in R^m$), and the fact that $\abs{h(a)}\leq 1$
for each $h\in\RX$, $a\in R^m$.
For the forward direction we argue as in the proof of Proposition~\ref{prop:int valued}. Suppose~$\abs{f(a)} \leq \abs{g(a)}$ for all $a\in R^m$
but $f\notin g H_F\RX$.
Then~$g\neq 0$. Take a real valuation ring~$\mathcal O_1$ of $F$ containing $\RX$
but not~$h:=f/g$. Turn~$F$ into an ordered field such that~$\mathcal O_1$ is convex.
Let~$\mathcal O$ be the convex hull of $\RX$ in~$F$; then~$\mathcal O\subseteq\mathcal O_1$, hence $h\notin\mathcal O$.
Equip~$F$ with the dominance relation associated to~$\preceq$, and
take $\Omega\models\operatorname{RCVF}$ extending $F$. Then in~$\Omega$ we have~$\RX\preceq 1$, and~$\fm\prec 1$ by Lemma~\ref{lem:sigma(fm)prec1}, so 
  Proposition~\ref{prop:Risler spec, rest}   yields an~$a\in R^m$ with~$\abs{f(a)} > \abs{g(a)}$, a contradiction. 
\end{proof}



\noindent
With Proposition~\ref{prop:Risler spec, rest} in place, we can  also use it in the same
way as Proposition~\ref{prop:Rueckert spec, rest} to prove the following real analogue of Proposition~\ref{prop:closure}.
As in that proposition   let $v\colon\k^\times\to\Gamma$ be the valuation on $\k$
associated to~$\abs{\,\cdot\,}$.

\begin{prop}\label{prop:closure, real}
Let $\varphi(y)$ be an $\mathcal L_{\operatorname{R}}$-formula, $f\in\k\lfloor X\rfloor^n$, and
$$S:=\big\{ a\in R^m:  \k\models\varphi(f(a)) \big\}.$$
Then $vS :=\{ va:a\in S \}\subseteq \Gamma^m$ is closed.
\end{prop}

\noindent
See \cite[Proposition~5.3]{HY} for a weaker result.
The preceding proposition now yields a version of Corollary~\ref{cor:int valued} in the present setting:

\begin{cor}\label{cor:closure, real}
Let $f,g\in\kX$. Then $\abs{f(a)}<\abs{g(a)}$ for every $a\in R^m$ iff~$f\in g \fm H_F\RX$.
\end{cor}
 \begin{proof}
 The backwards direction here is clear from that in Proposition~\ref{prop:int valued, kX}.
 For the converse
 suppose $\abs{f(a)}<\abs{g(a)}$ for all $a\in R^m$. Applying the  proposition above
to 
 $$S:=\big\{ (a,b)\in R^m\times R: f(a)=g(a)b  \big\}$$ implies that   $\big\{\abs{f(a)/g(a)}:a\in R^m\big\}\subseteq \abs{\k}$  is closed. 
 This yields an~$\varepsilon\in\fm$
such that~$\abs{f(a)}\leq\abs{\varepsilon g(a)}$ for all~$a\in R$. Then
  $f\in  \varepsilon g H_F\RX$ by Proposition~\ref{prop:int valued, kX}.
 \end{proof}

\noindent
The combination of Proposition~\ref{prop:int valued, kX} and Corollary~\ref{cor:closure, real} now yields Theorem~E from the introduction.
 
\subsection{The Nullstellensatz for $p$-adic restricted analytic functions}
{\it In this subsection we  assume $\k=\Q_p$.}\/ We first note:

\begin{lemma}\label{lem:indets infinitesimal, p-adic, rest}
Let $\Omega$ be a   $p$-valued  field and $\sigma\colon  \Q_p\lfloor X\rfloor\to \Omega$ be a ring morphism.
Then $\sigma(\Z_p\lfloor X\rfloor)\preceq_p 1$ and~$\sigma(p)\prec_p 1$.
\end{lemma}
\begin{proof}
Note that $\sigma$ restricts to an embedding of $\Q_p$ into $\Omega$; hence the $p$-valuation of~$\Omega$
restricts to the unique $p$-valuation of $\sigma(\Q_p)$, so
$\sigma(p)\prec_p 1$.
Now let $g\in\Z_p\lfloor X\rfloor$. Suppose first that $p$ is odd.
Then $1+pg^2$ is a square in $\Q_p\lfloor X\rfloor$, by Corollary~\ref{cor:roots, rest},
hence  $1+p\sigma(g)^2$ is a square in~$\Omega$, so  $\sigma(g)\preceq_p 1$.
For $p=2$ we note that $1+pg^3$ is a cube in $\Q_p\lfloor X\rfloor$, and as before this
yields $\sigma(g)\preceq_p 1$.
\end{proof}

\noindent
We also let $\mathcal L=\mathcal L_{\operatorname{Mac}}$ and~$y=(y_1,\dots,y_n)$ be a tuple of distinct $\mathcal L$-variables.  We now have  a restricted analytic analogue of
Proposition~\ref{prop:p-adic spec}:

\begin{prop}\label{prop:p-adic spec, rest}
Let $\varphi(y)$ be an $\mathcal L$-formula,   $f=(f_1,\dots,f_n)\in \Q_p\lfloor X\rfloor^n$, $\Omega\models p\operatorname{CF}$,
and~$\sigma \colon \Q_p\lfloor X\rfloor\to  \Omega$  be a   ring morphism 
such that $\Omega\models\varphi(\sigma(f))$. Then~$\Q_p\models\varphi(f(a))$ for some $a\in \Z_p^m$.
\end{prop}
\begin{proof}
By induction on $m$ in a similar way as in the proof of Proposition~\ref{prop:Rueckert spec, rest}. 
The case $m=0$ holds by model completeness of $p\operatorname{CF}$. 
Suppose $m\geq 1$. As in the proof of Proposition~\ref{prop:p-adic spec}, using Theorem~\ref{thm:Macintyre} and 
the argument to eliminate ${}^{-1}$ in the proof of Lemma~\ref{lem:Macintyre} in place of Corollary~\ref{cor:Belair},
 we   arrange that~$\varphi$ has the form
$$P(y) = 0 \wedge \bigwedge_{i\in I} \operatorname{P}_{k_i}\!\big(Q_i(y)\big) \wedge  \bigwedge_{j\in J} R_j(y)\, \preceq_p\, S_j(y)$$
where $I$, $J$ are  finite, $P,Q_i,R_j,S_j\in\Z[Y_1,\dots,Y_n]$,  $k_i\geq 1$.
As in the proof of Proposition~\ref{prop:Rueckert spec, rest} we arrange
that~$P_i$, $Q$, $R_j$, $S_j$ are   distinct elements of~$\{Y_1,\dots,Y_n\}$, and  $f_1,\dots,f_n$ are regular in $X_m$.
 Lem\-ma~\ref{lem:WP rest}  then yields~$u_j\in \Z_p\lfloor X\rfloor^\times$ and~$w_j\in\mathcal \Q_p\lfloor X'\rfloor[X_m]$ such that~$f_j=u_j w_j$~($j=1,\dots,n$). 
 
Let~$k\geq 1$ and $\lambda_1,\dots,\lambda_N\in\Z$  ($N\geq 1$) be representatives for the cosets  of the subgroup
$(\Q_p^\times)^k$ of $\Q_p^\times$; then for distinct variables $u$, $w$  we have
$$p\operatorname{CF}\models \operatorname{P}_k(uw) \leftrightarrow \bigvee_{\lambda_i\lambda_j\in (\Q_p^\times)^n} 
\big(\operatorname{P}_k(\lambda_i u) \wedge \operatorname{P}_k(\lambda_j w)\big).$$
Moreover,    $\sigma(u_j)\sim 1$ in~$\Omega$,  since $\sigma(p\ZpX)\prec 1$ (Lemma~\ref{lem:indets infinitesimal, p-adic, rest}), and~${u_j(a)\sim 1}$  for~$a\in \Z_p^m$.
Hence we can further arrange that for each $j=1,\dots,n$, either~$f_j\in\ZpX^\times$ or~$f_j\in\mathcal \Q_p\lfloor X'\rfloor[X_m]$.
Next, let $k\geq 1$; if we take $K\in\N$  so large that~$1+p^K\Z_p\subseteq (\Q_p^\times)^k$,
and let $\mu_1,\dots,\mu_M\in\Z$ ($M\geq 1$) be representatives for all congruence classes modulo $p^K$ of elements of $(\Z_p^\times)^k$,
then
$$p\operatorname{CF}\models \big(u\asymp_p 1 \wedge \operatorname{P}_k(u)\big) \leftrightarrow \bigvee_{i} (u - \mu_i)\preceq_p p^K.$$
Using this remark we   finally arrange that each $f_j$ is a polynomial in~$\mathcal \Q_p\lfloor X'\rfloor[X_m]$.
We now finish the inductive step as in 
the proof of Proposition~\ref{prop:Rueckert spec, rest}.  
\end{proof}

\noindent
By Lemma~\ref{lem:W-Lemma, rest},  $\Q_p\lfloor X\rfloor$ is regular, so
  $F:=\Frac(\Q_p\lfloor X\rfloor)$ is formally $p$-adic, and for each $a\in \Z_p^m$, the
  kernel of  $f\mapsto f(a)\colon \Q_p\lfloor X\rfloor \to \Q_p$ is a $p$-adic prime ideal  of~$\Q_p\lfloor X\rfloor$.
  (Corollary~\ref{cor:p-adic prime ideals}.)
 We  obtain a restricted analytic version of
Corollary~\ref{cor:p-adic H17}:

\begin{cor}[restricted $p$-adic analytic  Hilbert's 17th Problem]\label{cor:p-adic H17 rest}
$$\Lambda_F = \left\{ \frac{f}{g}:f,g\in \Q_p\lfloor X\rfloor,\ g\neq 0,\ \abs{f(a)}_p \leq \abs{g(a)}_p \text{ for all $a\in\Z_p^m$} \right\}.$$
\end{cor}

\noindent
This is proved just like  Corollary~\ref{cor:p-adic H17}, using Proposition~\ref{prop:p-adic spec, rest} in place
of Corollary~\ref{cor:p-adic spec}.
Similarly to Theorem~\ref{thm:p-adic NSS}, using again 
Proposition~\ref{prop:p-adic spec, rest} instead 
of Corollary~\ref{cor:p-adic spec}, we also show:

\begin{cor}[restricted $p$-adic analytic Nullstellensatz]\label{cor:p-adic NS rest}
Let $I$ be an ideal of~$\Q_p\lfloor X\rfloor$; then $\operatorname{I}\!\big(\!\operatorname{Z}(I)\big)=\sqrt[\leftroot{2}\uproot{5}p]{I}$.
\end{cor}

\noindent
The previous two corollaries apply to $\Q_p\lfloor X\rfloor=\QpX$ and $\Q_p\lfloor X\rfloor=\QpX^{\operatorname{a}}$.
For each $f\in\Q_p\{X\}$   there is some $k\geq 1$ such that $f(p^kX)\in\QpX$. Hence from the case 
$\Q_p\lfloor X\rfloor=\QpX$ of Corollaries~\ref{cor:p-adic H17 rest} and~\ref{cor:p-adic NS rest} we also deduce once again Theorems~A and~B from the introduction (but not Theorem~C, about $\Q_p[[X]]$, the proof of which seems to crucially require the use of $p\operatorname{CVF}$).
We finish with an immediate application of Lemma~\ref{lem:p-adic prime ideals} and Corollary~\ref{cor:p-adic NS rest}: a special case of the $p$-adic analytic
\L{}ojasiewicz Inequality of Denef and van den Dries~\cite[Theorem~3.37]{DD}:

\begin{cor}
Let $f,g\in\QpX$ be such that $\operatorname{Z}(f)\supseteq\operatorname{Z}(g)$. Then there is some~$k\geq 1$
such that $\abs{f(a)}_p^k \leq \abs{g(a)}_p$ for all $a\in\Z_p^m$.
\end{cor}

\bibliographystyle{amsplain}

\end{document}